\documentclass[11pt,oneside]{amsart}   	
\usepackage[margin = 1in]{geometry}
									
\usepackage{amsmath,amssymb, amsfonts,amsthm}
\usepackage{tensor}
\usepackage[dvipsnames]{xcolor}
\usepackage{hyperref}
\usepackage{mathtools}
\usepackage{tikz}
\usepackage{tikz-cd}
\usetikzlibrary{decorations.markings}
\usepackage{caption}
\usepackage{enumitem}
\usepackage{array}
\usepackage{makecell}
\setcellgapes{4pt}
\usepackage{longtable}
\newcolumntype{L}[1]{>{\raggedright\let\newline\\\arraybackslash\hspace{0pt}}m{#1}}
\newcolumntype{C}[1]{>{\centering\let\newline\\\arraybackslash\hspace{0pt}}m{#1}}
\newcolumntype{R}[1]{>{\raggedleft\let\newline\\\arraybackslash\hspace{0pt}}m{#1}}

\newtheorem{theorem}{Theorem}[subsection]
\newtheorem{definition-theorem}[theorem]{Definition-Theorem}
\newtheorem{lemma}[theorem]{Lemma}
\newtheorem{corollary}[theorem]{Corollary}
\newtheorem{proposition}[theorem]{Proposition}

\theoremstyle{definition}
\newtheorem{definition}[theorem]{Definition}
\newtheorem{example}[theorem]{Example}

\newtheorem{remark}[theorem]{Remark}

\newcommand{\lperp}[1]{\prescript{\perp}{}{#1}}

\newcommand{\rperp}[1]{#1^\perp}

\newcommand{\add}{\mathsf{add}}
\newcommand{\Filt}{\mathsf{Filt}}
\DeclareMathOperator{\Fac}{\mathsf{Fac}}
\newcommand{\wide}{\mathsf{wide}}
\newcommand{\brick}{\mathsf{brick}}
\newcommand{\sbrick}{\mathsf{sbrick}}

\newcommand{\tors}{\mathsf{tors}}

\newcommand{\tex}{\tau\text{-}\mathsf{ex}}
\newcommand{\arc}{\mathrm{arc}}
\newcommand{\btex}{\mathsf{br}\text{-}\tau\text{-}\mathsf{ex}}

\newcommand{\stopc}{\mathsf{stop\text{-}chain}}
\newcommand{\satc}{\mathsf{s\text{-}chain}}
\newcommand{\mods}{\mathsf{mod}}
\newcommand{\Hom}{\mathrm{Hom}}
\newcommand{\Ext}{\mathrm{Ext}}
\newcommand{\End}{\mathrm{End}}

\newcommand{\rk}{\mathrm{rk}}
\newcommand{\trig}{\tau\text{-}\mathsf{rigid}}
\newcommand{\itrig}{\mathsf{ind}\text{-}\tau\text{-}\mathsf{rigid}}
\newcommand{\trigW}[1]{\tau_{#1}\text{-}\mathsf{rigid}}
\newcommand{\itrigW}[1]{\mathsf{ind}\text{-}\tau_{#1}\text{-}\mathsf{rigid}}

\newcommand{\shard}{\mathrm{shard}}
\newcommand{\shardint}{\mathrm{shard}\text{-}\mathrm{int}}
\newcommand{\ppo}{\mathrm{ppo}}
\newcommand{\ppoadm}{\mathrm{ppo}\text{-}\mathrm{adm}}
\newcommand{\word}{\mathrm{word}}
\newcommand{\wrd}{\mathrm{wd}}

\newcommand{\shuff}{\mathrm{shuff}}
\newcommand{\brlab}{\mathrm{br}\textnormal{-}\mathrm{lab}}
\newcommand{\cov}{\mathrm{cov}}
\newcommand{\arclab}{\mathrm{arc}\textnormal{-}\mathrm{lab}}

\newcommand{\T}{\mathcal{T}}

\newcommand{\W}{\mathcal{W}}
\newcommand{\U}{\mathcal{U}}

\renewcommand{\P}{\mathcal{P}}
\newcommand{\Q}{\mathcal{Q}}
\newcommand{\J}{\mathcal J}
\newcommand{\R}{\mathbb{R}}

\newcommand{\N}{\mathbb{N}}
\newcommand{\X}{\mathcal{X}}

\usepackage{soul}

\author{Eric J. Hanson}

\address{D\'epartement de Math\'ematiques, LACIM, Universit\'e du Qu\'ebec \`a Montr\'eal and\newline\indent D\'epartement de Math\'ematiques, Universit\'e de Sherbrooke}
\email{eric.james.hanson@usherbrooke.ca}

\subjclass[2020]{16G20, 05E10 (primary), 06A07, 20F55, 52C35 (secondary)}

\title{$\tau$-exceptional sequences and the shard intersection order in type A}

\date{\today}

\begin{document}
\maketitle

\begin{abstract}
	Reading's ``shard intersection order'' on the symmetric group can be realized as the ``lattice of wide subcategories'' of the corresponding preprojective algebra. In this paper, we first use Bancroft's combinatorial model for the shard intersection order to associate a unique shard to each downward cover relation. We then show that, under the correspondence with wide subcategories, this process coincides with Jasso's ``$\tau$-tilting reduction''. In particular, this yields a combinatorial model for this algebras's ``$\tau$-exceptional sequences'' (defined by Buan and Marsh). We formulate this model using the combinatorics of arc diagrams. Finally, we use our model to give a new representation-theoretic proof that the shard intersection order is EL-shellable in type A.
\end{abstract}

\setcounter{tocdepth}{1}
\tableofcontents

\section{Introduction}

This paper comprises a story in two parts: Coxeter combinatorics and the representation theory of finite-dimensional algebras.

We begin our story with Reading's \emph{shard intersection order}. In \cite{readingSIO}, Reading defines a rule for cutting the hyperplanes of a simplicial hyperplane arrangement into pieces called \emph{shards}. He shows that each shard corresponds to a join-irreducible element of the poset of regions, and defines a lattice structure (the shard intersection order) on the set of intersections of shards. The shard intersection order of a Coxeter arrangement in particular has the corresponding noncrossing partition lattice as a sublattice.

Following its introduction, combinatorial models for the shard intersection order in types A, B, and D were established in \cite{bancroft,petersen}. In all three types, these models were used to show that the shard intersection order is EL-shellable, and one of our results is a new proof of this fact in type A. The objects comprising the model in type A are called ``permutation preorders'', and feature heavily throughout this paper. Furthermore, in \cite{reading}, Reading uses the related constructions of ``arcs'' and ``noncrossing arc diagrams'' 
to understand canonical join representations in the weak order. These objects again play a central role in the present paper. In short, we use arcs to define a ``reduction'' process for permutation preorders. Doing so allows us to associate an arc to each cover relation in the shard intersection order, and we classify the sequences of such arcs corresponding to maximal chains as (complete) ``ppo-admissible sequences'' (see Definition~\ref{def:ppo_admissible_sequence}).

While the definition of the shard intersection order utilizes the geometry of a hyperplane arrangement, it is shown in \cite[Section~7-7.4]{reading_lattice} that the order can also be defined from the ``join-irreducible labeling'' of any finite congruence uniform (or more generally semidistributive) lattice. (There are also extensions to infinite lattices, but we restrict to the finite case to simplify the exposition.) The resulting generalization is sometimes known as the ``core label order'', see e.g. \cite{muhle}.

Another class of semidistrbutive lattices which has received considerable attention in the last several years is that of ``lattices of torsion classes'' of finite-dimensional algebra, see e.g. \cite{BCZ,DIRRT,GM,IRRT,IRTT}. In the recent work \cite{enomoto_lattice}, it is shown explicitly that the core-label order of a lattice of torsion classes is isomorphic to the corresponding ``lattice of wide subcategories". A special case of this fact was also implicit in the work \cite{thomas}, in which Thomas explained how shards (of Coxeter arrangements) can be understood in terms of the representation theory of the corresponding ``preprojective algebra''. (Mizuno had previously shown in \cite{mizuno} that the lattice of torsion classes of the preprojective algebra is isomorphic to the weak order of the corresponding Coxeter group.) Thomas's work was extended in the recent paper \cite{mizuno_shards}, which established a notion of shards for any finite-dimensional algebra with a finite lattice of torsion classes. Such algebras are called ``$\tau$-tilting finite" \cite{DIJ}, a property which is defined using the \emph{$\tau$-tilting theory} of Adachi, Iyama, and Reiten \cite{AIR}. This brings us to the second part of our story.

Since its introduction a decade ago, $\tau$-tilting theory
has proven invaluable in generalizing results about representations of quivers to results about arbitrary finite-dimensional algebras. For example, one of the fundamental results in $\tau$-tilting theory is that, over any finite-dimensional algebra, every ``$\tau$-rigid module'' is a direct summand of a ``$\tau$-tilting module'' \cite[Theorem~2.10]{AIR}. Over hereditary algebras, this recovers the classical fact that every ``rigid module'' is a direct summand of a ``tilting module''.

In this paper, we focus on Buan and Marsh's generalization of \emph{exceptional sequences} of quiver representations \cite{CB_exceptional,ringel_exceptional} to \emph{$\tau$-exceptional sequences} over arbitrary finite-dimensional algebras \cite{BM_exceptional}. These sequences are connected to many other areas of mathematics. For example, classical exceptional sequences are known to parameterize certain saturated chains in the corresponding lattice of noncrossing partitions \cite{IS}, while $\tau$-exceptional sequences parameterize certain saturated chains in the corresponding lattice of wide subcategories \cite{BuH,BM_wide}. Similarly, classical ``signed'' exceptional sequences parameterize the morphisms of the ``cluster morphism category'', which in finite type serves as an Eilenberg MacLane space for the corresponding ``picture group'' \cite{IT_signed,ITW}, while ``signed'' $\tau$-exceptional sequences parameterize the morphisms of the ``$\tau$-cluster morphism category" \cite{BuH,BM_wide,borve}, which in special cases serve as Eilenberg MacLane spaces the associated (generalized) ``picture groups'' \cite{HI,HI2}.
$\tau$-exceptional sequences have also been linked to the existence of certain types of ``stratifying systems'' \cite{MT}.

The above results establish many motivations and theoretical foundations for the study of $\tau$-exceptional sequences, but there are very few concrete examples in the literature beyond the hereditary case. (See e.g. \cite{araya,CG,GIMO,IM,IguSen,maresca} for many examples over hereditary algebras.) The aim of this paper is to provide such an example. More precisely, we establish a bijection between the
the $\tau$-exceptional sequences over \emph{preprojective algebras of type $A_n$} (as well as certain quotients of these algebras) and the ppo-admissible sequences of arcs defined in Definition~\ref{def:ppo_admissible_sequence}. This continues a line of successful research aimed at modeling various representation-theoretic constructions over these algebras, such as their $\tau$-tilting modules \cite{IRRT,mizuno}, their bricks and semibricks \cite{asai,BCZ,enomoto}, and their 2-term simple minded collections \cite{BaH,mizuno2}.

\subsection{Organization and main results}\label{sec:organization}

The contents of this paper are as follows. Note also that we provide an index of notation at the end of the paper for the convenience of the reader.

In Section~\ref{sec:combinatorics}, we discuss the main combinatorial objects we use throughout the paper. We first consider \emph{arcs} (Definition~\ref{def:wd}) and two special collections thereof: \emph{clockwise-ordered arc diagrams} (Definition~\ref{def:clockwise}) and \emph{noncrossing arc diagrams} (Definition~\ref{def:noncrossing}). We then recall the notion of a \emph{permutation preorder} (Definition~\ref{def:perm_preorder}) from \cite{bancroft}, and explain the partial order which allows them to realize the ``shard intersection order'' of the symmetric group $\mathfrak{S}_{n+1}$ (Definition~\ref{def:perm_preorder2}). We also recall results from \cite{bancroft,reading} regarding how permutations, noncrossing arc diagrams, and permutation preorders are related (Proposition~\ref{prop:bijection_ppo}). In Section~\ref{sec:shards}, we provide definitions of \emph{shards}, \emph{shard intersections}, and the \emph{shard intersection order} which are suitable for use in this paper (Definition~\ref{def:shards} and Definition-Theorem~\ref{defthm:sio}). We then recall known results about cover relations in the shard intersection order (Proposition~\ref{prop:covers_ppo}). Finally, In Section~\ref{sec:recursive}, we study a recursive property of permutation preorders. More precisely, for a fixed permutation preorder $\P$, we show that the permutation preorders below $\P$ in the shard intersection order can be described by choosing a new permutation preorder on every partition element of $\P$ (Proposition~\ref{prop:perm_preorder_recursive}).

In Section~\ref{sec:admissibility}, we first define what it means for an arc to be \emph{ppo-admissible} with respect to a fixed permutation preorder $\P$ (Definition~\ref{def:ppo_admissible}). The definition is designed to yield the following.
 
 \begin{proposition}[Proposition~\ref{prop:ppo_admissible_shard}]\label{prop:mainA}
	Let $\P$ be a permutation preorder. Then the arcs which are $\P$-ppo-admissible are precisely those whose shards contain the shard intersection associated to $\P$.
\end{proposition}

\noindent We then introduce a process for moving down a cover relation in the shard intersection order called \emph{ppo-reduction} (Definition~\ref{def:ppo_reduction}). We discuss how this relates to the recursive property of permutation preorders, then prove our first main result.

\begin{theorem}[Theorem~\ref{thm:edge-labeling}]\label{thm:mainB}
	Let $\P < \Q$ be a cover relation in the shard intersection order (realized as an order on the set of permutation preorders). Then there exists a unique arc $\gamma$ such that $\P$ is the ppo-reduction of $\Q$ at $\gamma$. Moreover, the shard intersection corresponding to $\Q$ is the intersection of that corresponding to $\P$ with the shard corresponding to $\gamma$.
\end{theorem}

In Section~\ref{ppo_admissible3}, we introduce \emph{ppo-admissible sequences} (Definition~\ref{def:ppo_admissible_sequence}), which are sequences of arcs corresponding to successive applications of ppo-reduction. We then show that every ppo-admissible sequence is also a clockwise-ordered arc diagram (Corollary~\ref{cor:ppo_adm_clockwise_ordered}).

Section~\ref{sec:background} is where the focus of the paper shifts towards the representation theory of finite-dimensional algebras. We give the definition of \emph{$\tau$-exceptional sequences} (Definition~\ref{def:tau_exceptional}) and explain their relationship with wide subcategories (Theorem~\ref{thm:sat_top}) and bricks (Section~\ref{sec:bricks}). In Section~\ref{sec:preproj}, we study the bricks over the preprojective algebra $\Pi(A_n)$ and its ``gentle quotient'' $RA_n$ (see Definition~\ref{def:preproj}). We first recall the bijection between bricks over these algebras and arcs on $n+1$ nodes from \cite{BCZ,mizuno2} (Proposition~\ref{prop:arcBricks}). We then explain how this bijection can be used to understand when certain Hom- and Ext-spaces between bricks vanish (Proposition~\ref{prop:hom_ext}).

In Section~\ref{sec:mainResults}, we explicitly relate the notions of ppo-reduction and ppo-admissible sequences from Section~\ref{sec:admissibility} to the $\tau$-tilting reduction and (brick-)$\tau$-exceptional sequences of the algebras $RA_n$ and $\Pi(A_n)$. First, in Section~\ref{sec:wide}, we explain how wide subcategories over the algebras $RA_n$ and $\Pi(A_n)$ are related to noncrossing arc diagrams and permutation preorders by combining new and old bijections into a commutative diagram (Theorem~\ref{thm:wide}). Then, in Sections~\ref{sec:tau_ex_RAn} and~\ref{sec:tau_exceptional_preproj}, we prove the second main result of this paper.

\begin{theorem}[Theorem~\ref{thm:tau_exceptional_RAn} and Corollary~\ref{cor:brick_tau}]\label{thm:mainC}
    The bijection between the set of bricks over $RA_n$ (resp. over $\Pi(A_n)$) and arcs on $n+1$ nodes induces a bijection between the set of $\tau$-exceptional sequences over $RA_n$ (resp. brick-$\tau$-exceptional sequences over $\Pi(A_n)$) and the set of ppo-admissible sequences of arcs.
\end{theorem}

\noindent 
Note that it is implicit in Theorem~\ref{thm:mainC} is that the brick-$\tau$-exceptional sequences of $RA_n$ and over $\Pi(A_n)$ coincide (Corollary~\ref{cor:brick_tau}). The majority of the work towards proving Theorem~\ref{thm:mainC} is done in Lemma~\ref{lem:tau_perp}, where we show that $\tau$-perpendicular categories over $RA_n$ and ppo-reduction correspond to one another under this bijection.
We also conclude Section~\ref{sec:tau_exceptional_preproj} by explaining how one can use Asai's combinatorial realization of the ``brick-$\tau$-rigid correspondence'' for $\Pi(A_n)$ \cite{asai} to give a model for the $\tau$-exceptional sequences over $\Pi(A_n)$ (Corollary~\ref{cor:tau_An}).

Finally, in Section~\ref{sec:EL}, we use Theorem~\ref{thm:mainC} to give a new proof of the fact (shown in \cite{bancroft}) that the shard intersection order on the symmetric group, or equivalently the lattice of wide subcategories of either $RA_n$ or $\Pi(A_n)$, is EL-shellable (Theorem~\ref{thm:EL_labeling}).

\subsection{Acknowledgements}

The author is thankful to Emily Barnard, Karin Baur, Aslak Bakke Buan, Colin Defant, Benjamin Dequ{\^e}ne, Hern\'an Ibarra-Mejia, Ray Maresca, Hugh Thomas, Nathan Williams, and Xinrui You for many insightful discussions related to this project. The author was supported by the Canada Research Chairs program (CRC-2021-00120) and NSERC Discovery Grants (RGPIN-2022-03960 and RGPIN/04465-2019).


\section{Arcs, permutation preorders, and shard intersections}\label{sec:combinatorics}

In this section, we recall two types of combinatorial data which can be associated to permutations. We first consider \emph{arcs} (and more generally \emph{noncrossing arcs diagrams}), which were introduced in \cite{reading} to understand canonical join representations in the weak order. We also recall the definition of \emph{clockwise-ordered arc diagrams} from \cite{HY}. We then recall the definition of a \emph{permutation preorder}. These were introduced by Bancroft in \cite{bancroft} as a combinatorial model for Reading's \emph{shard intersection order} on the symmetric group $\mathfrak{S}_{n+1}$ \cite{readingSIO}. (See also \cite{petersen} where this formulation is reproved and extended to Coxeter groups of type B and D.) We explain this connection in Section~\ref{sec:shards}. Finally, in Section~\ref{sec:recursive}, we consider a recursive property of permutation preorders.

\subsection{Arcs and preorders associated to permutations}\label{sec:perm_preorder}

Fix $n \in \mathbb{N}$ and denote $[n] = \{0,1,\ldots,n\}$.

\begin{definition}\label{def:wd}\
    \begin{enumerate}
    	\item Let $\mathrm{word}(\text{ueo})$ be the set of nonempty words on the letters u, e, and o. We treat $\{\text{u}, \text{e}, \text{o}\}$ as a totally ordered set with $\text{u} < \text{e} < \text{o}$. An \emph{arc} on $n+1$ nodes is an element $\gamma = \left(\ell(\gamma),s_{\ell(\gamma)+1}s_{\ell(\gamma)+2}\cdots s_{r(\gamma)}\right) \in [n-1] \times \mathrm{word}(\text{ueo})$ which satisfies the following:
    \begin{enumerate}
    	\item $s_{r(\gamma)} = \text{e}$ and $s_{i} \neq \text{e}$ for $\ell(\gamma) < i < r(\gamma)$.
	\item $r(\gamma) \leq n$.
    \end{enumerate}
    We denote by $\arc(n)$ the set of arcs on $n+1$ nodes.
    \item Let $\gamma = \left(\ell(\gamma),s_{\ell(\gamma)+1}s_{\ell(\gamma)+2}\cdots s_{r(\gamma)}\right)\in \arc(n)$. We say that $\ell(\gamma)$ and $r(\gamma)$ are the \emph{(left and right) endpoints} of $\gamma$. For $\ell(\gamma) < i < r(\gamma)$, we say that $\gamma$ passes \emph{under} $i$ if $s_i = \text{u}$ and we say that $\gamma$ passes \emph{over} $i$ if $s_i = \text{o}$.
        \item Consider $[n]$ as a set of $n+1$ nodes in $\mathbb{R}^2$, arranged in increasing order at the points $(i,0)$. A \emph{realization} of an arc $\gamma = \left(\ell(\gamma),s_{\ell(\gamma)+1}s_{\ell(\gamma)+2}\cdots s_{r(\gamma)}\right) \in \arc(n)$ is a continuous function $[\ell(\gamma),r(\gamma)] \rightarrow \mathbb{R}^2$ (which we also denote by $\gamma$) which satisfies the following:
    \begin{enumerate}	
    	\item $\gamma(\ell(\gamma)) = 0 = \gamma(r(\gamma))$.
	\item For all $i \in (\ell(\gamma),r(\gamma))\cap \mathbb{Z}$: If $s_i = \text{u}$ then $\gamma(i) < 0$ and if $s_i = \text{o}$ then $\gamma(i) > 0$.
    \end{enumerate}
     \end{enumerate}
\end{definition}

For example, the arcs in Figure~\ref{fig:arcs} are given by $\gamma_1 = (2,\text{ue})$, $\gamma_2 = (4,\text{e})$, and $\gamma_3 = (0,\text{uooue})$.

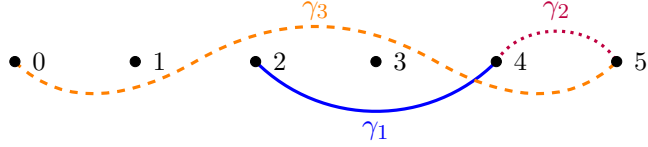
\begin{figure}
\begin{tikzpicture}[scale=0.8]

\draw[blue,very thick,smooth] (2,0) [out = -45,in = -135] to (6,0);

\draw[orange,dashed,very thick,smooth] (-2,0) [out = -45,in = -150] to (1,0) [out = 30,in = 150] to (5,0) [out = -30,in = -135] to (8,0);

\draw[purple,very thick,dotted] (6,0) [out = 60,in = 120] to (8,0);

\node at (4,-1.15) {{\color{blue}$\gamma_1$}};
\node at (3,0.85) {{\color{orange}$\gamma_3$}};
\node at (7,0.85) {{\color{purple}$\gamma_2$}};

\node at (-1.6,0) {{\small 0}};
\node at (0.4,0) {{\small 1}};
\node at (2.4,0) {{\small 2}};
\node at (4.4,0) {{\small 3}};
\node at (6.4,0) {{\small 4}};
\node at (8.4,0) {{\small 5}};

\filldraw  (0,0) circle (2.4pt)
 (-2,0) circle (2.4pt)
 (2,0) circle (2.4pt)
 (4,0) circle (2.4pt) 
 (6,0) circle (2.4pt)
 (8,0) circle (2.4pt)
 ;
\end{tikzpicture}
\caption{Realizations of arcs on 6 nodes.}\label{fig:arcs}
\end{figure}

\begin{remark}\label{def:wd}
	In several places in the literature (e.g. \cite{barnard,BCZ,BaH,BR,enomoto,mizuno2,reading}), one first defines an arc as a path in $\mathbb{R}^2$, then establishes an equivalence relation based on whether this arc passes above or below certain nodes. We have opted to instead define an arc in terms of its combinatorial data, as has been done in e.g. \cite{PPR}. Our choice of how to encode this data is made to simplify the definition of the total order used to prove the existence of an EL-labeling in Section~\ref{sec:EL}.
\end{remark}

\begin{definition}\label{def:crossing}
	Let $\gamma, \rho \in \arc(n)$.
	\begin{enumerate}
		\item We say that $\gamma$ and $\rho$ have a \emph{shared (left or right) endpoint} if $\ell(\gamma) = \ell(\rho)$ or $r(\gamma) = r(\rho)$.
		\item We say that $\gamma$ and $\rho$ have a \emph{contested endpoint} if $\ell(\gamma) = r(\rho)$ or $\ell(\rho) = r(\gamma)$.
		\item We say that $\gamma$ and $\rho$ have a \emph{nontrivial crossing} if for any two realizations of $\gamma$ and $\rho$ there exists $x \in (\ell(\gamma),r(\gamma)) \cap (\ell(\rho),r(\rho))$ such that $\gamma(x) = \rho(x)$.
	\end{enumerate}
\end{definition}

\begin{remark}\label{rem:crossing}
	Suppose that $\gamma = \left(\ell(\gamma),s_{\ell(\gamma)+1}\cdots s_{r(\gamma)}\right), \rho = \left(\ell(\rho),t_{\ell(\rho)+1} \cdots t_{\rho}\right) \in \arc(n)$ cross nontrivially. Then there must exist integers $i < j$ in $[\ell(\gamma),r(\gamma)] \cap [\ell(\rho),r(\rho)] \cap \mathbb{Z}$ such that $\gamma$ must be ``above'' $\rho$ and $i$ and ``below'' $\rho$ at $j$, or vice versa. More precisely, $i$ and $j$ satisfy the following.
	\begin{enumerate}
		\item $i$ is not a shared left endpoint.
		\item $j$ is not a shared right endpoint.
		\item One of the following holds:
		\begin{enumerate}
			\item (i) $s_j \leq \text{e} \leq t_j$, (ii) $\ell(\gamma) = i$ or $s_i = \text{o}$, and (iii) $\ell(\rho) = i$ or $t_i = \text{u}$.
			\item (i) $t_j \leq \text{e} \leq s_j$, (ii) $\ell(\rho) = i$ or $t_i = \text{o}$, and (iii) $\ell(\gamma) = i$ or $s_i = \text{u}$.
		\end{enumerate}
	\end{enumerate}
	See also \cite[Lemma~4.8]{HY} for a proof of a similar statement which implies this one.
	\end{remark}
	
For example, consider the arcs in Figure~\ref{fig:arcs}. Then $\gamma_1$ and $\gamma_2$ have a contested endpoint, while $\gamma_2$ and $\gamma_3$ have a shared right endpoint. There is a nontrivial crossing between $\gamma_1$ and $\gamma_3$, which can be seen from taking $\gamma = \gamma_3$, $\rho = \gamma_1$, $i \in \{2,3\}$, and $j = 4$ in the setup of Remark~\ref{rem:crossing}.

For use in relating crossings between arcs to morphisms between bricks over some algebra (Section~\ref{sec:preproj}), we also wish to take note of the fact that $\gamma_3$ is drawn above $\gamma_1$ immediately to the left of their intersection point. We do so with the following definition.

\begin{definition}\label{def:hom_direction}
	In the setup of Remark~\ref{rem:crossing}, suppose there exist $i$ and $j$ which satisfy (1), (2), and (3a). Then we say there is a nontrivial crossing \emph{directed} from $\gamma$ to $\rho$.
\end{definition}

So, in Figure~\ref{fig:arcs}, there is a nontrivial crossing directed from $\gamma_3$ to $\gamma_1$, but not vice versa.

We similarly want to record the fact that $\gamma_2$ is drawn above $\gamma_3$ at their shared right endpoint. Again for reasons seen later, the convention is reversed in the case of a shared left endpoint; that is, we use the following definition.

\begin{definition}\label{def:clockwise}\
    \begin{enumerate}
    	\item Let $\gamma = \left(\ell(\gamma),s_{\ell(\gamma)+1} \cdots s_{r(\gamma)}\right), \rho = \left(\ell(\rho),t_{\ell(\rho)+1}\cdots t_{r(\rho)}\right) \in \arc(n)$.
    \begin{enumerate}
        \item Suppose that $\ell(\gamma) = \ell(\rho)$. We say that $\rho$ is \emph{clockwise} of $\gamma$ if (i) $\gamma$ and $\rho$ do not have a nontrivial crossing, and (ii) either $r(\gamma) < r(\rho)$ and $t_{r(\gamma)} = \text{o}$ or $r(\gamma) > r(\rho)$ and $s_{r(\rho)} = \text{u}$.
        \item Suppose that $r(\gamma) = r(\rho)$. We say that $\rho$ is \emph{clockwise} of $\gamma$ if (i) $\gamma$ and $\rho$ do not have a nontrivial crossing, and (ii) either $\ell(\gamma) < \ell(\rho)$ and $s_{\ell(\rho)} = \text{u}$ or $\ell(\gamma) > \ell(\rho)$ and $t_{\ell(\gamma)} = \text{o}$.
    \end{enumerate}
    	\item  Let $\omega 
 = (\gamma_k,\ldots,\gamma_1)$ be an ordered set of arcs in $\arc(n)$. We say that $\omega$ is a \emph{clockwise-ordered arc diagram} (of length $k$) if for all $1 \leq i < j \leq k$, one of:
    \begin{enumerate}
        \item $\ell(\gamma_i) = \ell(\gamma_j)$ and $\gamma_j$ is clockwise of $\gamma_i$,
        \item $r(\gamma_i) = r(\gamma_j)$ and $\gamma_j$ is clockwise of $\gamma_j$, or
        \item $\gamma_i$ and $\gamma_j$ do not have a shared endpoint, a contested endpoint, or a nontrivial crossing.
    \end{enumerate}
    We denote by $\arc_{cw}(n,k)$ the set of clockwise-ordered arc diagrams on $n+1$ of length $k$ and $\arc_{cw}(n) := \bigcup_{k \in \N}\arc_{cw}(n,k)$.
    \end{enumerate}
\end{definition}

So, in Figure~\ref{fig:arcs} the arc $\gamma_2$ is clockwise of the arc $\gamma_3$. An example of a clockwise-ordered arc diagram is shown in Figure~\ref{fig:clockwise}.

\begin{figure}
\begin{tikzpicture}[scale=0.8]

\draw[blue,very thick,smooth]
    (-2,0) [out = 45,in = 180] to (1,1.5) [out = 0,in = 135] to (4,0);

\draw[purple,very thick, smooth,dotted]
    (-2,0) [out = 30,in = 180] to (0,1) [out = 0,in = 150] to (2,0);

\draw[orange,very thick,smooth,dashed]
    (0,0) to (2,0);

\draw[purple,very thick,smooth,dotted]
    (0,0) [out = -45,in = 180] to (2,-1) [out = 0,in = -135] to (4,0);

\node [blue] at (2.5,0.75) {$\gamma_3$};
\node [purple] at (0,0.75) {$\gamma_4$};
\node [orange] at (1,-0.25) {$\gamma_1$};
\node [purple] at (3,-0.5) {$\gamma_2$};

\filldraw  (0,0) circle (2.4pt)
 (-2,0) circle (2.4pt)
 (2,0) circle (2.4pt)
 (4,0) circle (2.4pt) 
 ;
\end{tikzpicture}
\caption{The sequence $(\gamma_4,\gamma_3,\gamma_2,\gamma_1)$ is a clockwise-ordered arc diagram on 4 nodes.}\label{fig:clockwise}
\end{figure}
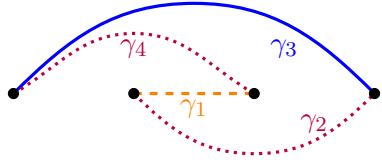

It is shown in \cite[Proposition~6.6]{HY} that the arcs comprising a clockwise-ordered arc diagram together form the edge-set of a planar bipartite graph. A consequence is that if $\omega \in \arc_{cw}(n)$, then there exist realizations of the arcs in $\omega$ which intersect only at shared endpoints. (It is immediate from the definition that this is possible for any pair of arcs in $\omega$, but requires proof that there exist fixed realizations that work for all pairs.) Another consequence is that the length of a clockwise-ordered arc diagram is (tightly) bounded above by $\max\{1,2n-2\}$.

We now recall the definition of a noncrossing arc diagram from \cite{reading}.

\begin{definition}\label{def:noncrossing}
    Let $\mathcal{A} \subseteq \arc(n)$ be a set of arcs. We say that $\mathcal{A}$ is a \emph{noncrossing arc diagram} if no pair or arcs $\gamma \neq \rho \in \arc(\mathcal{A})$ have either have a shared endpoint or a nontrivial crossing.
\end{definition}

For example, the arcs in Figure~\ref{fig:bijection_ppo} form a noncrossing arc diagram.

Similarly to the situation with clockwise-ordered arc diagrams, it is shown in \cite{reading} that if $\mathcal{A}$ is a noncrossing arc diagram, then there exist fixed realizations of the arcs in $\mathcal{A}$ which do not intersect (except possibly at contested endpoints). This justifies using the name ``noncrossing arc diagram'' rather than ``pairwise noncrossing arc diagram''.

\begin{figure}
\begin{tikzpicture}[scale=0.5]

\draw [very thick,smooth] 
(-2,0) to (0,0);

\node at (-1,0.5){$\gamma_{w,0}$};

\draw [very thick,smooth] 
(2,0) [out = 45, in = 180]  to (4,1) [out = 0,in = 135] to (6,0);

\draw [very thick,smooth] 
(6,0) [out = -60, in = 180]  to (9.5,-1) [out = 0,in = 180] to (12.5,1) [out = 0, in = 135] to (14,0);

\draw [very thick,smooth] 
(4,0) [out = -60, in = 180]  to (8,-1.5) [out = 0,in = -120] to (12,0);

\node at (4,1.5){$\gamma_{w,4}$};

\node at (8,-2){$\gamma_{w,6}$};

\node at (12,1.4){$\gamma_{w,3}$};
        
\filldraw  (0,0) circle (2.4pt)
 (-2,0) circle (2.4pt)
 (2,0) circle (2.4pt)
 (4,0) circle (2.4pt) 
 (6,0) circle (2.4pt)
 (8,0) circle (2.4pt)
 (10,0) circle (2.4pt)
 (12,0) circle (2.4pt)
 (14,0) circle (2.4pt)
 ;

\node at (-1.8,-0.5) {{\small 0}};
\node at (0.4,0) {{\small 1}};
\node at (2.4,0) {{\small 2}};
\node at (4.4,0) {{\small 3}};
\node at (6.4,0) {{\small 4}};
\node at (8.4,0) {{\small 5}};
\node at (10.4,0) {{\small 6}};
\node at (12.4,0) {{\small 7}};
\node at (14.4,0) {{\small 8}};

\begin{scope}[shift = {(20,0)}]
    \draw (3,0) -- (2,-2);
    \draw (3,0) -- (4,-2);
    \draw (3,0) -- (3,2);

    \node [draw,fill=white] at (-1,0) {01};
    \node [draw,fill=white] at (3,0) {248};
    \node [draw,fill=white] at (2,-2) {5};
    \node [draw,fill=white] at (4,-2) {6};
    \node [draw,fill=white] at (3,2) {37};
\end{scope}

\end{tikzpicture}
\caption{The noncrossing arc diagram (left) and permutation preorder (right) corresponding to the permutation $105684273 \in \mathfrak{S}_9$ from Example~\ref{ex:bijection_ppo}.}\label{fig:bijection_ppo}
\end{figure}
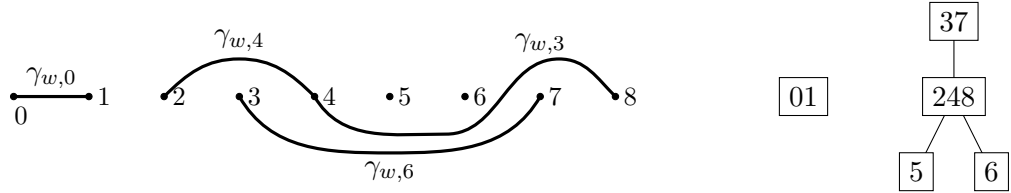

We also need the following.

\begin{definition}\label{def:cover}
    Let $\X = (\X,\leq_\X)$ be a poset.
    \begin{enumerate}
        \item We say a relation $x \leq_\X y$ in $\X$ is a \emph{cover relation} if $x \neq y$ and $x \leq_\X z \leq_\X y$ implies $z\in \{x,y\}$. We denote by $\cov(\X)$ the set of cover relations in $\X$.
        \item Let $\T = (\T,\leq_\T)$ be a totally ordered set. We call a function $f: \cov(\X) \rightarrow \T$ a \emph{edge-labeling} of $\X$ by $\T$.
     \end{enumerate}
\end{definition}

We now recall the definition of a permutation preorder from \cite{bancroft} (see also \cite{petersen}).

\begin{definition}\label{def:perm_preorder}\
    \begin{enumerate}
        \item Let $S \subseteq [n]$. Then the \emph{closed integer support} of $S$, is the set $$\overline{S} := \{\min(S),\min(S)+1,\ldots,\max(S)\}.$$ 
        Likewise, the \emph{closed support} of $S$ is the interval $[\min(S),\max(S)] \subseteq \R$. 
        \item  A \emph{permutation preorder} on $[n]$ is a partition $\P$ of $[n]$ and a partial order $\preceq_{\P}$ on the elements of $\P$ such that the following hold.
    \begin{enumerate}
        \item Let $S,S' \in \P$. If $\overline{S} \cap \overline{S'} \neq \emptyset$, then $S$ and $S'$ are related under $\preceq_{\P}$. In other words, if two partition elements are ``overlapping'' then they must be related
        \item Let $S \preceq_{\P} S'$ be a cover relation of $\preceq_{\P}$. Then $\overline{S} \cap \overline{S'} \neq \emptyset$. In other words, if two partition elements are related by a cover relation then they must be ``overlapping''.
    \end{enumerate}
    We denote by $\ppo(n)$ the set of permutation preorders of $[n]$.
    \end{enumerate}
\end{definition}

\begin{remark}\label{rem:perm_perorder}
    For simplicity, we will use $\P$ to represent the permutation preorder $(\P,\preceq_{\P})$. We note that this, and also using $\preceq_{\P}$ to denote the partial order on $\P$, is a slight abuse of notation. For example, $\P = \{\{0,2\},\{1\}\}$ with $\{0,2\} \preceq_{\P} \{1\}$ and $\Q = \{\{0,2\},\{1\}\}$ with $\{1\} \preceq_{\Q} \{0,2\}$ are both permutation preorders with the same underlying partition of $[n]$. 
\end{remark}

See Example~\ref{ex:bijection_ppo} below for an additional example.

We next recall the relationship between noncrossing arc diagrams and permutation preorders. We consider the elements of $\mathfrak{S}_{n+1}$ as automorphisms of the set $[n]$, with multiplication following function composition. We write a permutation $w \in \mathfrak{S}_{n+1}$ in its one-line notation $w = w_0w_1\cdots w_n$, where $w_i = w(i)$. For example, we write 021 for the permutation $0 \mapsto 0$, $1 \mapsto 2$, $2 \mapsto 1$. A \emph{descending run} of a permutation $w$ is a (consecutive) subword $w_iw_{i+1}\cdots w_j$ of $w$ such that (i) either $i = 0$ or $w_{i-1} < w_i$, (ii) either $j = n$ or $w_{j+1} > w_j$, and (iii) $w_k > w_{k+1}$ for all $i \leq k < j$. Each permutation $w$ can be partitioned into its descending runs. For example 021 = $0|21$ consists of two descending runs, with the descending run 0 coming before the descending run 21. It is then clear that each permutation is uniquely determined by its descending runs and the order in which they appear when $w$ is read from left to right. On the other hand, not all orderings will be possible. For example, in $\mathfrak{S}_{3}$, there is no way for the descending run 0 to come after 21, since the resulting permutation 210 would consist of only a single descending run. Thus in general, one must specify the relative order of only a subset of the decreasing runs in order to uniquely describe a permutation. In \cite{bancroft}, Bancroft showed that this is precisely the data determined by a permutation preorder. More precisely, we have the following.

\begin{proposition}\cite[Proposition~3.8]{bancroft}\cite[Theorem~3.1]{reading}\label{prop:bijection_ppo}\textnormal{(see also \cite[Proposition~2]{petersen})}
    There is a commutative diagram of bijections
    $$\begin{tikzcd}
        \mathfrak{S}_{n+1} \arrow[dr,"\mu" below left] \arrow[r,"\delta"] & \arc_{nc}(n) \arrow[d,"\mu'"]\\
        & \ppo(n)
    \end{tikzcd}$$
    where the maps $\delta$, $\mu'$, and $\mu^{-1}$ are given as follows.
    \begin{enumerate}
        \item Let $w = w_0\cdots w_n \in \mathfrak{S}_{n+1}$ and let $D(w) = \{i \in [n] \mid w_i > w_{i+1}\}$. For $i \in D(w)$ and $w_{i+1} < j < w_i$, let
        $$s_i^j = \begin{cases} \textnormal{u} & j \text{ precedes } w_i \text{ in } w\\\textnormal{o} & j \text{ succeeds } w_{i+1} \text{ in }w.\end{cases}$$
        For each $i \in D(w)$, let $\gamma_{w,i}  = (w_i,s_{w_i + 1}^i\cdots s_{w_{i+1}-1}^i\textnormal{e})$. Then $\delta(w) = \{\gamma_{w,i}\}_{i \in D(w)}$.
        \item Let $\mathcal{A} \in \arc_{nc}(w)$. The underlying partition of $\mu(w')$ is the set of connected components of the corresponding graph; that is, two nodes $i,j \in [n]$ lie in the same partition element if and only if they are connected by a path traveling along (some realization of) the arcs in $\mathcal{A}$. Given $S, S' \in \mu(w)$ with $\overline{S} \cap \overline{S'} \neq \emptyset$, we set $S \preceq_{\mu'(w)} S'$ if there exists an arc $\gamma \in \mathcal{A}$ and a node $j \in S'$ such that $\ell(\gamma), r(\gamma) \in S$ and $\gamma$ passes below $j$.
        \item Let $\P\in \ppo(n)$. For each $S \in \P$, let $w_{S}$ be the word formed by putting the elements of $S$ in decreasing order. Then the partition elements of $\P$ can be uniquely ordered $S_1,\ldots,S_{|\P|}$ such that for all $1 \leq i < j \leq |\P|$ either (i) $S_i \preceq_\P S_j$, or (ii) $S_i$ and $S_j$ are incomparable under $\preceq_\P$ and $\min(S_i) < \max(S_j)$. Then $\mu^{-1}(\P) = w_{S_1}\cdots w_{S_{|\P|}}$.
 	\end{enumerate}
\end{proposition}

\begin{example}\label{ex:bijection_ppo}
    Let $w = 105684273 \in \mathfrak{S}_{9}$. The partition of $w$ into descending runs is $w = 10|5|6|842|73$. In the notation of Proposition~\ref{prop:bijection_ppo}(1), we then have $D(w) = \{0, 4, 5, 7\}$. The noncrossing arc diagram $\delta(w)$ is then as shown in the left diagram of Figure~\ref{fig:bijection_ppo}.

    By construction, we have that the underlying partition of $\mu'(\delta(w)) = \mu(w)$ corresponds to the partition of $w$ into descending runs. Thus in symbols we have $\mu(w) = \{\{0,1\}, \{2,4,8\},\{3,7\}, \{5\}, \{6\}\}$ with $\{5\} \preceq_{\mu(w)} \{2,4,8\} \preceq_{\mu(w)} \{3,7\}$ and likewise replacing $\{5\}$ with $\{6\}$. In \cite{bancroft}, Bancroft visualizes $\mu(w)$ as in the right diagram of Figure~\ref{fig:bijection_ppo}. This visualization can be obtained from the right diagram by ``collapsing'' the nodes in the same connected component while remembering which connected components are drawn above one another. Note that going ``upwards'' in the right diagram represents going ``up'' a relation $\preceq_{\mu(w)}$, but corresponds to passing to a connected component which lies ``lower'' in the left diagram.

    Finally, the unique partial order on the partition elements of $\mu(w)$ which satisfies the conditions (i) and (ii) in item (3) is $\{0,1\} < \{5\} < \{2,4,7\} < \{3,6\}$. Thus we see that $\mu^{-1}(\mu(w)) = w$, as expected.
\end{example}

In \cite{bancroft}, Bancroft considers the set of permutations preorders equipped with the following partial order.

\begin{definition}\label{def:perm_preorder2}
    Let $\P,\Q \in \ppo(n)$. We write $\P \leq_{\ppo} \Q$ if the following hold.
    \begin{enumerate}
        \item As partitions of $[n]$, $\P$ refines $\Q$; that is, for every $S \in \P$ there exists (a unique) $S_\Q \in \Q$ such that $S \subseteq S_\Q$.
        \item For all $S, S' \in \P$ with $S \preceq_\P S'$, one has $S_\Q \preceq_\Q S'_\Q$.
    \end{enumerate}
    By abuse of notation, we denote the poset  $(\ppo(n),\leq_{\ppo})$ just by $\ppo(n)$.
\end{definition}

\begin{remark}
    As the name suggests, one can also view a permutation preorder as a preorder (a partial order without the antisymmetry condition) on $[n]$. The relation $\leq_{\ppo}$ can then be seen as the inclusion order on the set of relations.
\end{remark}

\subsection{The shard intersection order}\label{sec:shards}

In \cite{readingSIO}, Reading defines an alternative partial order on the elements of a finite Coxeter group called the \emph{shard intersection order}. (The partial order is actually defined more generally for any simplicial hyperplane arrangement, but we focus only on the Coxeter group case.) This is done by giving a set of rules which ``break'' the hyperplanes in the corresponding Coxeter arrangement into pieces (called shards), which are either full hyperplanes or convex cones of codimension 1. A shard intersection is then defined to be an intersection of such cones, and the shard intersection order is the reverse containment relation on the set of shard intersections. Based upon the discussion in \cite[Sections~3 and 4]{bancroft} (see also \cite[Section~2.1]{petersen}), we forego the formal definition of the shard intersection order in favor of an explanation of how arcs and permutation preorders can be used to define the corresponding cones in $\mathbb{R}^{n+1}$.

Let $x_0,\ldots,x_n$ denote the standard coordinates in $\mathbb{R}^{n+1}$. The root system of the Coxeter group $\mathfrak{S}_{n+1}$ then sits in the orthogonal complement of the vector $(1,1,\ldots,1)$. We can then specify a (convex) cone in their space via a choice of relations of the form $x_i \leq x_j$ or $x_i = x_j$. This leads to the following definitions.

\begin{definition}\label{def:shards}\
	\begin{enumerate}
		\item Let $\gamma = (\ell(\gamma),s_{\ell(\gamma)+1}\cdots s_{r(\gamma)})$. Denote $S_u = \{i \mid s_i = \textnormal{u}\}$ and $S_o = \{i \mid s_i = \textnormal{o}\}$ and $S_e = \{\ell(\gamma),r(\gamma)\}$. Define a cone $\Sigma(\gamma) \subseteq \mathbb{R}^n$ via the relations:
		$$\{x_i \leq x_j \mid i \in S_e, j \in S_u \cup S_e\} \cup \{x_i \leq x_j \mid i \in S_o \cup S_e, j \in S_e\}.$$
		We refer to $\Sigma(\gamma)$ as the \emph{shard} of $\gamma$. We denote by $\shard(n) = \{\Sigma(\gamma)\mid \gamma \in \arc(n)\}$ the set of shards.
		\item Let $\P \in \ppo(n)$. For $i \in [n]$, denote by $S_i \in \P$ the unique partition element for which $i \in S_i$. We then define a cone $\Psi(\P)$ via the relations
		$\{x_i \leq x_j \mid S_i \preceq_\P S_j\}.$
		We refer to $\Psi(\P)$ as the \emph{shard intersection} of $\P$. We denote by $\shardint(n) = \{\Psi(\P) \mid \P \in \ppo(n)\}$ the set of shard intersections.
	\end{enumerate}
\end{definition}

The following results are all discussed in \cite[Sections~3 and~4]{bancroft} and \cite[Section~2.1]{petersen}. They can also be deduced directly from the definitions.

\begin{proposition}\label{prop:shards}\
	\begin{enumerate}
		\item The map $\Sigma: \arc(n) \rightarrow \shard(n)$ is a bijection.
		\item The map $\Psi(\P): \ppo(n) \rightarrow \shardint(n)$ is a bijection.
		\item Given $\P, \Q \in \ppo(n)$, one has that $\P \leq_\ppo \Q$ if and only if $\Psi(\P) \supseteq \Psi(\Q)$.
	\end{enumerate}
\end{proposition}

In \cite{readingSIO}, Reading defines the shard intersection order as the reverse-containment order on the set of shard intersections. In light of Proposition~\ref{prop:shards}, we take the following as our definition for the purposes of this paper.

\begin{definition-theorem}\label{defthm:sio}
    The \emph{shard intersection order} on $\mathfrak{S}_{n+1}$ is the partial order $\leq_{\mathrm{sio}}$ defined by $w \leq_{\mathrm{sio}} u$ if and only if $\mu(w) \leq_{\ppo} \mu(u)$.
\end{definition-theorem}

Cover relations in $\ppo(n)$ can be characterized as follows:

\begin{proposition}\label{prop:covers_ppo}\cite[Proposition~1.1]{readingSIO}\cite[Proposition~4.1]{bancroft}
	Let $\P \leq_\ppo \Q$ be a relation in $\ppo(n)$. Then the following are equivalent.
		\begin{enumerate}
			\item $(\P \leq \Q) \in \cov(\ppo(n))$.
			\item $|\P| = |\Q| + 1$.
			\item $\dim(\Psi(\P)) = \dim(\Psi(\Q))-1$.
		\end{enumerate}
	In particular, the poset $\ppo(n)$ is graded by the rank function $\P \mapsto n + 1 - |\P|$.
\end{proposition}

One consequence of Proposition~\ref{prop:covers_ppo} is that if $(\P \leq_\ppo \Q) \in \cov(\ppo(n))$, then there exists a shard $\Sigma(\gamma)$ such that $\Psi(\P) \cap \Sigma(\gamma) = \Psi(\Q)$. However, this shard will in general not be unique. One of the goals of this paper is to describe a canonical choice of such a shard (or more precisely of the arc corresponding to such a shard) using permutation preorders. As a result, we establish an edge labeling of $\ppo(n)$ by $\arc(n)$. (We will see $\arc(n)$ as a totally ordered set under the lexicographic order on $[n-1] \times \wrd(\textnormal{ueo})$.) We will then give an interpretation of this edge-labeling using the representation theory of preprojective algebras (Theorem~\ref{thm:tau_exceptional_RAn}) and use this interpretation to prove that it is an ``EL-labeling'' (Theorem~\ref{thm:EL_labeling}).


\subsection{Recursive decomposition}\label{sec:recursive}

The aim of this section is to prove the following result, which will be useful in setting up induction arguments throughout the paper.

\begin{proposition}\label{prop:perm_preorder_recursive}
    Let $\P \in \ppo(n)$. Equip $\prod_{S \in \P} \ppo(|S|-1)$ with the product of the partial orders $\leq_{\ppo}$. (By abuse of notation, we also denote the resulting partial order by $\leq_{\ppo}$). Then there is an order-preserving bijection
    $$\prod_{S \in \P} \ppo(|S|-1) \rightarrow \{\Q \in \ppo(n) \mid \Q \leq_{\ppo} \P\}.$$
\end{proposition}

We prove Proposition~\ref{prop:perm_preorder_recursive} in several steps, beginning with establishing some notation. We then give a detailed example in Example~\ref{ex:ppo_reduction}.

 Fix $\P \in \ppo(n)$. For $S \in \P$, we denote by
    $v_0^S < \cdots < v_{|S|-1}^S$
the elements of $S$ in increasing order.

Let $(\Q_S)_{S \in \P} \in \prod_{S \in \P} \ppo(|S|-1)$. For $S \in \P$ and $T \in \Q_S$, we denote $\widetilde{T} := \{v_i^S\mid i \in T\}$. As a partition of $[n]$, we then define
    \begin{equation}\chi_\P((\Q_S)_S) := \bigcup_{S \in \P}\bigcup_{T \in \Q_S} \widetilde{T}.\end{equation}\label{eqn:chi}
 To make this into a permutation preorder, consider $S, S' \in \P$, $T \in \Q_S$, and $T' \in \Q_{S'}$. If $\overline{\widetilde{T}} \cap \overline{\widetilde{T'}} \neq \emptyset$, we set $\widetilde{T} \preceq_{\chi_\P((\Q_S)_S)} \widetilde{T'}$ if either (i) $S = S'$ and $T \preceq_{\Q_S} T'$ or (ii) $S \neq S'$ and $S \preceq_\P S'$. We then define $\preceq_{\chi_\P((\Q_S)_S)}$ to be the transitive closure of this relation.

\begin{lemma}\label{lem:preorder1}
    $\chi_\P((\Q_S)_S)$ (together with the relation $\preceq_{\chi_\P((\Q_S)_S)}$) is a permutation preorder on $[n]$ and $\chi_\P((\Q_S)_S) \leq_\ppo \P$.
\end{lemma}

\begin{proof}
    We first show that  $\preceq_{\chi_\P((\Q_S)_S)}$ is a partial order. Since we have defined $\preceq_{\chi_\P((\Q_S)_S)}$ as the transitive closure of a reflexive relation, it suffices to show that $\preceq_{\chi_\P((\Q_S)_S)}$ is antisymmetric. To see this, we note that any relation $\widetilde{T} \preceq_{\chi_\P((\Q_S)_S)} \widetilde{T'}$ can be broken into a sequence of relations of the form (i) and (ii) from the definition of $\preceq_{\chi_\P((\Q_S)_S)}$. Since the partial orders $\preceq_\P$ and $\preceq_{\Q_S}$ for $S \in \P$ are all antisymmetric, it follows that $\preceq_{\chi_\P((\Q_S)_S)}$ will be antisymmetric as well.
    
    We now show that $\chi_\P((\Q_S)_S$ is a permutation preorder. It is clear from the construction that $\chi_\P((\Q_S)_S$ satisfies property (b) in Definition~\ref{def:perm_preorder}(2), so consider $S, S' \in \P$, $T \in \Q_S$, and $T' \in \Q_{S'}$ such that $\overline{\widetilde{T}} \cap \overline{\widetilde{T}'} \neq \emptyset$. Then in particular $\overline{S} \cap \overline{S'} \neq \emptyset$, and so without loss of generality $S \preceq_\P S'$ since $\P$ is a permutation preorder. If $S = S'$, if follows that $\overline{T} \cap \overline{T}' \neq \emptyset$. Since $\Q_S$ is a permutation preorder and the map $T \rightarrow\widetilde{T}$ given by $i \mapsto v_i^S$ is order-preserving, this means that without loss of generality $T \preceq_{\Q_S} T'$ . Thus $\widetilde{T} \preceq_{\chi_\P((\Q_S)_S} \widetilde{T'}$, and so property (a) in Definition~\ref{def:perm_preorder}(2) is satisfied in this case. The other possibility is that $S \neq S'$. But then again $\widetilde{T} \preceq_{\chi_\P((\Q_S)_S} \widetilde{T'}$ by construction. We conclude that $\chi_\P((\Q_S)_S$ is indeed a permutation preorder.
    
    It remains to show that $\chi_\P((\Q_S)_S \leq_{\ppo} \P$. This follows by breaking the relations of $\preceq_{\chi_\P((\Q_S)_S)}$ into chains as in the proof of antisymmetry and noting that $\preceq_\P$ is transitive.
\end{proof}

We now construct a candidate inverse for $\chi_\P$. Let $\Q \in \ppo(n)$ and suppose that $\Q \leq_\ppo \P$. For $S \in \P$, and $T \in \Q$ such that $T \subseteq S$, we denote $\widehat{T} = \left\{i \in [|S|-1] \mid v_i^S \in T\right\}$. As a partition of $[|S|-1]$, we then define
    \begin{equation} \xi_\P(\Q)_S = \bigcup_{T \in \Q\text{ s.t. } T \subseteq S} \widehat{T}.\end{equation}\label{eqn:xi}
To make this into a permutation preorder, consider $\widehat{T},\widehat{T'} \in \xi_\P(\Q)_S$.  We then set $\widehat{T} \preceq_{\xi_\P(\Q)_S} \widehat{T'}$ whenever $T \preceq_\Q \widehat{T'}$.

\begin{lemma}\label{lem:preorder2}
    $\xi_\P(\Q)_S$ (together with the relation $\preceq_{\xi_\P(\Q)_S}$) is a permutation preorder on $[|S|-1]$.
\end{lemma}

\begin{proof}
    It follows immediately from the fact that $\preceq_\Q$ is a partial order and the fact that the association $v_i^S \mapsto i$ is order-preserving that each $\preceq_{\xi_\P(\Q)_S}$ is a partial order. The fact that $\xi_\P(\Q)_S$ is a permutation order then follows from the fact that $\overline{\widehat{T}} \cap \overline{\widehat{T'}} \neq \emptyset$ if and only if $\overline{T} \cap \overline{T'} \neq \emptyset$ for all $T, T' \in \Q$ such that $T \cup T' \subseteq S$. For example, we have that $\widehat{T} \preceq_{\xi_\P(\Q)_S} \widehat{T'}$ is a cover relation in $\xi_\P(\Q)_S$ if and only if $T \preceq_\Q T'$ and there does not exist $T'' \in \Q$ such that (i) $T'' \subseteq S$, (ii) $T \neq T'' \neq T'$, and (iii) $T \preceq_\Q T'' \preceq_\Q T'$. Finally, it is straightforward
\end{proof}

We denote $\xi_\P(\Q) := (\xi_\P(\Q)_S)_{S \in \P}$. We are now prepared to prove Proposition~\ref{prop:perm_preorder_recursive}.

\begin{proof}[Proof of Proposition~\ref{prop:perm_preorder_recursive}.]
    We will prove that $\chi_\P$ and $\xi_\P$ are order-preserving bijections.

    {\bf Step 1:} $\chi_\P$ is order-preserving. Suppose $(\Q_S)_{S} \leq_{\ppo} (\mathcal{R}_S)_{S}$. It is clear from the construction that $\chi_\P((\Q_S)_S)$ refines $\chi_\P((\mathcal{R}_S)_S)$ as a partition on $[n]$. Thus consider $S_1, S_2 \in \P$, $T_1 \in \Q_{S_1}$, and $T_2 \in \Q_{S_2}$ such that $\widehat{T_1} \preceq_{\chi_\P((\Q_S)_S)} \widehat{T_2}$. By induction, we can assume that this relation has either form (i) or form (ii) from the definition of $\chi_\P$. Suppose first the relation has form (i). Since $\Q_{S_1} \leq_{\ppo} \mathcal{R}_{S_1}$, it follows that there exist $T_1', T_2' \in \mathcal{R}_{S_1}$ such that $T_1 \subseteq T_1', T_2 \subseteq T_2'$, and $T_1' \preceq_{\mathcal{R}_{S_1}} T_2'$. Since the association $i \mapsto v_i^{S_1}$ is order preserving, then also $\widetilde{T_1'} \preceq_{\xi_\P((\mathcal{R}_S)_S)}\widetilde{T_2'}$ and we are done. Now suppose that the relation has form (ii). Then again there exist $T_1' \in \mathcal{R}_{S_1}$ and $T_2' \in \mathcal{R}_{S_2}$ such that $T_1 \subseteq T_1'$ and $T_2 \subseteq T_2'$. Since $S_1 \preceq_\P S_2$ by the assumption the relation has type (ii), we again conclude that $\widetilde{T_1'} \preceq_{\xi_\P((\mathcal{R}_S)_S)}\widetilde{T_2'}$.

    {\bf Step 2:} $\xi_\P$ is order-preserving. This follows using a similar argument as in Step 1.

    {\bf Step 3:} $\chi_\P$ is injective. Suppose $\chi_\P((\Q_S)_S) = \chi_\P((\mathcal{R}_S)_S)$. By construction, it follows that the underlying partitions of each $\Q_S = \mathcal{R}_S$ must coincide. Thus let $S \in \P$ and suppose $T \preceq_{\Q_S} T'$. Then there exists a chain $T = T_0 \preceq_{\Q_S} \cdots \preceq_{\Q_S} T_k = T'$ such that $\overline{T_i} \cap \overline{T_{i_1}} \neq \emptyset$ for all $0 \leq i < k$. By the definition of $\chi_\P$, we then have that $\overline{\widetilde{T_i}} \cap \overline{\widetilde{T_{i+1}}} \neq \emptyset$ and that $\widetilde{T_i} \preceq_{\chi_\P((\mathcal{R}_S)_S)} \widetilde{T_{i+1}}$ for all $i$. Since $\preceq_\P$ is antisymmetric, this means there exists a chain $T_i = T_i^0 \preceq_{\mathcal{R}_S} \cdots \preceq_{\mathcal{R}_S} T_i^{j_i} = T_{i+1}$. We conclude that $T_i \preceq_{\mathcal{R}_S} T_{i+1}$, and so $\Q_S \leq_\ppo \mathcal{R}_S$. By symmetry, it follows that $(\Q_S)_S = (\mathcal{R}_S)_S$.

    {\bf Step 4:} $\xi_\P$ is injective. This can be shown using a similar argument as in Step 3.

    {\bf Conclusion:} Since both $\xi_\P$ and $\chi_\P$ are injective, we have that both $\xi_\P \circ \chi_\P$ and $\chi_\P \circ \xi_\P$ are bijections. Since these compositions are also order-preserving, this implies that they must be the respective identities.
\end{proof}

\begin{example}\label{ex:ppo_reduction}
    Let $\P = \{\{0,2,3\}, \{1,4,5\}\} \in \ppo(5)$ with $S_1 := \{1,4,5\} \preceq_\P  \{0,2,3\} =: S_2$. In the notation above, we have
    $$v_0^{S_1} = 1,\qquad v_1^{S_1} = 4,\qquad v_2^{S_1} = 5,$$
    $$v_0^{S_2} = 0,\qquad v_1^{S_2} = 2,\qquad v_2^{S_2} = 3.$$
    Identifying each $v_i^{S_j}$ with $i$, we then let $\Q_{S_1} = \left\{\left\{v_0^{S_1},v_2^{S_1}\right\},\left\{v_1^{S_1}\right\}\right\}$ with $\left\{v_1^{S_1}\right\} \preceq_{\Q_{S_1}} \left\{v_0^{S_1},v_2^{S_1}\right\}$ and $\Q_{S_2} = \left\{\left\{v_0^{S_1},v_1^{S_2}\right\},\left\{v_2^{S_2}\right\}\right\}$ with $\preceq_{\Q_{S_2}}$ the empty relation. Then $$\mathcal{R}:= \chi_\P(\Q_{S_1},\Q_{S_2}) = \{\{0,2\},\{1,5\},\{3\},\{4\}\}$$
    with cover relations
    $$\{4\} \preceq_\mathcal{R} \{1,5\},\qquad \{1,5\} \preceq_\mathcal{R} \{0,2\},\qquad \{1,5\} \preceq_\mathcal{R} \{3\}.$$
    A visual depiction of the permutation preorders $\P$, $\Q_{S_1}$, and $\Q_{S_2}$ is given in Figure~\ref{fig:ppo_reduction}. Note that we do indeed have $\mathcal{R} \leq_{\ppo} \P$. Note also that the relation $\{4\} \preceq_\mathcal{R} \{3\}$ comes from taking the transitive closure at the end of the definition of $\chi_\P$. Finally, it is straightforward to verify that $\xi_\P(\chi_\P(\Q_{S_1},\Q_{S_2})) = (\Q_{S_1},\Q_{S_2})$.
\end{example}

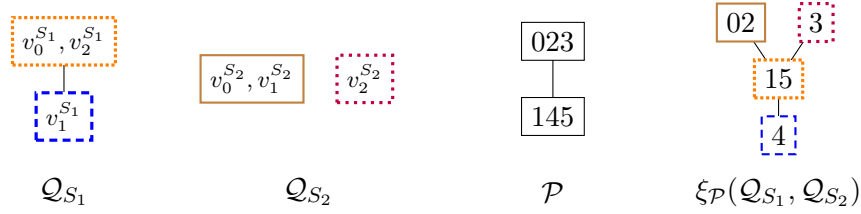
\begin{figure}
\begin{tikzpicture}
    \draw (0,-0.5)--(0,0.5);
    \node [draw=blue,fill = white,very thick, densely dashed] at (0,-0.5) {\footnotesize $v_1^{S_1}$};
    \node [draw=orange,fill = white,very thick, densely dotted] at (0,0.5) {\footnotesize $v_0^{S_1}, v_2^{S_1}$};

    \node at (0,-1.5) {$\Q_{S_1}$};

    \node [draw=purple,fill = white,very thick, dotted] at (4,0) {\footnotesize $v_2^{S_2}$};
    \node [draw=brown,fill = white,thick] at (2.5,0) {\footnotesize $v_0^{S_2}, v_1^{S_2}$};

    \node at (3.25,-1.5) {$\Q_{S_2}$};

    \draw (6.5,-0.5)--(6.5,0.5);
    \node [draw,fill = white] at (6.5,-0.5) {145};
    \node [draw,fill = white] at (6.5,0.5) {023};

    \node at (6.5,-1.5) {$\P$};

    \draw (9.5,0)--(9.5,-0.75);
    \draw (9.5,0)--(9,0.75);
    \draw (9.5,0)--(10,0.75);
    \node [draw=orange,fill = white,very thick,densely dotted] at (9.5,0) {15};
    \node [draw=blue,fill = white,thick,densely dashed] at (9.5,-0.75) {4};
    \node [draw=brown,fill = white,thick] at (9,0.75) {02};
    \node [draw=purple,fill = white,very thick,dotted] at (10,0.75) {3};

    \node at (9.5,-1.5) {$\xi_\P(\Q_{S_1},\Q_{S_2})$};

\end{tikzpicture}
\caption{The permutation preorders in Example~\ref{ex:ppo_reduction}. Boxes with the same color/texture represent the assignment of each $v_i^{S_j}$ in $[n]$. For example, $v_0^{S_1} = 1$ and $v_2^{S_1} = 5$.}\label{fig:ppo_reduction}
\end{figure}


\section{ppo-admissibility}\label{sec:admissibility}

In this section, we define a notion of ``admissibility'' for an arc with respect to a permutation preorder. We then use this notion to define a reduction process for permutation preorders and study the sequences of arcs corresponding to iterations of this reduction process.

\subsection{ppo-admissible arcs}\label{sec:ppo_adm}

In this section, we introduce \emph{ppo-admissible arcs} (Definition~\ref{def:ppo_admissible}). We show that these arcs correspond precisely to those shards which contain some shard intersection. Furthermore, we will see in Section~\ref{sec:wide} that these correspond to the ``bricks'' which live in some ``wide subcategory'' of the preprojective algebra (see Section~\ref{sec:background} for the definitions). The main result of this section is Proposition~\ref{prop:ppo_admissible_recursive}, in which we establish a compatibility between the notion of ppo-admissibility and the recursive nature of permutation preordered established in Section~\ref{sec:recursive}.

\begin{definition}\label{def:ppo_admissible}
    Let $\P \in \ppo(n)$ be a permutation preorder and let $\gamma = \left(\ell(\gamma),s_{\ell(\gamma)+1}\cdots s_{r(\gamma)}\right) \in \arc(n)$. We say that $\gamma$ is \emph{$\P$-ppo-admissible} if the following hold.
    \begin{enumerate}
        \item There exists $S_\gamma \in \P$ such that $\ell(\gamma), r(\gamma) \in S(\gamma)$.
        \item Let $S_\gamma \neq S \in \P$ and suppose that $S_\gamma \preceq_\P S$. Then $s_k = \textnormal{o}$ for all $\ell(\gamma) < k < r(\gamma)$ with $k \in S$.
        \item Let $S_\gamma \neq S \in \P$ and suppose that $S \preceq_\P S_\gamma$. Then $s_k = \textnormal{u}$ for all $\ell(\gamma) < k < r(\gamma)$ with $k \in S$.
    \end{enumerate}
    We denote by $\arc_\P(n)$ the set of $\P$-ppo-admissible arcs.
\end{definition}

\begin{example}\label{ex:ppo_adm}\
    \begin{enumerate}
        \item Let $\{[n]\} \in \ppo(n)$ be the permutation preorder consisting of a single partition element. Then every arc in $\arc(n)$ is $\{[n]\}$-$\ppo$-admissible.
        \item Let $\P \in \ppo(n)$ be the permutation preorder where each $i \in [n]$ lies in its own partition element. Note that this means $\preceq_{\P}$ is the empty relation. Then there are no arcs which are $\P$-$\ppo$-admissible.
        \item Let $w \in \mathfrak{S}_{n+1}$. Then every arc in the noncrossing arc diagram $\delta(w)$ is $\mu(w)$-$\ppo$-admissible. In particular, this example specializes to (1) and (2) by taking $w = n(n-1)\cdots 0$ and $w = 01\cdots n$, respectively.
        \item Consider $w$ as in Example~\ref{ex:bijection_ppo} (see Figure~\ref{fig:bijection_ppo}). Then there are six arcs which are $\mu(w)$-$\ppo$-admissible. They are the four arcs in $\delta(w)$ together with the arcs formed by gluing together the ends of $\gamma_{w,4}$ and $\gamma_{w,3}$ at their contested endpoint and perturbing the resulting ``arc'' so that it passes either above or below the node 4. For example, an arc connecting the nodes 5 and 6 will not be $\mu(w)$-$\ppo$-admissible because its endpoints will lie in distinct partition elements of $\mu(w)$. Likewise, an arc connecting the nodes 2 and 4 while passing below the node 3 will not be $\mu(w)$-$\ppo$-admissible because it fails condition (3) in Definition~\ref{def:ppo_admissible} for $S_\gamma = \{2,4,8\}$ and $S = \{3,7\}$.
    \end{enumerate}
\end{example}

Definition~\ref{def:ppo_admissible} is partially motivated by the following result.

\begin{proposition}[Proposition~\ref{prop:mainA}]\label{prop:ppo_admissible_shard}
	Let $\P \in \ppo(n)$ and $\gamma \in \arc(n)$. Then $\gamma$ is $\P$-ppo-admissible if and only if $\Psi(\P) \subseteq \Sigma(\gamma)$.
\end{proposition}

\begin{proof}
	For $i \in [n]$, let $S_i \in \P$ be the partition element for which $i \in S_i$. Similarly denote $S_u, S_o$, and $S_e$ as in Definition~\ref{def:shards}.
	
	Suppose first that $\gamma$ is $\P$-ppo-admissible and let $x \in \Sigma(\gamma)$. Let $S_\gamma$ be as in Definition~\ref{def:ppo_admissible}. Since $S_e \subseteq S_\gamma$, we have that $x_{\ell(\gamma)} = x_{r(\gamma)}$. Now let $i \in S_u$. Since $i \in \overline{S_\gamma}$, this means $S_i \preceq_\P S_\gamma$, and thus that $x_i \leq x_{\ell(\gamma)}$ and $x_i \leq x_{r(\gamma)}$. The situation with $i \in S_o$ is similar. We conclude that $x \in \Sigma(\gamma)$.
	
	Now suppose that $\Sigma(\gamma) \supseteq \psi(\P)$. Then $x_{\ell(\gamma)} = x_{r(\gamma)}$ for all $x \in \Psi(\P)$, and so $S_{\ell(\gamma)} = S_{r(\gamma)}$. Now consider $\ell(\gamma) < i < r(\gamma)$. If $i \in S_{\ell(\gamma)}$ there is nothing to show. Otherwise either $S_{\ell(\gamma)} \preceq_\P S_i$ (and so $i \in S_u$) or $S_i \preceq_\P S_{\ell(\gamma)}$ (and so $i \in S_o$). We conclude that $\gamma$ is $\P$-ppo-admissible.
	\end{proof}
	
As an immediate consequence of Proposition~\ref{prop:ppo_admissible_shard}, we obtain the following.

\begin{lemma}\label{lem:ppo_admissible}
    Let $\P \leq_\ppo \Q$ be permutation preorders. Then every $\P$-ppo-admissible arc is also $\Q$-ppo-admissible.
\end{lemma}

We conclude this section by establishing 
the relationship between ppo-admissibility and the recursive nature of permutation preorders described in Section~\ref{sec:recursive}. In what follows, we freely borrow notation from the arguments in Section~\ref{sec:recursive}.

\begin{proposition}\label{prop:ppo_admissible_recursive}
    Let $\P \in \ppo(n)$.
    \begin{enumerate}
    \item Then there is a bijection
    $$\phi_\P: \arc_\P(n) \rightarrow \bigsqcup_{S \in \P} \arc(|S|-1)$$
    given as follows. Let $\gamma = \left(\ell(\gamma),s_{\ell(\gamma)+1}\cdots s_{r(\gamma)}\right)$. Then there exists a unique $S_\gamma \in \P$ such that $\ell(\gamma) =: v_i^{S_\gamma}, r(\gamma) \in S_\gamma$. Let $s \in \word(\textnormal{oue})$ be the word obtained from $s_{\ell(\gamma)+1}\cdots s_{r(\gamma)}$ by deleting every letter $s_j$ with $j \notin S_\gamma$. We denote $\phi_\P(\gamma) := (i, s) \in \arc(|S_\gamma|-1)$.
    \item Let $\Q \leq_{\ppo}\P$, and let $\gamma \in \arc_{\P}(n)$. Let $S_\gamma \in \P$ be as in (1). Then $\gamma$ is $\Q$-$\ppo$-admissible if and only if $\phi_\P(\gamma)$ is $\xi_\P(\Q)_{S_\gamma}$-$\ppo$-admissible. In particular, $\phi_\P$ restricts to a bijection
    $$\phi_\P : \arc_\Q(n) \rightarrow \bigsqcup_{S \in \P} \arc_{\xi_\P(\Q)_{S}}(|S|-1).$$
    \end{enumerate}
\end{proposition}

\begin{proof}
    (1) It is clear that $\phi_\P$ gives a surjective map $\arc_\P(n) \rightarrow \bigsqcup_{S \in \P} \arc(|S|-1)$. Thus let $S \in \P$ and $\rho = \left(\ell(\rho),s_{\ell(\rho)+1}\cdots s_{r(\rho)}\right) \in \arc(|S|-1)$. We denote $\psi_\P(\rho) = \left(\ell(\psi_\P(\rho)),s'_{\ell(\psi_\P(\rho))+1}\cdots s'_{r(\psi_\P(\rho))}\right)$, the unique arc in $\arc(n)$ which satisfies all of the following.
    \begin{enumerate}
        \item $\ell(\psi_\P(\rho)) = v_{\ell(\rho)}^S$ and $r(\psi_\P(\rho)) = v_{r(\rho)}^S$.
        \item $s'_{v_i^S} = \textnormal{o}$ if and only if $s_i = \text{o}$ for all $\ell(\rho) < i < r(\rho)$.
        \item Let $v_{\ell(\rho)}^S < i < v_{r(\rho)}^S$ such that $i \notin S$, and let $T \in \P$ such that $i \in T$. Then $s'_i = \textnormal{u}$ if and only if $T \preceq_\P S$.
    \end{enumerate}
    It is an immediate consequence of the construction that $\psi_\P(\rho)$ is $\P$-$\ppo$-admissible and that $\phi_\P$ and $\psi_\P$ are inverses.

    (2) Let $\gamma$ and $\ell(\gamma) = v_i^{S_\gamma}$ be as in the statement of (1). Denote $r' := r(\phi_\P(\gamma))$ and $\phi_\P(\gamma) = (i,s'_{i+1}\cdots s'_{r'})$.

    Suppose that $\phi_\P(\gamma)$ is $\xi_\P(\Q)_{S_\gamma}$-$\ppo$-admissible. Then there exists $T_\gamma \in \Q$ with $T_\gamma \subseteq S_\gamma$ such that $\ell(\gamma), r(\gamma) \in T_\gamma$. That is, $\gamma$ satisfies condition (1) in the definition of $\Q$-$\ppo$-admissibility (Definition~\ref{def:ppo_admissible}).

    Now let $T_\gamma \neq T \in \P$ and suppose that $T_\gamma \preceq_\Q T$. Suppose there exists $\ell(\gamma) < k < r(\gamma)$ with $k \in T$. In particular, this means $\overline{T} \cap \overline{T_\gamma} \neq \emptyset$. We then have two possibilities.

    Suppose first that $T \subseteq S_\gamma$. Then by the construction of $\xi_\P$, we have that $\widehat{T_\gamma} \preceq_{\xi_\P(\Q} \widehat{T}$. Since $\phi_\P(\gamma)$ is $\xi_\P(\Q)_{S_\gamma}$-$\ppo$-admissible, it follows that $s_{v_k^{S_\gamma}} = s'_k = \textnormal{o}$.

    Now suppose that $T \not \subseteq S_\gamma$. Then there exists $S_\gamma \neq S \in \P$ such that $T \subseteq S$. In particular, we have that $\emptyset \neq T_\gamma \cap T \subseteq S_\gamma \cap S$, and so $S_\gamma \preceq_\P S$ by the assumption that $\Q \leq_{\ppo} \P$. Since $\gamma$ is $\P$-$\ppo$-admissible, it follows that $s_k = \textnormal{o}$.

    We have thus shown that $\gamma$ satisfies condition (2) in the definition of $\Q$-$\ppo$-admissibility (Definition~\ref{def:ppo_admissible}). The proof that is also satisfies condition (3) is analogous.

    We have proved that if $\phi_\P(\gamma)$ is $\xi_\P(\Q)_{S_\gamma}$-$\ppo$-admissible then $\gamma$ is $\Q$-$\ppo$-admissible. The reverse implication follows from a similar argument. The ``in particular'' part of the statement is then an immediate consequence of Lemma~\ref{lem:ppo_admissible}.
\end{proof}

\begin{example}
    Let $w, \delta(w)$, and $\mu(w)$ be as in Example~\ref{ex:bijection_ppo} (see Figure~\ref{fig:bijection_ppo}). Let $\gamma$ and $\gamma'$ be the arcs formed by gluing together the arcs $\gamma_{w,4}$ and $\gamma_{w,3}$ at their contested endpoint and perturbing upwards and downwards at node 4, respectively. Then $\phi_\P(\gamma)$ and $\phi_\P(\gamma')$ are the arcs on 3 nodes which connect 0 and 2 and pass above and below the node 1, respectively. Note in particular that $\psi_\P(\phi_\P(\gamma)) = \gamma$ and $\psi_\P(\phi_\P(\gamma')) = \gamma'$ pass above the nodes 3 and 7 and below the nodes 5 and 6 as a result of the relations in $\mu(w)$.
\end{example}

One consequence of Proposition~\ref{prop:ppo_admissible_recursive} is the following.

\begin{corollary}\label{cor:clockwise_reduction}
    Let $\P \in \ppo(n)$ and let $\gamma, \rho \in \arc_\P(n)$. Let $S_\gamma, S_\rho \in \P$ be the partition elements for which $\ell(\gamma),r(\gamma) \in S_\gamma$ and $\ell(\rho),r(\rho) \in S_\rho$.
    \begin{enumerate}
        \item Suppose $S_\gamma \neq S_\rho$. Then $\gamma$ and $\rho$ do not have a shared endpoint, a contested endpoint, or a nontrivial crossing.
        \item Suppose $S_\gamma = S_\rho$. Then:
            \begin{enumerate}
            	\item The following are equivalent:
				\begin{enumerate}
					\item $\ell(\gamma) = \ell(\rho)$ (resp. $r(\gamma) = r(\rho)$) and $\rho$ is clockwise of $\gamma$.
          				\item $\ell(\phi_\P(\gamma)) = \ell(\phi_\P(\rho))$ (resp. $r(\phi_\P(\gamma)) = r(\rho_\P(\rho))$) and $\phi_\P(\rho)$ is clockwise of $\phi_\P(\gamma)$.
				\end{enumerate}
            \item The following are equivalent:
				\begin{enumerate}
					\item There is a nontrivial crossing directed from $\gamma$ to $\rho$.
          				\item There is a nontrivial crossing directed from $\phi_\P(\gamma)$ to $\phi_\P(\rho)$.
				\end{enumerate}
	\item The following are equivalent:
				\begin{enumerate}
					\item $\gamma$ and $\rho$ have a contested endpoint.
          				\item $\phi_\P(\gamma)$ and $\phi_\P(\rho)$ have a contested endpoint.
				\end{enumerate}
        \end{enumerate}
    \end{enumerate}
\end{corollary}

\begin{proof}
    (1) We prove the contrapositive. If $\gamma$ and $\rho$ have either a shared endpoint or a contested endpoint, then necessarily $S_1 = S_2$ and we are done. Thus suppose $\gamma = \left(\ell(\gamma),s_{\ell(\gamma)+1}\cdots s_{r(\gamma)}\right)$ and $\rho = \left(\ell(\rho),t_{\ell(\rho)+1}\cdots t_{r(\rho)}\right)$ have a nontrivial crossing. Up to interchanging the roles of $\gamma$ and $\rho$, we can choose nodes $i < j$ which satisfy conditions (1), (2), and (3a) in Remark~\ref{rem:crossing}. We will prove the result in the case $\ell(\gamma) = i \neq \ell(\rho)$ and $s_j < t_j$, the other cases being similar. In this case, we have that $\ell(\rho) < \ell(\gamma) < r(\rho)$, and so $S_\gamma$ and $S_\rho$ are related by $\preceq_\P$. The assumption that $t_i = \text{u}$ then implies that $S_\gamma \preceq_\P S_\rho$. Now let $S \in \P$ such that $j \in S$. Since $s_j \in \{\text{u},\text{e}\}$, we have $S \preceq_\P S_\gamma$. Similarly, since $t_j \in \{\text{e},\text{o}\}$, we have $S_\rho \preceq_\P S$. We conclude that $S_\gamma = S_\rho$.

    (2a) This follows from Definition~\ref{def:clockwise} together with the definition of $\phi_\P$.
    
    (2b) This follows from Definition~\ref{def:hom_direction} and Remark~\ref{rem:crossing} together with the definition of $\phi_\P$.
    
    (2c) This follows immediately from the definition of $\phi_\P$.
\end{proof}

\subsection{ppo-reduction}\label{sec:ppo_adm2}

In this section, we introduce \emph{ppo-reduction} (Definitions~\ref{def:ppo_reduction}). We then use this definitions to construct an edge-labeling of $\ppo(n)$.

\begin{definition}\label{def:ppo_reduction}
    Let $\P \in \ppo(n)$ be a permutation perorder, and let $\gamma = \left(\ell(\gamma),s_{\ell(\gamma)+1}\cdots s_{r(\gamma)}\right)$ be a $\P$-ppo-admissible arc. Then the \emph{ppo-reduction} of $\P$ by $\gamma$, denoted $\J_\P(\gamma)$, is the permutation preorder defined as follows.
    \begin{itemize}
        \item Let $S_\gamma \in \P$ such that $\ell(\gamma), r(\gamma) \in S_\gamma$. Denote
        \begin{eqnarray*}
            U_\P(\gamma) &=& \{i \in S_\gamma \mid i < \ell(\gamma)\} \cup \{i \in S_\gamma \mid \ell(\gamma) < i < r(\gamma) \text{ and } s_i = \textnormal{u}\} \cup \{r(\gamma)\}\\
            L_\P(\gamma) &=& \{\ell(\gamma)\} \cup \{i \in S_\gamma \mid \ell(\gamma) < i < r(\gamma) \text{ and } s_i = \textnormal{o}\} \cup \{i \in S_\gamma) \mid r(\gamma) < i\}.
        \end{eqnarray*}
        As a partition of $[n]$, we set $\J_\P(\gamma) = (\P \setminus \{S_\gamma\}) \cup \{U_\P(\gamma), L_\P(\gamma)\}$.
        \item For $T, T' \in \J_\P(\gamma)$ such that $\overline{T} \cap \overline{T'} \neq \emptyset$, we set $T \preceq_{\J_\P(\gamma)} T'$ if one of the following hold.
        \begin{enumerate}
            \item $T, T' \notin \{U_\P(\gamma), L_\P(\gamma)\}$ and $T \preceq_\P T'$.
            \item $T \in \{U_\P(\gamma), L_\P(\gamma)\}$ and $S_\gamma \preceq_\P T'$. (In particular, $T' \notin \{U_\P(\gamma), L_\P(\gamma)\}$ in this case.)
            \item $T' \in \{U_\P(\gamma), L_\P(\gamma)\}$ and $T\preceq_\P S_\gamma$. (In particular, $T \notin \{U_\P(\gamma), L_\P(\gamma)\}$ in this case.)
            \item Either $T = T'$ or $T = U_\P(\gamma)$ and $T' = L_\P(\gamma)$.
        \end{enumerate}
        We then take $\preceq_{\J_\P(\gamma)}$ to be the transitive closure of this relation.
    \end{itemize}
    Note that $\J_\P(\gamma)$, together with the partial order $\preceq_{\J_\P(\gamma)}$, is indeed a permutation preorder which satisfies $\J_\P(\gamma) \leq_{\ppo} \P$.
\end{definition}

\begin{remark}\label{rem:cover_rel}
    Note that the number of partition elements increases by one when we transition from $\P$ to $\J_\P(\gamma)$. It follows that $\J_\P(\gamma) \leq_{\ppo} \P$ is a cover relation in $\ppo(n)$ by Proposition~\ref{prop:covers_ppo}. Moreover, we will show in Theorem~\ref{thm:edge-labeling} all cover relations are of this form.
\end{remark}

\begin{remark} The notation $\J_\P(\gamma)$ is intentionally based on that used to denote the ``$\tau$-perpendicular subcategory'' of some module (see Definition~\ref{def:tau_in_wide}(2) below). The precise relationship between these two reduction processes is described in Lemma~\ref{lem:tau_perp}.
\end{remark}

\begin{remark}\label{rem:ppo_reduction}
    The partial order $\preceq_{\J_\P(\gamma)}$ can also be described as follows. Define $\preceq'_{\J_\P(\gamma)}$ as in the second bullet point of Definition~\ref{def:ppo_reduction}, but without the assumption that $\overline{T}\cap \overline{T'} \neq \emptyset$. Then $\preceq'_{\J_\P(\gamma)}$ is a partial order which satisfies Definition~\ref{def:perm_preorder}(2a). The partial order $\preceq_{\J_\P(\gamma)}$ can then be obtained by iteratively deleting cover relations from the partial order $\preceq'_{\J_\P(\gamma)}$ until Definition~\ref{def:perm_preorder}(2b) is satisfied.
\end{remark}

\begin{example}\label{ex:ppo_reduction2}\
    \begin{enumerate}
        \item Let $\P = \{[11]\}$ and let $\gamma = (2,\textnormal{oouououe}) \in \arc(11)$. Then $\{0,1,5,7,9,10\}\preceq_{\J_{\{[11]\}}(\gamma)} \{2,3,4,6,8,11\}$, and these are the only partition elements of $\J_{\{[11]\}}(\gamma)$. Moreover, an arc $\rho \in \arc(n)$ is $\J_{\{[11]\}}(\gamma)$-$\ppo$-admissible if and only if it can be drawn entirely within one of the two outlined regions shown in Figure~\ref{fig:adm_sectors}. Note in particular that any arc which is both $\J_{\{[11]\}}(\gamma)$-$\ppo$-admissible and shares an endpoint with $\gamma$ will necessarily be clockwise from $\gamma$. We will prove that this is a general phenomenon in Corollary~\ref{cor:ppo_adm_clockwise_ordered}.
        
        \item Let $w, \delta(w)$, and $\mu(w)$ be as in Example~\ref{ex:bijection_ppo} (see also Figure~\ref{fig:bijection_ppo}). Let $\gamma =  (2,\textnormal{oouuoe})$ be the arc obtained by gluing together the contested endpoint of $\gamma_{w,4}$ and $\gamma_{w,3}$ and perturbing upwards at the node 4. Likewise, let $\rho = (2,\textnormal{ouuuoe})$ be the arc obtained by perturbing the endpoint downwards. Then $U_{\delta(w)}(\gamma) = \{4,8\}, L_{\delta(w)} = \{2\}$, $U_{\delta(w)}(\rho) = \{8\}$, and $L_{\delta(w)}(\rho) = \{2,4\}$. The cover relations in $\J_\P(\gamma)$ and $\J_\P(\rho)$ are then:
        $$\{5\} \preceq_{\J_\P(\gamma)} \{4,8\},\qquad \{6\} \preceq_{\J_\P(\gamma)} \{4,8\},\qquad \{4,8\} \preceq_{\J_\P(\gamma)} \{3,7\},$$
        $$\{5\} \preceq_{\J_\P(\rho)} \{3,7\},\qquad \{6\} \preceq_{\J_\P(\rho)} \{3,7\},\qquad \{2,4\} \preceq_{\J_\P(\rho)} \{3,7\}.$$
        Note in particular that $U_{\delta(w)}(\gamma) \not \preceq_{\J_\P(\gamma)} L_{\delta(w)}(\gamma)$ because $\overline{L_{\delta(w)}(\gamma)} \cap \overline{U_{\delta(w)}(\gamma)} = \emptyset,$ and likewise for $\rho$. 
    \end{enumerate}
\end{example}

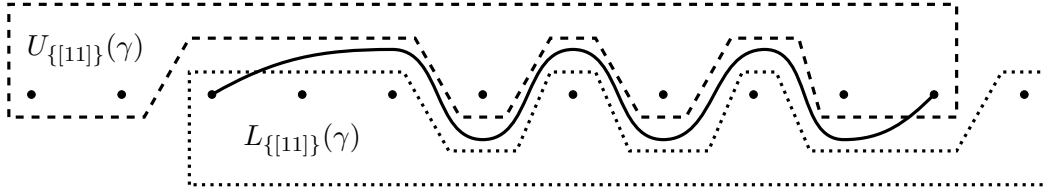
\begin{figure}
    \begin{tikzpicture}[scale=0.6]
        
    \draw[very thick,smooth]
     (2,0) [out = 30, in = 180] to (6,1) [out = 0, in = 180] to (8,-1) [out = 0, in = 180] to (10,1) [out = 0, in = 180] to (12, -1) [out = 0, in = 180] to (14.25, 1) [out = 0, in = 180] to (16, -1) [out = 0, in = -135] to (18,0);

     \draw[very thick, dashed]
     (-2.5, -0.5) -- (-2.5, 2) -- (18.5, 2) -- (18.5, -0.5) -- (15.5, -0.5) -- (15, 1.25) -- (13.5, 1.25) -- (12.5, -0.5) -- (11.5, -0.5) -- (10.5, 1.25) -- (9.5, 1.25) -- (8.5, -0.5) -- (7.5, -0.5) -- (6.5, 1.25) -- (1.5, 1.25) -- (0.5, -0.5) -- (-2.5, -0.5);

     \draw[very thick, dotted]
     (1.5, -2) -- (1.5, 0.5) -- (6.25, 0.5) -- (7.25, -1.25) -- (8.75, -1.25) -- (9.5, 0.5) -- (10.5, 0.5) -- (11.25, -1.25) -- (13, -1.25) -- (13.75, 0.5) -- (14.5, 0.5) -- (15.25, -1.25) -- (18.5, -1.25) -- (19.5, 0.5) -- (20.5, 0.5) -- (20.5,-2) -- (1.5,-2);

    \filldraw 
     (0,0) circle (2.4pt)
     (-2,0) circle (2.4pt)
     (2,0) circle (2.4pt)
     (4,0) circle (2.4pt) 
     (6,0) circle (2.4pt) 
     (8,0) circle (2.4pt) 
     (10,0) circle (2.4pt) 
     (12,0) circle (2.4pt) 
     (14,0) circle (2.4pt) 
     (16,0) circle (2.4pt) 
     (18,0) circle (2.4pt) 
     (20,0) circle (2.4pt) 
     ;

     \node at (-0.8,1) {$U_{\{[11]\}}(\gamma)$};
     \node at (4,-1) {$L_{\{[11]\}}(\gamma)$};
    \end{tikzpicture}
    \caption{The underlying partition of $\J_{\{[11]\}}(\gamma)$ for $\gamma$ as in Example~\ref{ex:ppo_reduction2}(1).}\label{fig:adm_sectors}
\end{figure}

The following establishes a compatibility between ppo-reduction and the recursive nature of permutation preorders established in Section~\ref{sec:recursive}. We freely borrow notation from the previous two subsections.

\begin{proposition}\label{prop:reduction_ppo_compatible}
    Let $\Q \leq_{\ppo} \P \in \ppo(n)$ and let $\gamma \in \arc_{\Q}(n)$. Let $S_\gamma \in \P$ be the partition element for which $\ell(\gamma), r(\gamma) \in S_\gamma$. For $S \in \P$, denote $\Q^0_S:= \xi_\P(\Q)_S$ and 
    $$\Q^1_S := \begin{cases} \J_{\Q^0_S}(\phi_\P(\gamma)) & S = S_\gamma\\\Q^0_S & S \neq S_\gamma.\end{cases}$$
    Then $\xi_\P(\J_\Q(\gamma))_S = \Q^1_S$.
\end{proposition}

\begin{proof}
    First note that, as is implicit in the statement of the theorem, the arc $\phi_\P(\gamma)$ is indeed $\Q_S^0$-$\ppo$-admissible by Proposition~\ref{prop:ppo_admissible_recursive}.

    Let $S \in \P$. It follows from the definitions that the underlying partitions of $\xi_\P(\J_\Q(\gamma))_S = \Q^1_S$ coincide. Thus we need only show that they also have the same partial order. To do so, it suffices to prove that they coincide on all pairs $T,T' \in \Q^1_S$ which satisfy $\overline{T} \cap \overline{T'} \neq \emptyset$. Recall also that for such pairs, we have that $T \preceq_{\xi_\P(\J_\Q(\gamma))_S} T'$ if and only if $\widetilde{T} \preceq_{\J_\Q(\gamma)} \widetilde{T}'$ by the fact that $\xi_\P$ and $\chi_\P$ are inverses (Proposition~\ref{prop:perm_preorder_recursive}). Thus we need only show that the later condition is equivalent to requiring that $T \preceq_{\Q_S^1} T'$.
    
    Suppose first that $S \neq S_\gamma$. Then $\widetilde{T} \preceq_{\J_\Q(\gamma)} \widetilde{T'}$ if and only if $\widetilde{T} \preceq_\Q \widetilde{T'}$ by Definition~\ref{def:ppo_reduction}. But this is equivalent to requiring $T \preceq_{\Q^0_S} T'$ by the argument at the end of the previous paragraph. We conclude that $\xi_\P(\J_\Q(\gamma))_S = \Q_S^1$ in this case.

    Now suppose $S = S_\gamma$. Then there exists $T_\gamma \subseteq S_\gamma$ such that $T_\gamma \in \Q$ and $\ell(\gamma), r(\gamma) \in T_\gamma$. Moreover, we have that $\ell(\phi_\P(\gamma)), r(\phi_\P(\gamma)) \in \widehat{T}_\gamma$ by construction. We then have four cases to consider.

    Suppose first that $T \not \subseteq \widehat{T}_\gamma$ and $T' \not \subseteq \widehat{T}_\gamma$. Then we have $T, T' \in \Q_S^0$. It follows that $T \preceq_{\Q_S^1} T'$ if and only if $T \preceq_{\Q_S^0} T'$. Similarly, we have that $\widetilde{T}, \widetilde{T}' \in \Q$ and that $\widetilde{T} \preceq_{\J_\Q(\gamma)} \widetilde{T}'$ if and only if $\widetilde{T} \preceq_{\Q} \widetilde{T}'$. We conclude that $\widetilde{T} \preceq_{\J_\Q(\gamma)} \widetilde{T}'$ if and only if $T \preceq_{\Q_S^1} T'$ since we already know the analogous statement holds for the partial orders $\preceq_\Q$ and $\preceq_{\Q_S^0}$.

    Now suppose that $T \subseteq \widehat{T}_\gamma$ and $T' \not \subseteq \widehat{T}_\gamma$. Then we have $T' \in \Q_0^S$ and $\widetilde{T}' \in \Q$. It follows that $T \preceq_{\Q_1^S} T'$ if and only if $\widehat{T}_\gamma\preceq_{\Q_S^0} T'$, and likewise that $\widetilde{T} \preceq_{\J_\Q(\gamma)} \widetilde{T}'$ if and only if $T_\gamma \preceq_\Q \widetilde{T}'$. As in the previous case, this implies that $T \preceq_{\Q_1^S} T'$ if and only if $\widetilde{T} \preceq_{\J_\Q(\gamma)} \widetilde{T}'$. The case where $T \not\subseteq \widehat{T}_\gamma$ and $T' \subseteq \widehat{T}_\gamma$ is similar.

    Finally, suppose that $T \cup T' \subseteq \widehat{T}_\gamma$. Then $T \preceq_{\Q_S^1} T'$ if and only if either (a) $T = T'$ or (b) $T = U_{\Q_S^0}(\phi_\P(\gamma)) =: U_0$ and $T' = L_{\Q_S^0}(\phi_\P(\gamma)) =: L_0$. Similarly, $\widetilde{T} \preceq_{\J_\Q(\gamma)} \widetilde{T}'$ if and only if either (a') $\widetilde{T} = \widetilde{T}'$ or (b') $\widetilde{T} = U_\Q(\gamma)$ and $\widetilde{T}' = L_\Q(\gamma)$. But by construction, we have $\widetilde{U}_0 = U_\Q(S)$ and $\widetilde{L}_0 = L_\Q(S)$. Thus (a) is equivalent to (a') and (b) is equivalent to (b'). This proves the result.
\end{proof}

We have now laid the framework to prove the main result of this section.

\begin{theorem}[Theorem~\ref{thm:mainB}]\label{thm:edge-labeling}
	Let $(\P \leq_\ppo \Q) \in \cov(\ppo(n))$. There there exists a unique arc $\gamma \in \ppo_\Q(n)$ such that $\P = \J_\Q(\gamma)$. Moreover, this arc satisfies $\Psi(\Q) = \Sigma(\gamma) \cap \Psi(\P)$.
\end{theorem}

\begin{proof}
(Existence): For $i \in [n]$, denote by $S_i \in \Q$ the partition element with $i \in S_i$. By Proposition~\ref{prop:covers_ppo}, we have that $|\P| = |\Q| + 1$. Thus there exist $T_1,T_2 \in \P$ such that $\Q = (\P \setminus \{T_1,T_2\}) \cup \{T_1 \cup T_2\}$. Moreover, by \cite[Section~4]{bancroft} we have that, up to interchanging the roles of $T_1$ and $T_2$, either $T_1 \preceq_\P T_2$ is a cover relation or $T_1$ and $T_2$ are not related by $\preceq_\P$. Set $a = \min T_2$ and $b = \max T_1$. In the first case, we have that $a < b$ automatically since $\overline{T_1} \cap \overline{T_2} \neq \emptyset$. In the second, we can assume $a < b$ without loss of generality.

For $i \in (a,b) \cap \mathbb{Z}$, denote
	$$s_i := \begin{cases} \textnormal{o} & i \in T_1 \text{ or } T_1 \cup T_2 \neq S_i \preceq_\Q T_1 \cup T_2\\ \textnormal{u} & i \in T_2 \text{ or } T_1 \cup T_2 \preceq_\Q S_i \neq T_1 \cup T_2.\end{cases}$$
We claim that $\gamma := (a,s_{a+1}\cdots s_{b-1}\textnormal{e})$ is $\Q$-ppo-admissible and satisfies $\P = \J_\Q(\gamma)$.

We first show that $\gamma \in \ppo_\Q(n)$. It is clear that $\ell(\gamma),r(\gamma) \in T_1 \cup T_2$. Thus let $i \in (a,b) \cap \mathbb{Z}$. If $i \in T_1 \cup T_2$ there is nothing to show. Otherwise there exists $j \in \{1,2\}$ such that $i \in \overline{T_j}$, and so $T_j$ and $S_i$ are related under $\preceq_\P$. If $T_j \preceq_\P S_i$, then $T_1 \cup T_2 \preceq_\Q S_i$ by Definition~\ref{def:perm_preorder2}, and so $s_i = \textnormal{o}$. Similarly, if $S_i \preceq T_j$, then $s_i = \textnormal{u}$. We conclude that $\gamma \in \ppo_\Q(n)$.

We now show that $\P = \J_\Q(\gamma)$. It suffices to show that $\P \leq_{\ppo} \J_\Q(\gamma) \leq \P$ by Proposition~\ref{prop:covers_ppo} and Remark~\ref{rem:cover_rel}. The fact that the underlying partitions of $\P$ and $\J_\Q(\gamma)$ coincide is an immediate consequence of the definitions. Thus let $S \preceq_{\P} S'$ be a cover relation in $\P$, so in particular $\overline{S} \cap \overline{S'} \neq \emptyset$. If neither $S$ nor $S'$ is in $\{T_1,T_2\}$, then by definition we have $S \preceq_\Q S'$, and thus also $S \preceq_{\J_\Q(\gamma)} S'$. Similarly, if $\overline{T_1} \cap \overline{T_2} \neq \emptyset$, then we have $T_1 = U_\Q(\gamma) \preceq_{\J_\Q(\gamma)} L_\Q(\gamma) = T_2$ by construction. Finally, suppose $S \in \{T_1,T_2\}$ and $S' \notin \{T_1,T_2\}$, the opposite case being similar. Then $T_1 \cup T_2 \preceq_\Q S'$ by Definition~\ref{def:perm_preorder2}, and so $S \preceq_{\J_\Q(\gamma)} S'$ by construction. This completes the proof.

(Uniqueness): We prove uniqueness in the case $\Q = \{[n]\}$. Let $\gamma, \rho \in \arc(n)$ and suppose $\J_{\{[n]\}}(\gamma) = \J_{\{[n]\}}(\rho)$. Denote $L := L_{\{[n]\}}(\gamma)$ and $U := U_{\{[n]\}}(\gamma)$. Then either (a) $L = L_{\{[n]\}}(\rho)$ and $U = U_{\{[n]\}}(\rho)$, or (b) $L = U_{\{[n]\}}(\rho)$ and $U = L_{\{[n]\}}(\rho)$. If we are in case (a), then it follows immediately from the definitions of $L$ and $U$ that $\gamma = \rho$. Thus suppose we are in case (b). Then we must have that $\overline{L} \cap \overline{U} = \emptyset$. The only two arcs which satisfy this property are those with left endpoint 0 and right endpoint $n$ which pass on the same side of every node in $(0,n)$. Thus suppose $\gamma$ is the arc which passes above every such node and that $\rho$ is the arc which passes below every such node. Then the arc $(0,\textnormal{e})$ is not $\J_{\{[n]\}}(\gamma)$-ppo-admissible, while the arc $(n-1,\textnormal{e})$ is. Conversely, the arc $(0,\textnormal{e})$ is $\J_{\{[n]\}}(\rho)$-ppo-admissible, while the arc $(n-1,\textnormal{e})$ is not. This contradiction the assumption that $\J_{\{[n]\}}(\gamma) = \J_{\{[n]\}}(\rho)$. This completes the proof in the case $\Q = \{[n]\}$.

We now prove uniqueness in the case $\Q \neq \{[n]\}$. Let $\gamma, \rho \in \arc_\Q(n)$ and suppose $\J_{\Q}(\gamma) = \J_{\Q}(\rho)$. Let $S \in \Q$ such that $\ell(\gamma), r(\gamma) \in S$, and note that also $\ell(\rho),r(\rho) \in S$ by the assumption. Denote $\Q_{S} := \xi_\Q(\Q)_{S} = \{[|S|-1]\} \in \ppo([|S|-1])$. Then by Propositions~\ref{prop:perm_preorder_recursive} and~\ref{prop:reduction_ppo_compatible}, we have that $\J_{\{[|S|-1]\}}(\phi_\Q(\gamma)) = \J_{\{[|S|-1]\}}(\phi_\Q(\rho))$. Since $\phi_\Q$ is a bijection, we have therefore reduced to the case $\Q = \{[n]\}$. This concludes the proof.

(Moreover part): Given the explicit description in the existence proof, the fact that $\Psi(\Q) = \Sigma(\gamma) \cap \Psi(\P)$ follows from direct computation.
\end{proof}

In the setup of Theorem~\ref{thm:edge-labeling}, we denote $\gamma =: \arclab(\P \leq_{\ppo} \Q)$. The theorem thus says we have a well-defined map $\arclab: \cov(\ppo(n)) \rightarrow \arc(n)$. In Section~\ref{sec:EL}, we will place a total order on $\arc(n)$ in order to treat this map as an edge-labeling.

\subsection{ppo-admissible sequences}\label{ppo_admissible3}

In this section, we iterate the process of ppo-reduction and examine the corresponding sequences of arcs arising from Theorem~\ref{thm:edge-labeling}. We begin with the following definition.

\begin{definition}\label{sat_top}
	Let $\X = (\X,\leq_\X)$ be a poset.
	\begin{enumerate}
		\item Let $\mathfrak{c} = (x_k \leq_\X x_{k-1} \leq_\X \cdots \leq_\X x_0)$ be a chain in $\X$. We say that $\mathfrak{c}$ is a \emph{saturated chain} if each relation $x_j \leq_\X x_{j-1}$ is a cover relation. We denote by $\satc(\X)$ the set of saturated chains in $\P$.
		\item Suppose that $\X$ has a unique maximal element $\widehat{1}$, and let $\mathfrak{c} = (x_k \leq_\X x_{k-1} <\leq_\X \cdots \leq_\X x_0)$ be a saturated chain $\X$. We say that $\mathfrak{c}$ is a \emph{saturated top chain} if $x_0 = \widehat{1}$. We denote by $\stopc(\X)$ the set of saturated top chains in $\X$.
	\end{enumerate}
\end{definition}

Note in particular that if $\X$ is finite, then every maximal chain is also a saturated top chain. We now state the main definition of this section.

\begin{definition}\label{def:ppo_admissible_sequence}
    Let $\omega = (\gamma_k,\ldots,\gamma_1)$ be a sequence of arcs in $\arc(n)$. Denote $\P_0 = \{[n]\}$, and for $1 \leq i < k$, recursively denote $\P_i := \J_{\P_{i-1}}(\gamma_i)$. We say that $\omega$ is a \emph{ppo-admissible sequence} if $\gamma_i$ is $\P_{i-1}$-ppo-admissible for all $1 \leq i \leq k$. For $k \in \{0,1,\ldots,n\}$, we denote by $\ppoadm(n,k)$ the set of ppo-admissible sequences of length $k$ in $\arc(n)$ and $\ppoadm(n) = \bigcup_{k \in \N} \ppoadm(n,k)$.
\end{definition}

\begin{remark}
    Observe that the maximum possible length of a ppo-admissible sequence is $n$ as a consequence of Remark~\ref{rem:cover_rel}. Moreover, by either Example~\ref{ex:ppo_adm}(3) or Theorem~\ref{thm:edge-labeling}, we have that for every $k < n$ and $(\gamma_k,\ldots,\gamma_1) \in \ppo(n,k)$ there exist $\gamma_{k+1},\ldots,\gamma_n \in \arc(n)$ such that $(\gamma_n,\ldots,\gamma_1) \in \ppo(n,n)$. In other words, every $\ppo$-admissible sequence can be ``completed'' to a $\ppo$-admissible sequence of length $n$. 
\end{remark}

\begin{example}\label{ex:ppo_admissible_sequence}\
    \begin{enumerate}
        \item Let $\omega = (\gamma_4,\gamma_3,\gamma_2,\gamma_1) \in \arc_{cw}(3)$ be as in Figure~\ref{fig:clockwise}. Then $\omega$ is not a $\ppo$-admissible sequence. Indeed, we have $\J_{\{[3]\}}(\gamma_1) = \{\{0,2\}, \{1,3\}\}$ with $\{1,3\} \preceq_{\J_{\{[3]\}}(\gamma_1)} \{0,2\}$. Thus while $\gamma_2$ is $\J_{\{[3]\}}(\gamma_1)$-$\ppo$-admissible, the arc $\gamma_3$ is not. It then follows from Lemma~\ref{lem:ppo_admissible} that $\gamma_3$ is also not $\J_{\J_{\{[4]\}}(\gamma_1)}(\gamma_2)$-$\ppo$-admissible.
        \item Let $\omega = (\gamma_4,\gamma_3,\gamma_2,\gamma_1) \in \arc_{cw}(4)$ be as in Figure~\ref{fig:ppo_adm1}. Then $\omega$ is a $\ppo$-admissible sequence. Indeed, let $\P_0 = \{[4]\}$ and for $i \in \{1,2, 3\}$, let $\P_i = \J_{\P_{i-1}}(\gamma_{i + 4})$. Then $\P_1 = \{\{0\}, \{1,2,3,4\}\}$, $\P_2 = \{\{0\}, \{1\}, \{2, 3, 4\}\}$, and $\P_3 = \{\{0\}, \{1\}, \{2\}, \{3,4\}\}$ with the relations $\preceq_{\P_1}$, $\preceq_{\P_2}$, and $\preceq_{\P_3}$ all empty. 
        \item Let $\omega = (\gamma_3,\gamma_2,\gamma_1) \in \arc_{cw}(6)$ be as in Figure~\ref{fig:ppo_adm2}. Then $\omega$ is a ppo-admissible sequence. Indeed, let $\P_0 = \{[6]\}$ and for $i \in \{1,2\}$, let $\P_i = \J_{\P_{i-1}}(\gamma_i)$. Then $\P_1 = \{\{0,4\},\{1,2,3,5,6\}\}$ with $\{0,4\} \preceq_{\P_1} \{1,2,3,5,6\}$ and $\P_2 = \{\{0,4\},\{1,2,3,5\},\{6\}\}$ with $\{0,4\} \preceq_{\P_2} \{1,2,3,5\}$. 
    \end{enumerate}
\end{example}

\begin{figure}
\begin{tikzpicture}[scale=0.8]

\draw[orange,very thick,smooth,dashed]
    (-2,0) [out = -45,in = 180] to (2,-1.75) [out = 0,in = -120] to (6,0);
\draw[blue,very thick,smooth]
    (0,0) [out = -45,in = 180] to (3,-1.25) [out = 0,in = -135] to (6,0);
\draw[orange,very thick,smooth,dashed]
    (2,0) [out = -45,in = 180] to (4,-0.75) [out = 0,in = -150] to (6,0);
\draw[blue,very thick,smooth]
    (4,0)  to (6,0);

\node [blue] at (4.5,-0.25) {$\gamma_4$};
\node [orange] at (3,-0.3) {$\gamma_3$};
\node [blue] at (1,-0.45) {$\gamma_2$};
\node [orange] at (-1,-0.5) {$\gamma_1$};

\filldraw 
 (0,0) circle (2.4pt)
 (-2,0) circle (2.4pt)
 (2,0) circle (2.4pt)
 (4,0) circle (2.4pt) 
 (6,0) circle (2.4pt) 
 ;
 
\end{tikzpicture}
\caption{The ppo-admissible sequence of arcs in Example~\ref{ex:ppo_admissible_sequence}(2)}\label{fig:ppo_adm1}
\end{figure}

\begin{figure}
\begin{tikzpicture}[scale=0.8]

\draw [blue, very thick] 
    (-2,0) [out = 45,in = 180] to (0.5,1.25) [out = 0,in = 135] to (3,0) [out = -45,in = 180] to (4,-0.5) [out = 0,in = -135] to (5,0) [out = 45,in = 180] to (6.5,1) [out = 0,in = 135] to (8,0);

\draw[orange, very thick,dashed]
    (-2,0) [out = 30,in = 180] to (-0,0.75) [out = 0,in = 135] to (1.5,0) [out = -45,in = 180] to (4,-1.25) [out = 0,in = -120] to (6,0);
;

\draw[purple, very thick,smooth,dotted]
    (-2,0) [out = 60,in=180] to (1,1.75) [out=0,in=120] to (4,0);

\node [purple] at (1,2) {$\gamma_1$};
\node [blue] at (6.5,1.25) {$\gamma_2$};
\node [orange] at (1.5,-0.5) {$\gamma_3$};

\filldraw  (0,0) circle (2.4pt)
(-4,0) circle (2.4pt)
 (-2,0) circle (2.4pt)
 (2,0) circle (2.4pt)
 (4,0) circle (2.4pt) 
 (6,0) circle (2.4pt)
 (8,0) circle (2.4pt)
 ;

\end{tikzpicture}
\caption{The ppo-admissible sequence of arcs in Example~\ref{ex:ppo_admissible_sequence}(3)}\label{fig:ppo_adm2}
\end{figure}
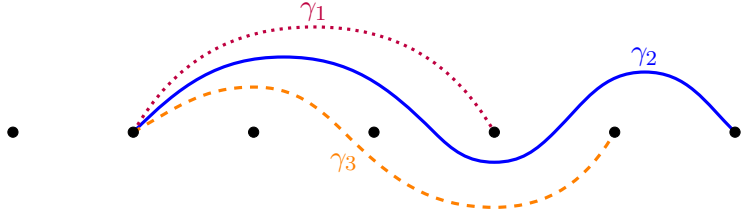

The following is a consequence of Theorem~\ref{thm:edge-labeling}.

\begin{corollary}\label{cor:sat_chain}
	The map $\arclab:\cov(\ppo(n)) \rightarrow \arc(n)$ induces a bijection $\stopc(\ppo(n)) \rightarrow \ppoadm(n)$.
\end{corollary}

We now use Proposition~\ref{prop:reduction_ppo_compatible} to establish several properties of ppo-admissible sequences.

\begin{definition}\label{def:shuffle}\ 
 \begin{enumerate}
        \item Let $\omega_1 = (x_{k_1}^1,\ldots,x_1^1)$ and $\omega_2 = (x_{k_2}^2,\ldots,x_1^2)$ be finite sequences. We denote by $\shuff(\omega_1,\omega_2)$ the set of sequences $(x_{k_1+k_2},\ldots, x_1)$ which satisfy both of the following.
        \begin{enumerate}
            \item For all $m \in \{1,2\}$ and $i \in \{1,\ldots,k_m\}$ there exists a unique $j(m,i) \in \{1,\ldots,k_1+k_2\}$ such that $x_i^m = x_{j(m,i)}$.
            \item For all $m \in \{1,2\}$ and $i < i' \in \{1,\ldots,k_m\}$ we have that $j(m,i) < j(m,i')$.
        \end{enumerate}
        We refer to $\shuff(\omega_1,\omega_2)$ as the set of \emph{shuffles} of $\omega_1$ and $\omega_m$.
        \item Let $\Omega_1$ and $\Omega_2$ be sets of finite sequences, and let $k \in \N$. For each $m \in \{1,2\}$ and $k' < k$, we denote by $\Omega_m^{k'}$ the set of sequences in $\Omega_m$ of length $k'$. We then denote
        $$\shuff(k,\Omega_1,\Omega_2) = \bigcup_{k' = 0}^k \left(\bigcup_{(\omega_1,\omega_2) \in \Omega_1^{k_1} \times \Omega_2^{k-k'}} \shuff(\omega_1,\omega_2)\right).$$
        We refer to $\shuff(k; \Omega_1,\Omega_2)$ as the set of \emph{$k$-shuffles} of the sequences in $\{\Omega_1,\Omega_2\}$.
        \item Let $\Omega_1,\Omega_2,\ldots,\Omega_p$ be sets of finite sequences. We recursively define $$\shuff(k;\Omega_1,\Omega_2,\ldots,\Omega_p) := \shuff(k;\Omega_1,\shuff(k;\Omega_2,\ldots,\Omega_p)).$$
    \end{enumerate}
\end{definition}

\begin{corollary}\label{cor:ppo_seq_shuffle}
    Let $\gamma \in \arc(n)$, and let $\ppoadm(n,k; \gamma)$ be the set of $\ppo$-admissible arc diagrams of length $k$ which start (on the right) with $\gamma$. Let $U:=U_{\{[n]\}}(\gamma)$ and $L:=L_{\{[n]\}}(\gamma)$ be as in Definition~\ref{def:ppo_reduction} and denote $\P_1 = \J_{\{[n]\}}(\gamma)$. Then
    there is a bijection $$\Phi^\gamma_k: \ppoadm(n,k; \gamma) \rightarrow \shuff\left(k-1; \ppoadm(|U|-1), \ppoadm(|L|-1)\right)$$
    given by $(\gamma_k,\ldots,\gamma_2,\gamma) \mapsto (\phi_{\P_1}(\gamma_k),\ldots,\phi_{\P_1}(\gamma_2))$.
\end{corollary}

\begin{proof}
    We prove the result by induction on $k$. For $k = 1$ there is nothing to show. Thus suppose $k > 1$ and that the result holds for all $k-1$.

    Now let $\omega = (\gamma_{k-1},\ldots,\gamma_2,\gamma) \in \ppoadm(n,k-1; \gamma)$ and denote $\Phi_{k-1}^\gamma = (\omega_U,\omega_L)$. By the induction hypothesis, we have that $\omega_1 \in \ppoadm(|U|-1)$ and $\omega_2 \in \ppoadm(|L|-1)$. For $1 < k < k$ recursively denote $\P_i = \J_{\P_{i-1}}(\gamma_i)$. Similarly, let $\Q_U^1 = \{[|U|-1]\}$ and $\Q_L^1 = \{[|L|-1]\}$, and for $S \in \{U,L\}$ and $1 < i < k$ resursively denote
    $$\Q_S^i = \begin{cases} \J_{\Q_S^{i-1}}(\phi_{\P_1}(\gamma_i)) & \ell(\gamma_i) \in S\\\Q_S^{i-1} & \ell(\gamma_i) \notin S.\end{cases}$$
    Then by Proposition~\ref{prop:reduction_ppo_compatible} we have that $\xi_{\P_1}(\P_i) = (\Q_U^i,\Q_L^i)$ for all $1 \leq i < k$. Moreover, we have that the map $\phi_{\P_1}$ induces a bijection
    $$\arc_{\P_{i-1}}(n) \rightarrow \arc_{\Q_U^i}(|U|-1) \sqcup \arc_{\Q_L^i}(|L|-1).$$
    This shows that the map $\phi_{\P_1}$ induces a bijection between those elements of $\ppo(n,k;\gamma)$ which begin with $(\gamma_{k-1},\ldots,\gamma_2,\gamma)$ and those $(k-1)$-shuffles of the sequences in $\ppoadm(|U|-1)$ and $\ppoadm(|L|-1)$ which start with $(\phi_{\P_1}(\gamma_{k-1}),\ldots,\phi_{\P_1}(\gamma_2))$. The fact that $\Phi_k^\gamma$ then follows from the induction hypothesis.
\end{proof}

\begin{corollary}\label{cor:ppo_adm_clockwise_ordered}
    Every ppo-admissible sequence is also a clockwise-ordered arc diagram.
\end{corollary}

\begin{proof}
    Let $\omega = (\gamma_k,\ldots,\gamma_1) \in \ppoadm(n,k)$. We prove the result by induction on $k$. For $k = 1$ there is nothing to show, thus suppose that $k > 1$ and that the result holds for all $k' < k$.

    We first claim that $(\gamma_j,\gamma_1) \in \arc_{cw}(n,2)$ for all $1 < j \leq k$. Indeed, recall that $\gamma_j$ is $\J_{\{[n]\}}(\gamma_1)$-$\ppo$-admissible. Thus suppose $\ell(\gamma_j), r(\gamma_j) \in L_{\{[n]\}}(\gamma_1) =: L$, the case where $\ell(\gamma_j), r(\gamma_j) \in U_{\{[n]\}}(\gamma_1) =:U$ being similar. Write $\gamma_1 = (\ell(\gamma_1),s_{\ell(\gamma_1)+1}^1\cdots s_{r(\gamma_1)})$ and likewise for $\gamma_j$. Now let $i$ be an integer in the common closed support of $\gamma_1$ and $\gamma_j$. There are then three possibilities.
    \begin{itemize}
        \item If $\ell(\gamma_1) = i \neq \ell(\gamma_j)$, then $s_i^j = \textnormal{u}$ by the definition of $\J_{\{[n]\}}(\gamma_1)$-$\ppo$-admissible.
        \item If $\ell(\gamma_1) \neq i = \ell(\gamma_2)$, then $s_i^1 = \textnormal{o}$ by the definition of $L$.
        \item If $\ell(\gamma_1) \neq i \neq \ell(\gamma_2)$, then either (a) $i \in U$ and $s_i^j = \textnormal{u}$ or (b) $i \in L$ and $s_i^1 = \textnormal{o}$ by the same reasoning as in the previous two bullet points.
    \end{itemize}

   Now let $i < i'$ be integers in the common closed support of $\gamma_1$ and $\gamma_j$. Then $i'$ must be as in the third bullet point above, and so $s_{i'}^j \leq s_{i'}^1$. It then follows from Remark~\ref{rem:crossing} and Definition~\ref{def:hom_direction} that there does not exist a nontrivial crossing directed from $\gamma_1$ to $\gamma_j$. Similarly, the three bullet points above, Remark~\ref{rem:crossing}, and Definition~\ref{def:hom_direction} together imply that there does not exist a nontrivial crossing directed from $\gamma_j$ to $\gamma_1$.

   Now suppose that $\ell(\gamma_1) = \ell(\gamma_j)$. (Note that $r(\gamma_1) = r(\gamma_j)$ is not possible due to the assumption that $\ell(\gamma_1) \in L$.) Again using the three bullet points above, it follows that $\gamma_j$ is clockwise from $\gamma_1$. This proves the claim.

   Now by Corollary~\ref{cor:ppo_seq_shuffle} there exist $\omega_U \in \ppoadm(|U|-1)$ and $\omega_L \in \ppoadm(|L|-1)$ and $\omega' \in \shuff(\omega_U,\omega_L)$ such that $\omega = (\omega',\gamma_1)$. Moreover, $\omega_U$ and $\omega_L$ are both clockwise-ordered by the induction hypothesis. The result then follows from Corollary~\ref{cor:clockwise_reduction}.
\end{proof}


\section{Representation theory background}\label{sec:background}

In this section, we begin to shift our focus to the representation theory of finite-dimensional algebras. In particular, we recall the definition of $\tau$-perpendicular subcategories (Definitions~\ref{def:tau_perp} and~\ref{def:tau_in_wide}) and (brick-)$\tau$-exceptional sequences (Definitions~\ref{def:tau_exceptional} and~\ref{def:tau_exceptional_brick}). As we will show explicitly in Sections~\ref{sec:tau_ex_RAn} and~\ref{sec:tau_exceptional_preproj}, these serve as representation-theoretic counterparts to the notions of ppo-reduction and ppo-admissible sequences. Finally, we discuss preprojective algebras of type $A$ in Section~\ref{sec:preproj}. 

Let $\Lambda$ be a finite-dimensional basic algebra over a field $K$. We denote by $\mods\Lambda$ the category of finitely-generated (right) $\Lambda$-modules. By a subcategory, we always means a full subcategory of $\mods\Lambda$ which is closed under isomorphisms. All modules will be assumed to be basic unless otherwise stated, and we write $\rk(-)$ for the number of indecomposable direct summands of a module. In particular, $\rk(\Lambda)$ denotes the number of indecomposable projective modules, or equivalently the number of simple modules. Given $M \in \mods\Lambda$, we denote by $\Fac(M)$ the set of quotient modules of $M$ and by $\add(M)$ the subcategory of direct summands of (finite) direct sums of copies of $M$. We also denote by $M^\perp$ (resp, $\lperp{M}$) the subcategory of those modules $X$ for which $\Hom_\Lambda(M, X) = 0$ (resp. $\Hom_\Lambda(X, M) = 0)$).

\subsection{$\tau$-exceptional sequences}\label{sec:tau_exceptional}

We denote by $\tau$ the Auslander-Reiten translation in the category $\mods\Lambda$. We will use the following characterization of $\tau$ in this paper.

\begin{proposition}\cite[Proposition~5.8]{AS}\label{prop:AS}
    Let $M, N \in \mods\Lambda$. Then $\Hom_\Lambda(N,\tau M) = 0$ if and only if $\Ext^1_\Lambda(M, N') = 0$ for all $N' \in \Fac(M)$.
\end{proposition}

A module $M$ is said to be \emph{$\tau$-rigid} if $\Hom_\Lambda(M, \tau M) = 0$. Such objects are central to the ``$\tau$-tilting theory'' of Adachi, Iyama, and Reiten \cite{AIR}, which has become one of the main research focuses in the study of finite-dimensional algebras in the past decade. Throughout this paper, we make the assumption that $\Lambda$ is \emph{$\tau$-tilting finite} \cite{DIJ}; that is, that there are only finitely many indecomposable $\tau$-rigid modules in $\mods\Lambda$. We denote by $\trig(\Lambda)$ and $\itrig(\Lambda)$ the sets of $\tau$-rigid and indecomposable $\tau$-rigid modules in $\mods\Lambda$, respectively.

A subcategory $\W \subseteq \mods\Lambda$ is called \emph{wide} if it is closed under extensions, kernels, and cokernels. Equivalently, $\W$ is an exact embedded abelian subcategory. Similarly, a subcategory $\mathcal{T} \subseteq \mods\Lambda$ is called a \emph{torsion class} if it is closed under extensions and quotients. We denote by $\wide(\Lambda)$ and $\tors(\Lambda)$ the posets of wide subcategories and torsion classes of $\mods\Lambda$, ordered by inclusion. It is well-known that these posets are (complete) lattices; that is, that any set of wide subcategories (resp. torsion classes) admits a unique least upper bound and a unique greatest lower bound. Moreover, it is a consequence of the $\tau$-tilting finite property that $\wide(\Lambda)$ and $\tors(\Lambda)$ are finite lattices, see \cite[Theorem~4.18]{DIRRT} and \cite[Theorem~3.8]{DIRRT}.

In order to introduce $\tau$-exceptional sequences, we will need the following definition due to Jasso \cite{jasso}.

\begin{definition}\label{def:tau_perp}
    Let $M \in \mods\Lambda$ be $\tau$-rigid. Then the \emph{$\tau$-perpendicular subcategory} of $M$ is
    $$\J(M) := M^\perp \cap \lperp{\tau M}.$$
\end{definition}

The following is the basis of ``$\tau$-tilting reduction'', specialized to the $\tau$-tilting finite case.

\begin{theorem}\label{thm:jasso}
    Let $M \in \mods\Lambda$ be $\tau$-rigid. Then the following hold.
    \begin{enumerate}
        \item \cite[Theorem~2.10]{AIR} There is a unique $\tau$-rigid module $B_M\in \mods\Lambda$ such that (i) $M \oplus B_M$ is basic and $\tau$-rigid, (ii) $\rk(M \oplus B_M) = \rk(\Lambda)$, and (iii) $\lperp{\tau M} = \lperp{\tau(M \oplus B_M)}$. The module $B_M$ is called the \emph{Bongartz complement} of $M$.
        \item \cite[Section~3]{jasso} \textnormal{(see also \cite[Theorem~4.12]{DIRRT})} Let $\Gamma_M = \End_\Lambda(B_M)/[M]$, where $[M]$ is the ideal of morphisms which factor through an object in $\add(M)$. Then there is an exact equivalence of categories $\Hom_\Lambda(B_M,-): \J(M) \rightarrow \mods\Gamma_M$. Moreover, the algebra $\Gamma_M$ is $\tau$-tilting finite and satisfies $\rk(\Gamma_M) + \rk(M) = \rk(\Lambda)$.
        \item \cite[Theorems~4.12 and~4.18]{DIRRT} \textnormal{(see also \cite[Remark~4.3]{BaH2})} The subcategory $\J(M)$ is wide, and every wide subcategory of $\mods\Lambda$ is of the form $\J(M)$ for some $\tau$-rigid module $M$.
    \end{enumerate}
\end{theorem}

\begin{definition}\label{def:rank}
    Let $\W = \J(M)$ in the setting of Theorem~\ref{thm:jasso}(2). Then the \emph{rank} of $\W$ is $\rk(\W) := \rk(\Gamma_M)$. Note that this depends only on the wide subcategory $\W$, and not on the module $M$ realizing $\W = \J(M)$.
\end{definition}

The module $M$ which allows one to realize an arbitrary wide subcategory as $\J(M)$ is generally not unique. We do, however, have the following.

\begin{theorem}\cite[Theorem~10.1]{BM_wide}
    Let $M, N \in \mods\Lambda$ be $\tau$-rigid, and suppose that $M$ is indecomposable. Then $\J(M) = \J(N)$ if and only if $M = N$.
\end{theorem}

As established in Jasso's seminal work \cite{jasso}, Theorem~\ref{thm:jasso} allows one to define $\tau$-rigid modules and $\tau$-perpendicular subcategories within a wide subcategory of $\mods\Lambda$. More precisely, we consider the following.

\begin{definition}\label{def:tau_in_wide}
    Let $\W \subseteq \mods\Lambda$ be a wide subcategory, write $\W = \J(M)$. Let $\tau_\W$ denote the Auslander-Reiten translation in $\W$ (or equivalently in $\Gamma_M$ in the notation of Theorem~\ref{thm:jasso}).
\begin{enumerate}
        \item We say that $N \in \W$ is \emph{$\tau_\W$-rigid} (or \emph{$\tau$-rigid in $\W$}) if $\Hom_\Lambda(N, \tau_\W N) = 0$.
        \item If $N$ is $\tau_\W$-rigid, we denote $\J_\W(N) := \W \cap N^\perp \cap \lperp{\tau_\W N}$, which we call the \emph{$\tau_\W$-perpendicular subcategory of $N$} (or the \emph{$\tau$-perpendicular subcategory of $N$ in $\W$}).
        \item We denote by $\trigW{\W}(\W)$ and $\itrigW{\W}(\W)$ the sets of $\tau_\W$-rigid and indecomposable $\tau_\W$-rigid modules in $\W$, respectively.
    \end{enumerate}
\end{definition}

\begin{remark}
    We note that the definition of $\tau_\W$ in Definition~\ref{def:tau_in_wide} does not depend on the choice of the $\tau$-rigid module $M$ realizing $\W = \J(M)$. Indeed, it follows from Theorem~\ref{thm:jasso} that if $\J(M) = \J(M')$, then there is an equivalence of categories $\mods \Gamma_M \rightarrow \mods\Gamma_{M'}$. Alternatively, if $N \in \J(M)$ is not projective in $\J(M)$, then $\tau_\W M$ is the unique module in $\J(M)$ admitting an exact sequence
    $0 \rightarrow \tau_\W M \rightarrow E \rightarrow M \rightarrow 0$
    which is almost split in $\J(M)$.
\end{remark}

The following are immediate consequences of Theorem~\ref{thm:jasso}, Proposition~\ref{prop:AS}, and the definitions.

\begin{corollary}\label{cor:AS}
    Let $\W \subseteq \mods\Lambda$ be a wide subcategory.
    \begin{enumerate}
        \item Let $M, N \in \W$. Then $\Hom_\Lambda(N, \tau_\W M) = 0$ if and only if $\Ext^1_\Lambda(M, N') = 0$ for all $N' \in \W \cap \Fac N$.
        \item Let $M \in \W$ be $\tau_\W$-rigid. Then $\J_\W(M)$ is a wide subcategory of $\mods\Lambda$.
        \item Let $\U, \W \subseteq \mods\Lambda$ be wide subcategories, and suppose that $\U \subseteq \W$. Then $\U$ is a wide subcategory of $\W$. In particular, there exists $M \in \W$ which is $\tau_\W$-rigid and satisfies $\J_\W(M) = \U$.
        \item In the setting of (3), we have that $\U = \W$ if and only if $\rk(\U) = \rk(\W)$.
    \end{enumerate}
\end{corollary}

We are now prepared to define $\tau$-exceptional sequences.

\begin{definition}\cite{BM_exceptional}\label{def:tau_exceptional}
    Let $\Delta = (M_k,\ldots,M_1)$ be a sequence of indecomposable modules in $\mods\Lambda$. We say that $\Delta$ is a \emph{$\tau$-exceptional sequence} if the following hold.
    \begin{enumerate}
        \item $M_1$ is $\tau$-rigid.
        \item $(M_k,\ldots,M_2)$ is a $\tau$-exceptional sequence in $\J(M)$.
    \end{enumerate}
    We refer to $k$ as the \emph{length} of $\Delta$. We say that $\Delta$ is a \emph{complete $\tau$-exceptional sequence} if in addition $k = \rk(\Lambda)$. For $k \in [\rk(\Lambda)]$, we denote by $\tex(\Lambda,k)$ the set of $\tau$-exceptional sequences of length $k$ in $\mods\Lambda$ and $\tex(\Lambda) = \bigcup_{k = 0}^{\rk(\Lambda)}\tex(\Lambda,k)$.
\end{definition}

Note that the length of a $\tau$-exceptional sequence cannot exceed $\rk(\Lambda)$. This is implicit in Buan and Marsh's original work on (signed) $\tau$-exceptional sequences \cite{BM_exceptional}, and follows from the fact that $\rk(M) \leq \rk(\Lambda)$ for any $M \in \trig(\Lambda)$.

We can also use Corollary~\ref{cor:AS} to characterize $\tau$-exceptional sequences without reference to the Auslander-Reiten translation as follows.

\begin{proposition}\label{prop:tau_exceptional}
    Let $\Delta = (M_k,\ldots,M_1)$ be a sequence of indecomposable modules. Then $\Delta$ is a $\tau$-exceptional sequence if and only if the following hold.
    \begin{enumerate}
        \item $\Hom_\Lambda(M_i,M_j) = 0$ for all $1 \leq i < j \leq k$.
        \item $\Ext_\Lambda^1(M_i,M_i') = 0$ for all $1 \leq i \leq k$ and for all $M'_i \in \left(\bigoplus_{\ell < i} M_\ell\right)^\perp \cap \Fac(M_j)$
        \item $\Ext^1_\Lambda(M_i,M'_j) = 0$ for all $1 \leq i < j \leq k$ and for all $M'_j \in \left(\bigoplus_{\ell \leq i} M_\ell\right)^\perp \cap \Fac(M_j)$.
    \end{enumerate}
\end{proposition}

\begin{remark}\label{rem:stratifying}
   Proposition~\ref{prop:tau_exceptional} highlights some of the similarities of $\tau$-exceptional sequences and other definitions in the literature. For example, if one replaces (2) in Proposition~\ref{prop:tau_exceptional} with (2') $\Ext^1_\Lambda(M_i,M_j) = 0$ for all $1 \leq i \leq j \leq k$, the result is the definition of a \emph{stratifying system}. The explicit connection between stratifying systems and $\tau$-exceptional sequences is described in \cite{MT}. We also refer to the introduction of \cite{MT} and the references therein for information about the rich history and body of research related to stratifying systems.
\end{remark}

The following can be seen as a representation-theoretic analog of Corollary~\ref{cor:sat_chain}.

\begin{theorem}\cite[Theorem~10.1]{BM_wide}\label{thm:sat_top}
    There is a bijection $\tex(\Lambda) \rightarrow \stopc(\wide(\Lambda))$ given as follows. Let $\Delta = (M_k,\ldots,M_1) \in \tex(\Lambda)$ and denote $\W_0 = \mods\Lambda$. For $1 \leq i \leq k$, recursively define $\W_i = \J_{\W_{i-1}}(M_i)$. Then $\Delta$ corresponds to $(\W_k \subset \cdots \subset \W_0)$. Moreover, this map restricts to a bijection between the complete $\tau$-exceptional sequences in $\mods\Lambda$ and the maximal chains in the poset $\wide(\Lambda)$.
\end{theorem}

We note that Theorem~\ref{thm:sat_top} can be generalized to algebras which are not $\tau$-tilting finite by replacing $\wide(\Lambda)$ with the ``poset of $\tau$-perpendicular subcategories'', see \cite[Theorem~6.16]{BuH}.


\subsection{Bricks}\label{sec:bricks}

A module $X \in \mods\Lambda$ is called a \emph{brick} if $\End_\Lambda(X)$ is a division algebra. When $K$ is algebraically closed, this condition simplifies to $\End_\Lambda(X) \cong K$. This is also known to be the case for the type A preprojective algebras considered in the paper. We denote by $\brick(\Lambda)$ the set of bricks in $\mods\Lambda$. More generally, given a wide subcategory $\W \subseteq \mods\Lambda$, we denote by $\brick(\W)$ the set of bricks in $\W$. Note that by the definition we have $\brick(\W) = \W \cap \brick(\Lambda)$. This stands in contrast to $\tau$-rigid modules, where in general $\trigW{\W}(\W)$ and $\W \cap \trig(\Lambda)$ are not related by containment in either direction.

Bricks are intimately related with $\tau$-rigid modules, as the ``brick-$\tau$-rigid correspondence'' of Demonet, Iyama, and Jasso shows.

\begin{theorem}\cite[Theorem~4.1]{DIJ}\label{thm:DIJ}
    Let $M \in \mods\Lambda$ be an indecomposable $\tau$-rigid module. Let $r(M,M)$ be the set of endomorphisms of $M$ which are not invertible, and denote $$\beta(M) := M/\left(\sum_{f \in r(M,M)} \mathrm{im}(f)\right).$$ Then $\beta(M)$ is a brick. Moreover, the association $\beta: \itrig(\Lambda) \rightarrow \brick(\Lambda)$ is a bijection.
\end{theorem}

We emphasize that the version of Theorem~\ref{thm:DIJ} stated here is only true under the assumption of $\tau$-tilting finiteness. Without this assumption, $\beta$ becomes an injection whose image can be described in terms of ``functorially finite torsion classes''.

We can also extend Theorem~\ref{thm:DIJ} to wide subcategories of $\mods\Lambda$ as follows.

\begin{corollary}\label{cor:DIJ}
    Let $\W \subseteq \mods\Lambda$ be a wide subcategory. For $M \in \trigW{\W}(\W)$, define $\beta(M)$ as in Theorem~\ref{thm:DIJ}. Then the association $\beta: \itrigW{\W}(\W) \rightarrow \brick(\W)$ is a bijection.
\end{corollary}

\begin{proof}
    It is straightforward to show that if $M \in \W$, then so is $\beta(M)$. The result thus follows by combining Theorems~\ref{thm:DIJ} and~\ref{thm:jasso}(2).
\end{proof}

\begin{remark}\label{rem:DIJ}
    It is a slight abuse of notation to use $\beta$ for both the brick-$\tau$-rigid correspondence in $\mods\Lambda$ and in a wide subcategory $\W$. To be precise, let us temporarily denote by $\beta_\Lambda$ and $\beta_\W$ the relevant bijections in $\mods\Lambda$ and in $\W$. In general, there may be modules $M \in \W$ which are $\tau_\W$-rigid, but which are not $\tau$-rigid in $\mods\Lambda$. If this is the case, then the domain of $\beta_\W$ will not be contained in that of $\beta_\Lambda$. On the other hand, one will always have that $\mathrm{im}(\beta_\W) = \W \cap \mathrm{im}(\beta_\Lambda)$. This is because any brick in $\W$ is also a brick in $\mods\Lambda$ by definition. Moreover, the formula for $\beta_\W$ and $\beta_\Lambda$ are the same. In light of these observation, we will use the symbol $\beta$ to represent both $\beta_\W$ and $\beta_\Lambda$; however, when we speak of the inverse of $\beta$, we will always indicate the relevant wide subcategory. For example, $\beta_\W^{-1}$ will indicate the inverse of $\beta$ in $\W$.
\end{remark}

In the present paper, we use the brick-$\tau$-rigid correspondence to associate a sequence of bricks to each $\tau$-exceptional sequence. This idea was used in~\cite{BaH2} to study the relationship between $\tau$-exceptional sequences and lattices of wide subcategories. In our work, this will allow us to relate the $\tau$-exceptional sequences of the type A preprojective algebra with those of a certain quotient algebra. To be precise, we make the following definition.

\begin{definition}\cite[Definition~1(3)]{HT}\label{def:tau_exceptional_brick}
    Let $(M_k,\ldots,M_1)$ be a $\tau$-exceptional sequence. We refer to $(\beta(M_k),\ldots,\beta(M_1))$ as the corresponding \emph{brick-$\tau$-exceptional sequence}. We denote by $\btex(\Lambda)$ the set of brick-$\tau$-exceptional sequences in $\mods\Lambda$.
\end{definition}

We conclude this section by recalling the definition semibricks and their relationship with wide subcategories.

\begin{definition}\label{def:semibricks}
    Let $\mathcal{X} \subseteq \mods\Lambda$ be a set of bricks. We say that $\mathcal{X}$ is a \emph{semibrick} if $\Hom_\Lambda(X, Y) = 0 = \Hom_\Lambda(Y,X)$ for all $X \neq Y \in \mathcal{X}$.
\end{definition}

We denote by $\sbrick(\Lambda)$ the set of semibricks in $\mods\Lambda$.

\begin{theorem}\label{thm:semibrick_wide}\cite[1.2]{ringel}
    There is a bijection $\sbrick(\Lambda) \rightarrow \wide(\Lambda)$ given as follows. Each semibrick $\mathcal{X}$ is sent to the subcategory $\Filt(\mathcal{X})$ consisting of those modules $Y$ admiting a filtration
    $$0 = Y_0 \subset \cdots \subset Y_m = Y$$
    such that $Y_j/Y_{j-1} \in \mathcal{X}$ for all $j$. Conversely, each wide subcategory $\W$ is sent to the set of objects which are simple in $\W$.
\end{theorem}


\subsection{Preprojective algebras and gentle quotients}\label{sec:preproj}

In this section, we give background information on preprojective algebras of type $A_n$ and their ``gentle quotients''. We begin with the definitions. (Note that while one may assign a ``preprojective algebra'' to any quiver, see e.g. \cite{CB}, we have simplified the definition for ease of exposition.)

\begin{definition}\label{def:preproj}
Let $\overline{A_n}$ be a double quiver of type $A_n$, with vertices $(A_n)_0 = \{1,\ldots n\}$ and arrows $(A_n)_{1} = \cup_{i = 1}^{n-1} \{a_i: i \rightarrow i+1, a_i^*: i+1 \rightarrow i\}$. Let $x = \sum_{i = 1}^{n-1} a_ia_i^* - a_i^*a_i$ and let $c_2 = \cup_{i = 1}^{n-1} \{a_ia_i^*,a_i^*a_i\}$.
We denote $\Pi(A_n) := K\overline{A_n}/(x)$ and $RA_n := K\overline{A_n}/(c_2)$, which we refer to as the \emph{preprojective algebra} and \emph{gentle quotient algebra} of type $A_n$, respectively.
\end{definition}

\begin{remark}
	Note in particular that $RA_n = \Pi(A_n)/((x)/(c_2))$ is a quotient of the preprojective algebra which has the same underlying quiver. Thus there is a fully faithful functor $\mods RA_n \rightarrow \mods \Pi(A_n)$ which acts as the identity on the level of quiver representations. Moreover, this functor induces a vector space inclusion $\Ext^1_{RA_n}(M,N) \subseteq \Ext^1_{\Pi(A_n)}(M,N)$ for any $M, N \in \mods RA_n$.
\end{remark}

The algebra $RA_n$ is an example of a ``gentle algebra'', and as such its representation theory has been well-studied, see e.g. the introduction of \cite{HI} and the references therein for historical detail. Preprojective algebras have similarly received a lot of attention since their introduction due to their connection with Lie theory, Coxter groups, and cluster algebras. See e.g. the introductions of \cite{GLS,IRRT} and the references therein for a more detailed discussion of both historical and modern perspective.

It was shown by Mizuno \cite{mizuno} that the lattice of torsion classes of $\Pi(A_n)$ is isomorphic to the weak order on $\mathfrak{S}_{n+1}$. (Mizuno proved this for all simply-laced finite Coxeter groups.) In particular, this means $\mods(\Pi(A_n))$ contains only finitely many bricks by Theorem~\ref{thm:DIJ}. These bricks were first modeled in \cite{asai} by applying the (dual) brick-$\tau$-rigid correspondence of \cite{DIJ} to the results of \cite{IRRT}. (See Section~\ref{sec:tau_exceptional_preproj} for an overview of the models of \cite{IRRT} and \cite{asai}.) In the same year, a bijection between the bricks over $RA_n$ and arcs on $n+1$ nodes was given in \cite{BCZ} based in part on classical works on ``string algebras'' \cite{BR,CB_string}. Furthermore, it is shown in both \cite{BCZ} and \cite{DIRRT} that the lattice of torsion classes of $RA_n$ is also isomorphic to the weak order on the corresponding Coxeter group (see also \cite{kase}), a consequence of which is that the fully faithful functor $\mods(RA_n) \rightarrow \mods(\Pi(A_n))$ induces an equality $\brick(RA_n) = \brick(\Pi(A_n))$. (The bijection between bricks and arcs is also established in \cite{enomoto,mizuno2} without explicit reference to the algebra $RA_n$.) It is also worth noting that, while the algebra $\Pi(A_n)$ has wild representation type for $n \geq 6$, the only indecomposable modules over $RA_n$ are the bricks. Thus there are advantages to working over both algebras.

In the remainder of this section, we recall the bijection between $\brick(\Pi(A_n)) = \brick(RA_n)$ and $\arc(n)$ and explain how to determine when the Hom- and Ext-spaces between bricks vanish using the corresponding arcs.

\begin{proposition}\cite[Proposition~4.6]{BCZ}\cite[Proposition~3.7]{mizuno2}\label{prop:arcBricks}
    \begin{enumerate}
        \item There is a bijection $\sigma: \brick(\Pi(A_n)) = \brick(RA_n) \rightarrow \arc(n)$ given as follows. Let $X \in \brick(\Pi(A_n))$. Then $X$ satisfies the following.
        \begin{enumerate}
            \item There exists $i \leq j \in [n]$ such that $\dim(X(k)) = 1$ for $i \leq k \leq j$ and $\dim(X(k)) = 0$ for $k < i$ or $k > j$.
            \item If $\dim(X(i)) = 1 = \dim(X(i+1))$, then one of $X(a_i)$ and $X(a_i^*)$ is 0 and the other is an isomorphism.
        \end{enumerate}
        Denoting $i$ and $j$ as in (a), we define $\sigma(X) = (i-1,s_i\cdots s_j)$ with
        $$s_k = \begin{cases} \textnormal{u} & \text{if } X(a_k^*) \neq 0\\\textnormal{o} & \text{if } X(a_k) \neq 0 \\ \textnormal{e} & \text{if }k = j.\end{cases}$$.

        \item Let $\gamma = (\ell(\gamma),s_{\ell(\gamma)+1}\cdots s_r) \in \arc(n)$. Then the brick $\sigma^{-1}(\gamma)$ is given by $$\sigma^{-1}(\gamma)(k) = \begin{cases} K & \ell(\gamma) < k \leq r(\gamma)\\0 & \text{ otherwise,}\end{cases}$$$$\sigma^{-1}(\gamma)(a_i) = \begin{cases} 1_K & s_k = \textnormal{o}\\0 & \text{otherwise},\end{cases}\qquad\qquad \sigma^{-1}(\gamma)(a_i^*) = \begin{cases} 0 & s_k = \textnormal{u}\\1_K & \text{otherwise}.\end{cases}$$
    \end{enumerate}
\end{proposition}

For brevity, we defer examples of the bijection $\sigma$ to Example~\ref{ex:tex}.

Our ultimate goal is to describe (brick-)$\tau$-exceptional sequences using ppo-admissible sequences of arcs. Doing so requires us to understand when a sequence of bricks satisfies the vanishing conditions in Proposition~\ref{prop:tau_exceptional}. As such, we recall arc-characterizations for the existence of quotient maps and the vanishing of Hom and Ext$^1$. We begin with the following.

\begin{definition}\label{def:quotient_arc}
    Let $\gamma = \left(\ell(\gamma),s_{\ell(\gamma)+1}\cdots s_{r(\gamma)}\right), \rho =  \left(\ell(\rho),t_{\ell(\rho)+1}\cdots t_{r(\rho)}\right)$.
    \begin{enumerate}
        \item We say that $\rho$ is a \emph{restriction} of $\gamma$ if (i) $\ell(\gamma) \leq \ell(\rho) < r(\rho) \leq r(\gamma)$, and (ii) $s_i = t_i$ for every node $i \in (\ell(\rho),r(\rho)) \cap \mathbb{Z}$.
        \item We say that $\rho$ is a \emph{quotient arc} of $\gamma$ if it is a restriction of $\gamma$ which satisfies both of the following.
        \begin{enumerate}
            \item If $\gamma$ and $\rho$ do not have a shared left endpoint, then $s_{\ell(\rho)} = \text{o}$.
            \item If $\gamma$ and $\rho$ do not have a shared right endpoint, then $s_{r(\rho)} = \text{u}$.
        \end{enumerate}
        \item We say that $\rho$ is a \emph{quotient arc} of $\gamma$ if it is a restriction of $\gamma$ which satisfies both of the following.
        \begin{enumerate}
            \item If $\gamma$ and $\rho$ do not have a shared left endpoint, then $s_{\ell(\rho)} = \text{u}$.
            \item If $\gamma$ and $\rho$ do not have a shared right endpoint, then $s_{r(\rho)} = \text{o}$.
        \end{enumerate}
    \end{enumerate}
\end{definition}

For example, in Figure~\ref{fig:clockwise}, we have that $\gamma_1$ is a quotient arc of both $\gamma_2$ and $\gamma_4$ and that $\gamma_4$ is a submodule arc of $\gamma_2$. Moreover, $\gamma_1$ is a restriction of $\gamma_3$, but it is neither a quotient arc nor a submodule arc. The modules corresponding to these ares will be discussed in Example~\ref{ex:tex}(1).

The following is proved in \cite[Section~3.1]{mizuno2} and \cite[Propositions~4.7 and~4.8]{BCZ}. (We recall that the Hom-spaces over $RA_n$ and over $\Pi(A_n)$ always coincide.)

\begin{proposition}\label{prop:quotients}\
    \begin{enumerate}
        \item Let $X \in \brick(\Pi(A_n)) = \brick(RA_n)$ and let $Y \in \mods \Pi(A_n)$ be indecomposable. Then the following are equivalent.
        \begin{enumerate}
            \item $Y$ is a brick and $\sigma(Y)$ is a quotient arc of $\sigma(X)$.
            \item $Y$ is a quotient of $X$ in $\mods(\Pi(A_n))$.
            \item $Y \in \mods(RA_n)$.
        \end{enumerate}
        Moreover, if these equivalent conditions hold then $$\dim\Hom_{RA_n}(X,Y) = 1\qquad\text{and}\qquad\dim\Hom_{RA_n}(Y,X) = 0.$$
        \item Let $X \in \brick(\Pi(A_n)) = \brick(RA_n)$ and let $Y \in \mods \Pi(A_n)$ be indecomposable. Then the following are equivalent.
        \begin{enumerate}
            \item $Y$ is a brick and $\sigma(Y)$ is a submodule arc of $\sigma(X)$.
            \item $Y$ is a submodule of $X$ in $\mods(\Pi(A_n))$.
            \item $Y \in \mods(RA_n)$.
        \end{enumerate}
        Moreover, if these equivalent conditions hold then $$\dim\Hom_{RA_n}(X,Y) = 0\qquad\text{and}\qquad\dim\Hom_{RA_n}(Y,X) = 1.$$
    \end{enumerate}
\end{proposition}

It is also shown in \cite{BCZ,mizuno2} that Proposition~\ref{prop:quotients} can be used to construct a basis of $\Hom_{RA_n}(X,Y)$ for any $X,Y \in \brick(RA_n)$. Furthermore, as shown explicitly in \cite[Proposition~4.9]{HY}, the cardinality of this basis can be computed by counting certain types of intersection points between representatives of the chosen arcs. The dimensions of the vectors spaces $\Ext^1_{RA_n}(X,Y)$ and $\Ext^1_{\Pi(A_n)}(X,Y)$ can then be determined using methods from \cite{BDMTY,CPS} and \cite{CB}, respectively, see \cite[Section~5]{HY}. For the purposes of this paper, we recall only the results we need in order to study $\tau$-exceptional sequences.

\begin{proposition}\label{prop:hom_ext}
	Let $X, Y \in \brick(\Pi(A_n)) = \brick(RA_n)$.
	\begin{enumerate}
		\item If there is a nontrivial crossing directed from $\sigma(X)$ to $\sigma(Y)$, then $\Hom_{RA_n}(X,Y) \neq 0$ and $\Ext^1_{RA_n}(Y,X) \neq 0$ (and thus also $\Ext^1_{\Pi(A_n)}(Y,X) \neq 0$).
		\item If $\sigma(X)$ and $\sigma(Y)$ have a contested endpoint, then $\Hom_{RA_n}(X,Y) = 0$ and $\Ext^1_{RA_n}(X,Y) \neq 0$ (and thus also $\Ext^1_{\Pi(A_n)}(X,Y) \neq 0$).
		\item If $\sigma(X)$ and $\sigma(Y)$ have a shared endpoint and $\sigma(Y)$ is clockwise of $\sigma(X)$, then (i) either $X$ is a quotient of $Y$ or $Y$ is a submodule of $X$, and (ii) $\Hom_{RA_n}(X,Y) = 0 = \Ext^1_{\Pi(A_n)}(X,Y)$ (and thus also $\Ext^1_{RA_n}(X,Y) = 0$).
		\item If $\sigma(X)$ and $\sigma(Y)$ do not have a nontrivial crossing, a contested endpoint, or a shared endpoint, then $\Hom_{RA_n}(X,Y) = 0 = \Ext^1_{\Pi(A_n)}(X,Y)$ (and thus also $\Ext^1_{RA_n}(X,Y) = 0$).
	\end{enumerate}
\end{proposition}

In particular, we note that Proposition~\ref{prop:hom_ext} includes the characterization of semibricks over $RA_n$ in terms of noncrossing arc diagrams. (See also \cite[Theorem~3.13]{mizuno2} for an explicit proof of this characterization over $\Pi(A_n)$.) 

\begin{corollary}\label{cor:noncrossing}\cite[Section~4]{BCZ}
    The map $\sigma$ induces a bijection $\sbrick(RA_n) = \sbrick(\Pi(A_n)) \rightarrow \arc_{nc}(n)$.
\end{corollary}

Another consequence of Proposition~\ref{prop:hom_ext} is the following, see also \cite[Section~6.2]{BaH}, \cite[Section~3.2]{mizuno2}, or \cite[Theorem~6.8]{HY}.

\begin{corollary}\label{cor:weak}
	Let $(X_k,\ldots,X_1)$ be a sequence of objects in $\brick(RA_n) = \brick(\Pi(A_n))$. Then the following are equivalent.
	\begin{enumerate}
		\item $\Hom_{RA_n}(X_i,X_j) = 0 = \Ext^1_{RA_n}(X_i,X_j)$ for all $1 \leq i < j \leq k$.
		\item $\Hom_{\Pi(A_n)}(X_i,X_j) = 0 = \Ext^1_{\Pi(A_n)}(X_i,X_j)$ for all $1 \leq i < j \leq k$.
		\item $(\sigma(X_k),\ldots,\sigma(X_1))$ is a clockwise-ordered arc diagram.
	\end{enumerate}
\end{corollary}


\section{Models for wide subcategories and (brick-)$\tau$-exceptional sequences}\label{sec:mainResults}

In this section, we explicitly relate the notions of ppo-reduction and ppo-admissible sequences from Section~\ref{sec:admissibility} to the $\tau$-tilting reduction and (brick-)$\tau$-exceptional sequences of the algebras $RA_n$ and $\Pi(A_n)$.


\subsection{Wide subcategories and permutation preorders}\label{sec:wide}

The aim of this section is to prove Theorem~\ref{thm:wide} below, which characterizes the wide subcategories over $RA_n$ and $\Pi(A_n)$ using permutation preorders and the notion of ppo-admissibility. Before stating the result, we establish some notation.

Recall from Corollary~\ref{cor:noncrossing} and Proposition~\ref{prop:bijection_ppo} that the map $\sigma^{-1} \circ \delta$ gives a bijection $\mathfrak{S}_{n+1} \rightarrow \sbrick(RA_n) = \sbrick(\Pi(A_n))$. We consider each of $\arc_{nc}(n)$, $\sbrick(RA_n)$ and $\sbrick{\Pi(A_n)}$ as partially ordered sets with the relations induced by the shard intersection order $\leq_{\mathrm{sio}}$ on $\mathfrak{S}_{n+1}$ (see Definition-Theorem~\ref{defthm:sio}) via the bijections $\delta$ and $\sigma^{-1}\circ \delta$. Moreover, we have from Theorem~\ref{thm:semibrick_wide} that there are bijections $\Filt_{RA_n}(-): \sbrick(RA_n) \rightarrow \wide(RA_n)$ and $\Filt_{\Pi(A_n)}(-): \sbrick(\Pi(A_n)) \rightarrow \wide(\Pi(A_n))$, where $\Filt_{RA_n}(-)$ and $\Filt_{\Pi(A_n)}$ denote $\Filt(-)$ considered in the categories $\mods RA_n$ and $\mods \Pi(A_n)$, respectively. We then have the following.

\begin{theorem}\label{thm:wide}\
    \begin{enumerate}
        \item Let $w\in \mathfrak{S}_{n+1}$ and $\gamma \in \arc(n)$. Then the following are equivalent.
        \begin{enumerate}
            \item $\gamma$ is $\mu(w)$-$\ppo$-admissible.
            \item $\sigma^{-1}(\gamma) \in \Filt_{RA_n}(\sigma^{-1} \circ \delta(w))$.
            \item $\sigma^{-1}(\gamma) \in \Filt_{\Pi(A_n)}(\sigma^{-1}\circ \delta(w))$.
        \end{enumerate}

        \item There is a map $\eta: \ppo(n) \rightarrow \wide(RA_n)$ given by $\eta(\P) = \Filt_{RA_n}\left(\sigma^{-1}(\arc_\P(n))\right)$ which fits into a commutative digram of order-preserving bijections as follows:
        $$\begin{tikzcd}[column sep =4em]
        (\mathfrak{S}_{n+1},\leq_{\mathrm{sio}}) \arrow[dr,"\mu" below left] \arrow[r,"\delta"] & \arc_{nc}(n) \arrow[d,"\mu'"] \arrow[r,"\sigma^{-1}"] & \sbrick(RA_n) \arrow[d,"\Filt_{RA_n}"] \arrow[r,equal]& \sbrick(\Pi(A_n))\arrow[d,"\Filt_{\Pi(A_n)}"]\\
        & \ppo(n) \arrow[r,"\eta"] & \wide(RA_n)\arrow[r,"\Filt_{\Pi(A_n)}",yshift = 0.1cm] & \wide(\Pi(A_n))\arrow[l,"(-) \cap \mods RA_n",yshift = -0.1cm].
    \end{tikzcd}$$
    \end{enumerate}
\end{theorem}

\begin{proof}
    (1) $(a \implies b)$ Denote $\W := \Filt_{RA_n}(\sigma^{-1}\circ \delta(w))$ and let $\gamma \in \arc_{\mu(w)}(n)$. We will show that $\sigma^{-1}(\gamma) \in \W$ by induction on $r(\gamma) - \ell(\gamma)$. For the base case, we note that if $\gamma \in \delta(w)$, then clearly $\sigma^{-1}(\gamma) \in \W$. Thus suppose that $\gamma \notin \delta(w)$ and write $\gamma = (\ell(\gamma),s_{\ell(\gamma)+1}\cdots s_{r(\gamma)})$. Let $S \in \P$ be the partition element for which $\ell(\gamma), r(\gamma) \in S$. By the assumption that $\gamma \notin \delta(w)$, it follows that there exists $\ell(\gamma) < j < r(\gamma)$ such that $j \in S$. We then define arcs $\gamma' = \left(\ell(\gamma),s_{\ell(\gamma)+1}\cdots s_{j-1}\textnormal{e}\right)$ and $\gamma'' = \left(j,s_{j+1}\cdots s_{r(\gamma)}\right)$. Then $\gamma'$ and $\gamma''$ are both also $\mu(\delta)$-$\ppo$-admissible. By the induction hypothesis, this means $\sigma^{-1}(\gamma'), \sigma^{-1}(\gamma'') \in \W$. Moreover, we have that one of $\gamma'$ and $\gamma''$ is a quotient arc of $\gamma$ and the other is a submodule arc of $\gamma$. Since the modules $\sigma^{-1}(\gamma)$, $\sigma^{-1}(\gamma')$, and $\sigma^{-1}(\gamma'')$ are at most one-dimensional at every vertex and the supports of $\sigma^{-1}(\gamma')$ and $\sigma^{-1}(\gamma'')$ have no intersection, it then follows from Proposition~\ref{prop:quotients} that there is a short exact sequence in $\mods RA_n$ with middle term $\sigma^{-1}(\gamma)$ and outer terms $\sigma^{-1}(\gamma')$ and $\sigma^{-1}(\gamma'')$. We conclude that $\sigma^{-1}(\gamma) \in \W$.
    
    $(b \implies c)$ This follows from the observation that the fully faithful functor $\mods RA_n \rightarrow \mods \Pi(A_n)$ induces an inclusion $\Filt_{RA_n}(\mathcal{X}) \rightarrow \Filt_{\Pi(A_n)}(\mathcal{X})$ for any $\mathcal{X} \in \sbrick(RA_n) = \sbrick(\Pi(A_n))$.

    $(c \implies a)$ Denote $\W' := \Filt_{\Pi(A_n)}(\sigma^{-1}\circ \delta(w))$ and let $X \in \W'$ be a brick. We will show that $\sigma(X)$ is $\mu(w)$-$\ppo$-admissible by induction on $\dim X$. For the base case, we note that if $X$ is simple in $\W'$ then $X \in \delta(w)$, and so $\sigma(X) \in \ppo(n)$ by Example~\ref{ex:ppo_adm}(3). Otherwise, there exists a short exact sequence
    $$0 \rightarrow Y \rightarrow X \rightarrow Z \rightarrow 0$$
    in $\W'$ with $\dim Y, \dim Z < \dim X$. Since $X$ has dimension at most one at every vertex, it follows that $Y \oplus Z$ is a semibrick. Now decompose $Y \oplus Z = \bigoplus_{j = 1}^k X_j$ intro a direct sum of bricks. By the induction hypothesis, we have that $\sigma(X_j) \in \arc_{\mu(w)}(n)$ for all $j$. Moreover, each $\sigma(X_j)$ is either a subarc arc or quotient arc of $\sigma(X)$ by Proposition~\ref{prop:quotients}. For each $i \in (\ell(\sigma(X)),r(\sigma(X)) \cap \mathbb{Z}$, there are then two possibilities. Either (i) there exists a unique index $j$ such that (i.a) $\ell(X_j) < j < r(X_j)$, (i.b) both $\sigma(X)$ and $\sigma(X_j)$ pass on the same side of the node $i$, and (i.c) for all $j' \neq j$ either $i < \ell(X_{j'})$ or $i > r(X_{j'})$, or (ii) there exists a unique pair of indices $j, j'$ such that (ii.a) $i$ is a contested endpoint between $\sigma(X_j)$, and $\sigma(X_{j'})$, and (ii.b) for all $j' \neq j'' \neq j$ either $i < \ell(X_{j''})$ or $i > r(X_{j'})$. Likewise, there exist unique indices $j_\ell$ and $j_r$ such that $\ell(\sigma(X)) = \ell(\sigma(X_{j_\ell}))$ and $r(\sigma(X_j)) = r(\sigma(X_{j_r}))$. By the assumption that each $\sigma(X_j)$ is $\mu(w)$-$\ppo$-admissible, it follows that all of the endpoints of the arcs $\sigma(X_j)$ lie in the same partition element of $\mu(w)$. In particular, the endpoints of $\sigma(X)$ lie in some partition element $\P$. Moreover, for $\ell(\sigma(X)) < i < r(\sigma(X))$ with $i \notin \P$, we must be in case (i) above. We conclude that $\sigma(X)$ is $\mu(w)$-$\ppo$-admissible.

    (2) The fact that $\Filt_{\Pi(A_n)}\circ\sigma\circ \delta: (\mathfrak{S}_{n+1},\leq_{\mathrm{sio}}) \rightarrow \wide(\Pi(A_n))$  is an order-preserving bijection is implicit in \cite{thomas} and is made precise in \cite[Corollary~4.33]{enomoto_lattice}. Moreover, the fact that $\delta, \mu$, and $\mu'$ are (order-preserving) bijections and fit into the desired commutative diagram is established in Proposition~\ref{prop:bijection_ppo} and Definition-Theorem~\ref{defthm:sio}.

   The fact that $\eta$ is a bijection and that the diagram commutes follows from (1) and the fact that $\W = \Filt(\brick(\W))$ for any wide subcategory $\W$. (The latter is a consequence of Theorem~\ref{thm:semibrick_wide}.) The fact that $\eta$ is order-preserving then follows from the commutativity of the diagram.
\end{proof}

\begin{example}\label{ex:wide}
    Consider $w \in \mathfrak{S}_9$ as in Example~\ref{ex:bijection_ppo}(see also Figure~\ref{fig:bijection_ppo}). Then, using the names for the arc in Figure~\ref{fig:bijection_ppo}, $\eta(\mu(w))$ is the wide subcategory of $\mods RA_n$ whose simple objects are $\sigma^{-1}(\gamma_{w,0})$, $\sigma^{-1}(\gamma_{w,4})$, $\sigma^{-1}(\gamma_{w,3})$, and $\sigma^{-1}(\gamma_{w,6})$. The only bricks in this wide subcategory other than these simple modules are those corresponding to the arcs formed by gluing together $\gamma_{w,4}$ and $\gamma_{w,3}$ at their shared endpoint and perturbing the resulting arc so that it passes either above or below the node 4. As in Example~\ref{ex:ppo_adm}(3), these are precisely the arcs which are $\mu(\delta)$-$\ppo$-admissible.
\end{example}


\subsection{$\tau$-exceptional sequences over $RA_n$}\label{sec:tau_ex_RAn}

In this section, we will prove that the map $\sigma$ induces a bijection between the $\tau$-exceptional sequences over $RA_n$ and the $\ppo$-admissible sequences of arcs on $n$ nodes.

Recall that the bricks, indecomposable $\tau$-rigid modules, and indecomposable modules over $RA_n$ all coincide. In particular, the ``brick-$\tau$-rigid correspondence'' acts as the identity in any wide subcategory of $\mods(RA_n)$, and so the $\tau$-exceptional sequences and brick-$\tau$-exceptional sequences over $RA_n$ are one in the same. Proposition~\ref{prop:hom_ext} and Corollary~\ref{cor:AS} then imply that the arcs corresponding to the bricks in a $\tau$-exceptional sequence must form a clockwise-ordered arc diagram.
The purpose of this section is to show that a those clockwise-ordered arc diagrams which correspond to a $\tau$-exceptional sequences over $RA_n$ under $\sigma$ are precisely the ppo-admissible sequences. The key to this is the following lemma.

\begin{lemma}\label{lem:tau_perp}
    Let $\W \in \wide(RA_n)$ and let $X \in \brick(RA_n)$. Then $\J_\W(X) = \eta(\J_{\eta^{-1}(\W)}(\sigma(X)))$.
\end{lemma}

\begin{proof}
    For readability, denote $\P = \eta^{-1}(\W)$. Note that $\J_\W(X)$ and $\eta(\J_{\P}(\sigma(X)))$ both have rank $\rk(\W)-1$ by Theorem~\ref{thm:DIJ}, Proposition~\ref{prop:covers_ppo}, and the fact that $\eta$ is order-preserving (Theorem~\ref{thm:wide}). Thus by Corollary~\ref{cor:DIJ}(4), it suffices to show that $\J_\W(X) \subseteq \eta(\J_{\P}(\sigma(X)))$, or equivalently that every brick in $\J_\W(X)$ also lies in $\eta(\J_{\P}(\sigma(X)))$.

    Let $Y \in \brick(\W)$ and suppose that $Y \notin \J_\W(X)$. Denote $\sigma(X) =: (\ell_X,s_{\ell_X+1}^X\cdots s_{r_X}^X)$ and likewise for $\sigma(Y)$. We will show that $\sigma(Y)$ is not $\J_{\eta^{-1}(\W)}(\sigma(X))$-$\ppo$-admissible, so in particular $Y \notin \eta(\J_{\P}(\sigma(X)))$ by Theorem~\ref{thm:wide}(1). Let $U(X) := U_{\P}(\sigma(X))$ and $L(X) := L_{\P}(\sigma(X))$.
    
    Recall from Corollary~\ref{cor:DIJ}(1) that $\Hom_{RA_n}(Y,\tau_\W X) = 0$ if and only if $\Ext^1_{RA_n}(X,Y') = 0$ for every (indecomposable) quotient $Y'$ of $Y$ which lies in $\W$. We now have two cases to consider.

    Suppose first that one of $\Hom_{RA_n}(X,Y)$ and $\Ext^1_{RA_n}(X,Y)$ is nonzero. Then $(\sigma(Y),\sigma(X)) \notin \arc_{cw}(n,2)$ by Proposition~\ref{prop:hom_ext}. By Corollary~\ref{cor:clockwise_reduction}(1), this means there exists $S \in \P$ such that that the endpoints of $\sigma(X)$ and $\sigma(Y)$ all lie in $S$. Thus $(\phi_\P(\sigma(Y)),\phi_\P(\sigma(X)) \notin \arc_{cw}(|S|-1,2)$ by Corollary~\ref{cor:clockwise_reduction}(2). Corollary~\ref{cor:ppo_adm_clockwise_ordered} then implies that $\phi_\P(\sigma(Y))$ is not $\J_{\{[|S|-1]\}}(\phi_\P(\sigma(Y)))$-$\ppo$-admissible. We conclude that $\sigma(Y)$ is not $\J_{\P}(\sigma(Y))$-$\ppo$-admissible by Propositions~\ref{prop:reduction_ppo_compatible} and~\ref{prop:ppo_admissible_recursive}(2). This concludes the proof in this case.

    It remains to consider the case where $\Hom_{RA_n}(X,Y) = 0 = \Ext^1_{RA_n}(X,Y)$, or equivalently (by Proposition~\ref{prop:hom_ext}) where $(\sigma(Y),\sigma(X)) \in \arc_{cw}(n,2)$. In this case, there exists $Y' \neq Y$ an indecomposable quotient of $Y$ such that $Y' \in \W$ and $\Ext^1_{RA_n}(X,Y') \neq 0$. Note in particular that $Y'$ is a brick, and so $\sigma(Y') =: \gamma = (\ell(\gamma),s_{\ell(\gamma)+1}\cdots s_{r(\gamma)})$ is a quotient arc of $\sigma(X)$ by Proposition~\ref{prop:quotients}. We now have two cases to consider.
    
   Suppose first that $\sigma(X)$ and $\gamma$ have a nontrivial crossing. If there is a nontrivial crossing directed from $\sigma(X)$ to $\gamma$, then $\Hom_{RA_n}(X,Y') = 0$, contradicting the assumption that $Y' \in \W$. Thus there is a nontrivial crossing directed from $\gamma$ to $\sigma(X)$; that is, there exist nodes $i < j$ satisfying the conditions (1), (2), and (3a) in Remark~\ref{rem:crossing} for $\rho = \sigma(X)$. If $\ell(\gamma) = i$, the fact that $\gamma$ is a quotient arc of $\sigma(Y)$ implies that either $\ell_Y = i$ or $s_i^Y = \textnormal{o}$. Likewise we must have $s_j^Y \leq s_j$. Thus, again by Remark~\ref{rem:crossing}, $\sigma(X)$ and $\sigma(Y)$ have a nontrivial crossing, contradicting the assumption that $(\sigma(Y),\sigma(X)) \in \arc_{cw}(n,2)$.
    
    By Proposition~\ref{prop:hom_ext}, it remains to consider the case where $\sigma(X)$ and $\gamma$ have a contested endpoint. We consider the case $\ell(\gamma) = r_X$, the case where $\ell_X = r(\gamma)$ being similar. In particular, we must then have $\ell_Y < \ell(\gamma)$ by the assumption that $(\sigma(Y),\sigma(X)) \in \arc_{cw}(n,2)$. Let $S_X, S_Y \in \P$ be such that $\ell_X, r_X \in S_X$ and $\ell_Y, r_Y \in S_Y$. Since $\sigma(X)$ and $\gamma$ have a contested endpoint, we must also have that $r(\gamma) \in S_X$. The fact that $\gamma$ is a quotient arc of $\sigma(Y)$ with $\ell_Y < \ell(\gamma)$ then implies that $s_{r_X}^Y = \textnormal{o}$. Since $\sigma(Y)$ is $\P$-$\ppo$-admissible, this implies that $s_{r(\gamma)}^Y \neq \text{u}$. Again since $\gamma$ is a quotient arc of $\sigma(Y)$, we conclude that $r_Y = r(\gamma)$, and so $S_X = S_Y$. Moreover, $r(\gamma) \in L(X)$ and $r_X \in U(X)$, and by construction $U(X) \preceq_{\J_\P(\sigma(X))} L(X)$. The fact that $s_{r_X}^Y = \textnormal{o}$ then implies that $\sigma(Y)$ is not $\J_{\P}(\sigma(X))$-ppo-admissible. This concludes the proof.
\end{proof}

\begin{remark}
	Lemma~\ref{lem:tau_perp} in particular establishes a criteria for determining when two bricks $X, Y \in \brick(RA_n)$ satisfy $\Hom_{RA_n}(X,Y) = 0 = \Hom_{RA_n}(Y,\tau X)$. We note that another combinatorial criterion for the vanishing of $\Hom_{RA_n}(Y,tau X)$ is given independently in \cite[Lemma~3.14]{IW}.
\end{remark}

As a consequence of Lemma~\ref{lem:tau_perp}, we can explicitly describe the subcategory $\J(X)$ as a module category of the product of two smaller algebras which are also of the form $RA_n$.

\begin{corollary}\label{cor:reduction_algebra}
    Let $X \in \brick(RA_n)$. Denote $U(X) := U_{\{[n]\}}(\sigma(X))$ and $L(X) := L_{\{[n]\}}(\sigma(X))$ as in Definition~\ref{def:ppo_reduction}, and denote $\Lambda_X := RA_{|U(X)|-1} \times RA_{|L(X)|-1}$. Then the following hold.
    \begin{enumerate}
	\item The bijection $\phi_{\J_{\{[n]\}}(\sigma(X))}: \arc_{\J_{\{[n]\}}(\sigma(X))}(n) \rightarrow \arc(|U(X)|-1) \sqcup \arc(|U(X)|-1)$  from Proposition~\ref{prop:ppo_admissible_recursive} induces a bijection $F_X := \sigma^{-1} \circ \phi_{\J_{\{[n]\}}(\sigma(X))} \circ \sigma: \brick(\J(X)) \rightarrow \brick(\Lambda_X)$.
	\item Let $Y, Z \in \brick(\J(X))$. Then:
		\begin{enumerate}
			\item $Z$ is a quotient of $Y$ if and only if $F_X(Z)$ is a quotient of $F_X(Y)$.
			\item $\Hom_{RA_n}(Y,Z) = 0$ if and only if $\Hom_{\Lambda_X}(F_X(Y),F_X(Z)) = 0$.
			\item $\Ext^1_{RA_n}(Y,Z) = 0$ if and only if $\Ext^1_{RA_n}(Y,Z) = 0$.
		\end{enumerate}
    \end{enumerate}
\end{corollary}

\begin{proof}
    (1) This is an immediate consequence of Lemma~\ref{lem:tau_perp}.
    
    (2) Let $Y, Z \in \J(X)$. Then $\sigma(Y) = \left(\ell_Y,s_{\ell_Y+1}^Y\cdots s_{r_Y}\right)$ and $\sigma(Z) = \left(\ell_Z,s_{\ell_Z+1}^Z\cdots s_{r_Z}^Z\right)$ are $\J_{\{[n]\}}(\sigma(X))$-admissible by Lemma~\ref{lem:tau_perp}. Now suppose that $\ell(\sigma(Y)) \in U(X)$, the case where  $\ell(\sigma(Y)) \in L(X)$ being similar. If $\ell(\sigma(Z)) \in L(X)$, then $\sigma(Y)$ and $\sigma(Z)$ do not have a shared endpoint, a contested endpoint, or a nontrivial crossing by Corollary~\ref{cor:clockwise_reduction}. Thus the Hom- and Ext$^1$-spaces between $Y$ and $Z$ are all zero by Proposition~\ref{prop:hom_ext}. Moreover, the brick $\sigma^{-1} \circ \phi_{\J_{\{[n]\}}(\sigma(X))} \circ \sigma(Y)$ is a module over $RA_{|U(X)|-1}$ while the brick $\sigma^{-1} \circ \phi_{\J_{\{[n]\}}(\sigma(X))} \circ \sigma(Z)$ is a module over $RA_{|L(X)|-1}$. Thus the Hom- and Ext$^1$-spaces between $\sigma^{-1} \circ \phi_{\J_{\{[n]\}}(\sigma(X))} \circ \sigma(X)$ and $\sigma^{-1} \circ \phi_{\J_{\{[n]\}}(\sigma(X))} \circ \sigma(Z)$ are also zero. This proves (2a), (2b), and (2c) in this case.
    
	Suppose from now on that $\ell(\sigma(Z)) \in U(X)$. Now recall from Proposition~\ref{prop:quotients} that $Z$ is a quotient of $Y$ if and only if $\sigma(Z)$ is a quotient arc of $\sigma(Y)$. By the definition of $\phi_{\J_{\{[n]\}}(\sigma(X))}$, the latter condition is equivalent to $\sigma(F_X(Z))$ being a quotient arc of $\sigma(F_X(Y))$, which again by Proposition~\ref{prop:quotients} is equivalent to $F_X(Z)$ being a quotient of $F_X(Y)$. This proves (2a).
	
	To prove (2b), note that Proposition~\ref{prop:hom_ext} implies that $\Hom_{RA_n}(Y,Z) = 0$ if and only if both (i) there is not a nontrivial crossing directed from $\sigma(Y)$ to $\sigma(Z)$, and (ii) if $\sigma(Y)$ and $\sigma(Z)$ have a shared endpoint, then $\sigma(Z)$ is clockwise of $\sigma(Y)$. By Corollary~\ref{cor:clockwise_reduction}, (i) is equivalent to (i') there is not a nontrivial crossing directed from $\sigma(F_X(Y))$ to $\sigma(F_X(Z))$, and (ii) is equivalent to (ii') if $\sigma(F_X(Y))$ and $\sigma(F_X(Z))$ have a shared endpoint, then $\sigma(F_X(Z))$ is clockwise of $\sigma(F_X(Y))$. But (ii) and (ii') are equivalent to $\Hom_{\Lambda_X}(F_X(Y),F_X(Z)) = 0$ by another application of Proposition~\ref{prop:hom_ext}.
	
	(2c) similarly follows from combining the cases of Proposition~\ref{prop:hom_ext} with Corollary~\ref{cor:clockwise_reduction}. This concludes the proof.
\end{proof}

\begin{remark}\label{rem:reduction_algebra}
   In fact, one can further show that there is an exact equivalence of categories $\J(X) \rightarrow \Lambda_X$ as follows. It is shown in \cite{jasso} that there is a lattice isomorphism $\tors(\J(X)) \cong [\Filt\Fac(X),\lperp{\tau X}] \subseteq \tors(RA_n)$. Since $\tors(RA_n)$ is isomorphic to the weak order on the Coxeter group $A_n$, it follows that $\tors(\J(X))$ is isomorphic to the weak order on the Coxeter group $A_{|L(X)|-1} \times A_{|U(X)|-1}$, see e.g. \cite[Proposition~3.1.6]{BjB}. Now let $B_X$ be the Bongartz complement of $X$, and write $\Gamma_X := \End_{RA_n}(B_X)/[X] \cong KQ/I$ as a quotient of a path algebra by an admissible ideal. In particular, the vertices of $Q$ correspond to the indecomposable direct summands of $B_X$. Since each of these summands has trivial endomorphism ring, $Q$ cannot have any loops and every directed cycle in $Q$ must lie in the ideal $I$. It then follows from \cite[Theorem~3.3]{kase} that $Q = \overline{A_{|L(X)-1|}} \sqcup \overline{A_{|U(X)-1|}}$ and that $I$ is generated by all 2-cycles; i.e., that $KQ/I = RA_{|L(X)|-1} \times RA_{|U(X)|-1}$.
\end{remark}

We now prove the main result of this section.

\begin{theorem}[Theorem~\ref{thm:mainC}, part 1]\label{thm:tau_exceptional_RAn}
    Let $k \in \N$. Then the map $\sigma$ induces a bijection $\tex(RA_n,k)\rightarrow \ppoadm(n,k)$; that is, a sequence of bricks $(X_k,\ldots,X_1)$ is a $\tau$-exceptional sequence in $\mods RA_n$ if and only if $(\sigma(X_k),\ldots,\sigma(X_1))$ is a ppo-admissible sequence.
\end{theorem}

\begin{proof}
    We prove the result by induction on $k$. For $k = 0$ and $k = 1$ there is nothing to show, as either both sets are empty ($k = 0$ case) or $\tex(RA_n,1) = \brick(RA_n)$ and $\ppoadm(n,1) = \arc(n)$ ($k = 1$ case). Thus let $k > 1$ and suppose the result holds for all $k' < k$.

    Let $(X_k,\ldots,X_1) \in \tex(RA_n, k)$, and denote $U(X_1)$ and $L(X_1)$ as in Corollary~\ref{cor:reduction_algebra}. It then follows from Lemma~\ref{lem:tau_perp} that $\sigma(X_j)$ is $\J_{\{[n]\}}(\sigma(X_1))$-$\ppo$-admissible for all $2 \leq j \leq k$. Now denote $\widetilde{\phi}:= \phi_{\J_{\{[n]\}}(\sigma(X_1)} \circ \sigma$. Then by Corollary~\ref{cor:reduction_algebra}, we have
    $$\Delta:= (\widetilde{\phi}(X_k),\ldots,\widetilde{\phi}(X_2)) \in \tex(RA_{|U(X_1)|-1} \times RA_{|L(X_1)|-1}).$$
  Since the algebra $RA_{|U(X_1)|-1} \times RA_{|L(X_1)|-1}$ decomposes as a direct product, there exist $\Delta_1 \in \tex(RA_{|U(X_1)|-1})$ and $\Delta_2 \in \text(RA_{|L(X_1)|-1})$ such that $\Delta \in \shuff(\omega_1,\omega_2)$. By the induction hypothesis, applying the map $\sigma$ to the bricks in $\Delta_1$ and $\Delta_2$ thus yields $\ppo$-admissible sequences $\omega_1 \in \ppoadm(|U(X_1)|-1)$ and $\omega_2 \in \ppoadm(|L(X_1)|-1)$. The map $\sigma$ also associates $\Delta$ to some $\omega \in \shuff(\omega_1,\omega_2)$. Finally, by Corollary~\ref{cor:ppo_seq_shuffle}, we conclude that $(\sigma(X_k),\ldots,\sigma(X_1)) = (\Phi_k^{\sigma(X_1)})^{-1} (\omega)$ is a $\ppo$-admissible sequence. The proof that applying $\sigma^{-1}$ to each element of a $\ppo$-admissible sequence yields a $\tau$-exceptional sequence is similar.
\end{proof}

\begin{example}\label{ex:tex}\
\begin{enumerate}
    \item Consider the clockwise-ordered arc diagram $(\gamma_4,\gamma_3,\gamma_2,\gamma_1)$, which is not a ppo-admissible sequence, from Example~\ref{ex:ppo_admissible_sequence}. The corresponding sequence of bricks in $\mods RA_n$ is $$\left({\footnotesize \begin{matrix}2\\1\end{matrix}}, \ {\scriptsize \begin{matrix}3\\2\\1\end{matrix}}, \ {\footnotesize\begin{matrix}2\\3\end{matrix}}, \ 2\right).$$
    The fact that this is not a $\tau$-exceptional sequence (over either algebra) follows from Proposition~\ref{prop:AS}. Indeed, we have that the simple module $S(3)$ is a quotient of ${\tiny\begin{matrix}3\\2\\1\end{matrix}}$ and that $\Ext^1_{RA_n}(S(2), S(3)) \cong K \cong \Ext^1{\Pi(A_n)}(S(2),S(3))$.
    \item Consider the ppo-admissible sequence $(\gamma_4,\gamma_3,\gamma_2,\gamma_1)$ from Example~\ref{ex:ppo_admissible_sequence}(2). The sequence of bricks in $\mods RA_n$ is $$\left(4,\ {\footnotesize\begin{matrix}3\\4\end{matrix}},\ {\scriptsize \begin{matrix}2\\3\\4\end{matrix}},\ {\tiny\begin{matrix}1\\2\\3\\4\end{matrix}}\right).$$
    It is also straightforward to verify directly that this is also a $\tau$-exceptional sequence, since each module is projective in the relevant wide subcategory.
    \item Consider theppo-admissible sequence $(\gamma_3,\gamma_2,\gamma_1)$ from Example~\ref{ex:ppo_admissible_sequence}(3). The corresponding sequence of bricks in $\mods RA_n$, which is a $\tau$-exceptional sequence by Theorem~\ref{thm:tau_exceptional_RAn}, is
    $$\left({\scriptsize \begin{matrix} \ 3\\2 \ 4\\\phantom{44}5\end{matrix}}, \ {\scriptsize\begin{matrix}\phantom{44} 4 \ 6\\ 3 \ 5\\2\ \phantom{44}\end{matrix}},\ {\scriptsize \begin{matrix}4\\3\\2\end{matrix}}\right).$$
\end{enumerate}
\end{example}


\subsection{(Brick-)$\tau$-exceptional sequences over $\Pi(A_n)$}\label{sec:tau_exceptional_preproj}

We next show how Theorem~\ref{thm:tau_exceptional_RAn} can be used to understand the $\tau$-exceptional sequences over the preprojective algebra $\Pi(A_n)$. We will first show that the brick-$\tau$-exceptional sequences over $RA_n$ and $\Pi(A_n)$ coincide. We will then explain how this can be leveraged into a model for the $\tau$-exceptional sequences over $\Pi(A_n)$.

The following is an immediate consequence of \cite[Theorem~8.10]{BaH2}.

\begin{proposition}\label{prop:BaH2}
    Let $\Lambda, \Lambda'$ be $\tau$-tilting finite-algebras. Suppose there is a lattice isomorphism $G: \tors\Lambda \rightarrow \tors\Lambda'$. In particular, for $X \in \brick(\Lambda)$ there exists a unique $O(X) \in \brick(\Lambda')$ such that $G(\Filt(\Fac X)) = \Filt(\Fac O(X))$. Then the map $O$ induces a bijection $\btex(\mods\Lambda) \rightarrow \btex(\mods\Lambda')$.
\end{proposition}

Theorem~\ref{prop:BaH2} in particular implies the following.

\begin{corollary}[Theorem~\ref{thm:mainC}, part 2]\label{cor:brick_tau}
    For all $k$, the map $\beta$ induces a bijection $\tex(\Pi(A_n),k) \rightarrow \tex(RA_n,k)$; that is, the brick-$\tau$-exceptional sequences over $\Pi(A_n)$ are precisely the $\tau$-exceptional sequences over $RA_n$. In particular, $\sigma \circ \beta$ induces a bijection $\tex(\Pi(A_n),k) \rightarrow \ppoadm(n,k)$.
\end{corollary}

\begin{proof}
    As discussed in \cite[Section~4]{mizuno2}, the map $\T \mapsto \T \cap \mods(RA_n)$ is a lattice isomorphism $\tors(\Pi(A_n)) \rightarrow \tors(RA_n)$. Moreover, this isomorphism induces the identity on bricks as the map $O$ in Proposition~\ref{prop:BaH2}.
\end{proof}

In order to construct the inverse of the bijection $\sigma\circ \beta$ above, we need the following analog of Corollary~\ref{cor:reduction_algebra}.

\begin{proposition}\label{prop:reduction_algebra}
    Let $M \in \trig(\Pi(A_n))$ be indecomposable and choose a representative of $\sigma(\beta(M))$. Denote $U(M) := U_{[n]}(\sigma(\beta(M)))$ and $L(M) := L_{[n]}(\sigma(\beta(M)))$ as in Definition~\ref{def:ppo_admissible}. Then there is an exact equivalence of categories
    $$\J(M) \rightarrow \mods\left(\Pi\left(A_{|L(M)|-1}\right) \times \Pi\left(A_{|U(M)-1|}\right)\right)$$
    which is given on each brick by $\sigma^{-1}\circ \phi_{\J_{\{[n]\}}(\sigma(X))} \circ \sigma$.
\end{proposition}

\begin{proof}
    \cite[Theorem~A]{marks} says that there exists some finite Coxeter group $W$ admitting an exact equivalence of categories $\J(M) \rightarrow \mods(\Pi(W))$. Moreover, as mentioned in Remark~\ref{rem:reduction_algebra}, we have that $\tors(\Pi(W))$ must be isomorphic to the weak order on $A_{|L(M)|-1} \times A_{|U(M)|-1}$. This implies that $W = A_{|L(M)|-1} \times A_{|U(M)|-1}$. The description of the equivalence on the level of bricks then follows from the same argument as in Corollary~\ref{cor:reduction_algebra}.
\end{proof}

We now give a recipe for turning a $\ppo$-admissible sequence on $n$ nodes into a $\tau$-exceptional sequence over $\Pi(A_n)$.

Let $X \in \brick(\Pi(A_n))$. Following \cite[Section~3.1]{asai}, we draw $X$ in ``Young-diagram notation" as follows. Denote $\sigma(X) = \left(\ell(\sigma(X)),s_{\ell(\sigma(X)+1},\ldots,s_{r(\sigma(X)}\right)$. Draw a box with the number $r(\sigma(X))$ at the origin $(0,0) \in \R^2$. For $\ell(\sigma(X)) < i < r(\sigma(X))$ in reverse order, recursively do the following:
\begin{itemize}
    \item If $s_i = \text{u}$, draw a box with the number $i$ one unit above the box with the number $i+1$.
    \item If $s_i = \text{o}$, draw a box with the number $i$ one unit to the left of the box with the number $i+1$.
\end{itemize}
We denote the resulting diagram by $\mathfrak{x}(X)$. (Note that Asai's version of this diagram is obtained by ours by reflecting across the $y$-axis. We have reversed the convention so that the integers increase left-to-right as they do in our arc diagrams.) For example, consider the arc $\gamma_2$ in Figure~\ref{fig:clockwise}. Then
$$
    \begin{tikzcd}
        \node [anchor = center] at (-3,0.25) {\mathfrak{x}(X_2) = };
        
        \node [anchor = center] at (0,0) {6};
        \node [anchor = center] at (-0.5,0) {5};
        \node [anchor = center] at (-0.5,0.5) {4};
        \node [anchor = center] at (-1,0.5) {3};
        \node [anchor = center] at (-1.5,0.5) {2};

        \draw (-1.75,0.75)--(-0.25,0.75);
        \draw (-1.75,0.25)--(0.25,0.25);
        \draw (-0.75,-0.25)--(0.25,-0.25);
        \draw (-1.75,0.75)--(-1.75,0.25);
        \draw (-1.25,0.75)--(-1.25,0.25);
        \draw (-0.75,0.75)--(-0.75,-0.25);
        \draw (-0.25,0.75)--(-0.25,-0.25);
        \draw (0.25,0.25)--(0.25,-0.25);
    \end{tikzcd}
$$

In \cite[Section~6.1]{IRRT}, Iyama, Reading, Reiten, and Thomas model the indecomposable $\tau^{-1}$-rigid modules over $\Pi(A_n)$ using ``Young-like diagrams''. In \cite[Section~3.1]{asai}, Asai obtained the visualization of $\mathfrak{x}(X)$ by passing this model through the dual of the brick-$\tau$-rigid correspondence of Demonet, Iyama, and Jasso. By dualizing the inverse of Asai's bijection, we associate to $\mathfrak{x}(X)$ a ``co-Young diagram'' as follows. Suppose the box of $\mathfrak{x}(X)$ containing the number $\ell(\sigma(X)) + 1$ is located at $(i,j)$. Then place a box at each point $(i',j')$ which satisfies $i \leq i' \leq 0$ and $0 \leq j' \leq j$. Place the number $r(\sigma(X))$ in the box located at $(0,0)$, then add numbers to the remaining boxes so that moving one unit up or to the right decreases the number by one. We denote the resulting diagram by $\mathfrak{y}(X)$. For example, if $X_2$ is as above, then
\begin{equation}\label{eqn:tex_ex}
    \begin{tikzcd}
        \node [anchor = center] at (-3,0.25) {\mathfrak{y}(X_2) = };
        
        \node [anchor = center] at (0,0) {6};
        \node [anchor = center] at (-0.5,0) {5};
        \node [anchor = center] at (-0.5,0.5) {4};
        \node [anchor = center] at (-1,0.5) {3};
        \node [anchor = center] at (-1.5,0.5) {2};
        \node [anchor = center] at (-1.5,0) {3};
        \node [anchor = center] at (-1,0) {4};

        \draw (-1.75,0.75)--(-0.25,0.75);
        \draw (-1.75,0.25)--(0.25,0.25);
        \draw (-1.75,-0.25)--(0.25,-0.25);
        \draw (-1.75,0.75)--(-1.75,-0.25);
        \draw (-1.25,0.75)--(-1.25,-0.25);
        \draw (-0.75,0.75)--(-0.75,-0.25);
        \draw (-0.25,0.75)--(-0.25,-0.25);
        \draw (0.25,0.25)--(0.25,-0.25);
    \end{tikzcd}
\end{equation}

The next step is to associate an indecomposable $\tau$-rigid module $\overline{X}$ to the diagram $\mathfrak{y}(X)$. For each vertex $i$, let $\overline{X}(i)$ be the vector spaces whose basis is the set of boxes in $\mathfrak{y}(X)$ with label $i$. If $B$ is such a box, then if there is a box $B'$ below $B$, then $X(a_i)$ sends the basis element corresponding to $B$ to that corresponding to $B'$. Likewise if there is a box $B''$ to the left of $B$, then the map $X(a_{i-1}^*)$ sends the basis element corresponding to $B$ to that corresponding to $B''$. It then follows from \cite{asai} that $\overline{X} = \beta_{\mods\Pi(A_n)}^{-1}(X)$. For example, the diagram in Equation~\ref{eqn:tex_ex} corresponds to the ($\tau$-rigid) $\Pi(A_4)$-module

\begin{equation}\label{eqn:tex_ex2}
    \begin{tikzcd}
        \beta_{\mods\Pi(A_6)}^{-1}(X_2) = & 0 \arrow[r,shift left]& K \arrow[r,shift left, "\text{\tiny$\begin{bmatrix} 0 \\1\end{bmatrix}$}"] \arrow[l,shift left] & K^2 \arrow[r,shift left, "\text{\tiny$\begin{bmatrix} 0 \ 0\\1 \ 0\end{bmatrix}$}"] \arrow[l,shift left,"\text{\tiny$\begin{bmatrix}1 \ 0\end{bmatrix}$}"]& K^2\arrow[r,shift left,"\text{\tiny$\begin{bmatrix} 1 \ 0\end{bmatrix}$}"]\arrow[l,shift left,"\text{\tiny$\begin{bmatrix} 1 \ 0\\0 \ 1\end{bmatrix}$}"] & K \arrow[r,shift left,"0"] \arrow[l,shift left,"\text{\tiny$\begin{bmatrix} 0\\1\end{bmatrix}$}"]& K\arrow[l,shift left,"1"]
    \end{tikzcd}
\end{equation}

Now let $\W \in \wide(\Pi(A_n))$ and denote by $\P \in \ppo(n)$ the permutation preorder which satisfies $\W = \Filt_{\Pi(A_n)}(\eta(\P))$. As in the proof of Proposition~\ref{prop:reduction_algebra}, it follows from \cite[Theorem~A]{marks} that each partition element $S \in \P$ corresponds to an exact embedding $\mods(\Pi(A_{|S|-1})) \rightarrow \W$. On the level of bricks, this bijection is given by $Z \mapsto \sigma^{-1}\circ \phi_\P^{-1}\circ \sigma(Z)$. By construction, we note that if $X$ is simple in $\mods\Pi(A_{|S|-1})$, then  $X \mapsto \sigma^{-1}\circ \phi_\P^{-1}\circ \sigma(X)$; that is, $ \phi_\P^{-1}\circ \sigma(X) \in (\mu')^{-1}(\P)$ by Theorem~\ref{thm:wide}. Thus for an arbitrary module $Z \in \mods(\Pi(A_{|S|-1}))$, we can describe the image of $Z$ under the embedding $\mods(\Pi(A_{|S|-1})) \rightarrow \W$ by applying the map $\sigma^{-1}\circ \phi_\P^{-1}\circ \sigma(-)$ to each term of its composition series.

More precisely, consider a brick $X \in \brick(\W)$ such that the endpoints of $\sigma(X)$ both lie in $S$. Then we can form the diagram $\mathfrak{y}(\sigma^{-1}\circ\phi_\P\circ\sigma(X))$. Identifying the numbers in the boxes of $\mathfrak{y}(\sigma^{-1}\circ\phi_\P\circ\sigma(X))$ with the simple modules over $\Pi(A_{|S|-1})$, we then apply the map $\sigma^{-1} \circ \phi_\P^{-1} \circ \sigma(-)$ to each of the boxes. We denote the resulting diagram by $\mathfrak{y}_\mathcal{W}(X)$.  For example, consider the wide subcategory $\W_1 = \Filt\left\{\text{\tiny$\begin{matrix}4\\3\\2\\1\end{matrix}$}, \ 2, 3, \text{\footnotesize$\begin{matrix}4\\5\end{matrix}$},\ 6\right\} \subseteq \mods \Pi(A_6)$ and let $X_2$ be as above. Then
\begin{equation*}
    \begin{tikzcd}
        \node [anchor = center] at (-6,0.25) {\mathfrak{y}_{\W_1}(X_2) = };

        \node [anchor = center] at (-4.5,0.25) {2};
        \node [anchor = center] at (-4,0.25) {3};
        \node [anchor = center] at (-3.5,0.25) {\text{\tiny$\begin{matrix}4\\5\end{matrix}$}};
        \node [anchor = center] at (-3,0.25) {6};

        \draw (-4.75,0.6) -- (-2.75,0.6);
        \draw (-4.75,-0.1) -- (-2.75,-0.1);
        \draw (-4.75,0.6) -- (-4.75,-0.1);
        \draw (-4.25,0.6) -- (-4.25,-0.1);
        \draw (-3.75,0.6) -- (-3.75,-0.1);
        \draw (-3.25,0.6) -- (-3.25,-0.1);
        \draw (-2.75,0.6) -- (-2.75,-0.1);

        \node [anchor = center] at (-2.25,0.25) {=};
        
        \node [anchor = center] at (0,0) {6};
        \node [anchor = center] at (-0.5,0) {5};
        \node [anchor = center] at (-0.5,0.5) {4};
        \node [anchor = center] at (-1,0.5) {3};
        \node [anchor = center] at (-1.5,0.5) {2};

        \draw (-1.75,0.75)--(-0.25,0.75);
        \draw (-1.75,0.25)--(0.25,0.25);
        \draw (-0.75,-0.25)--(0.25,-0.25);
        \draw (-1.75,0.75)--(-1.75,0.25);
        \draw (-1.25,0.75)--(-1.25,0.25);
        \draw (-0.75,0.75)--(-0.75,-0.25);
        \draw (-0.25,0.75)--(-0.25,-0.25);
        \draw (0.25,0.25)--(0.25,-0.25);

        \node [anchor = center] at (1.5,0.25) {= \mathfrak{x}(X_2).};
    \end{tikzcd}
\end{equation*}
In particular, we have that $\beta_{\W_1}^{-1}(X_2) = X_2$. Thus $X_2 \in \trigW{\W_1}(\W_1)$, but $X_2 \notin \trig(\Pi(A_6))$. Similarly, we have $\beta_{\mods\Pi(A_6)}^{-1}(X_2) \notin \W_1$, even though $X_2 \in \W_1$.

To summarize the above discussion, we have proven the following.

\begin{proposition}\label{prop:tau_An}
    The diagram $\mathfrak{y}_\mathcal{W}(X)$ contains precisely the data of a $\W$-composition series of the module $\beta_\mathcal{W}^{-1}(X)$.
\end{proposition}

An immediate consequence is the following.

\begin{corollary}\label{cor:tau_An}
    Let $\omega = (X_k,\ldots,X_1) \in \btex(\Pi(A_n))$, and denote $\W_0 := \mods\Pi(A_n)$. For $1 \leq i \leq k$, do the following in order:
    \begin{enumerate}
        \item Let $Y_i \in \mods\Pi(A_n)$ be the module whose composition series is encoded by the diagram $\mathfrak{y}_{\W_{i-1}}(X_i)$.
        \item Let $\W_i = \J_{\W_{i-1}}(Y_i)$.
    \end{enumerate}
    Then $(Y_k,\ldots,Y_1) \in \tex(\Pi(A_n))$ is the $\tau$-exceptional sequence which is sent to $\omega$ by $\beta$.
\end{corollary}

We conclude this section with a detailed example.

\begin{example}\label{ex:tau_An}\
\begin{enumerate}
    \item Consider brick-$\tau$-exceptional sequence $(X_4,X_3,X_2,X_1)$ corresponding to the ppo-admissible sequence $(\gamma_4,\gamma_3,\gamma_2,\gamma_1)$ from Example~\ref{ex:ppo_admissible_sequence}(2). Now denote $\W_0 = \mods\Pi(A_4)$, and for $i \in \{1,2,3,4\}$, denote $\W_i = \J_{\W_{i-1}}(X_{i+4})$. We then have the following data. (Note that the corresponding permutation preorders were also computed in Example~\ref{ex:ppo_admissible_sequence}(2).)
    \end{enumerate}
\begin{center}
    \begin{tabular}{|c||c|c|c||c|c|}
        \hline
        i & $\text{simples}(\W_i)$ & $\mathfrak{x}(X_{i+4})$ & $\mathfrak{y}_{\W_i}(X_{i+4})$ & $\eta^{-1}(\W_i) \in \ppo(4)$ & $(\eta \circ \mu)^{-1}(\W_i) \in \mathfrak{S}_5$\\
        \hline
        0 & 1, 2, 3, 4 & \begin{tikzpicture}[baseline = -3pt]
        \node at (0,0.3) {};
        \node at (0,-0.3) {};
        \node [anchor = center] at (0,0) {4};
        \node [anchor = center] at (-0.5,0) {3};
        \node [anchor = center] at (-1,0) {2};
        \node [anchor = center] at (-1.5,0) {1};

        \draw (-1.75,0.25)--(0.25,0.25);
        \draw (-1.75,-0.25)--(0.25,-0.25);
        \draw (-1.75,0.25)--(-1.75,-0.25);
        \draw (-1.25,0.25)--(-1.25,-0.25);
        \draw (-0.75,0.25)--(-0.75,-0.25);
        \draw (-0.25,0.25)--(-0.25,-0.25);
        \draw (0.25,0.25)--(0.25,-0.25);
    \end{tikzpicture} & \begin{tikzpicture}[baseline = -3pt]
        \node at (0,0.3) {};
        \node at (0,-0.3) {};
        \node [anchor = center] at (0,0) {4};
        \node [anchor = center] at (-0.5,0) {3};
        \node [anchor = center] at (-1,0) {2};
        \node [anchor = center] at (-1.5,0) {1};

        \draw (-1.75,0.25)--(0.25,0.25);
        \draw (-1.75,-0.25)--(0.25,-0.25);
        \draw (-1.75,0.25)--(-1.75,-0.25);
        \draw (-1.25,0.25)--(-1.25,-0.25);
        \draw (-0.75,0.25)--(-0.75,-0.25);
        \draw (-0.25,0.25)--(-0.25,-0.25);
        \draw (0.25,0.25)--(0.25,-0.25);
    \end{tikzpicture} & $\{\{0,1,2,3,4\}\}$ & 43210 \\
        \hline
        1 & 1, 2, 3 & \begin{tikzpicture}[baseline = -3pt]
        \node at (0,0.3) {};
        \node at (0,-0.3) {};
        \node [anchor = center] at (0,0) {3};
        \node [anchor = center] at (-0.5,0) {2};
        \node [anchor = center] at (-1,0) {1};

        \draw (-1.25,0.25)--(0.25,0.25);
        \draw (-1.25,-0.25)--(0.25,-0.25);
        \draw (-1.25,0.25)--(-1.25,-0.25);
        \draw (-0.75,0.25)--(-0.75,-0.25);
        \draw (-0.25,0.25)--(-0.25,-0.25);
        \draw (0.25,0.25)--(0.25,-0.25);
    \end{tikzpicture} & \begin{tikzpicture}[baseline = -3pt]
        \node at (0,0.3) {};
        \node at (0,-0.3) {};
        \node [anchor = center] at (0,0) {3};
        \node [anchor = center] at (-0.5,0) {2};
        \node [anchor = center] at (-1,0) {1};

        \draw (-1.25,0.25)--(0.25,0.25);
        \draw (-1.25,-0.25)--(0.25,-0.25);
        \draw (-1.25,0.25)--(-1.25,-0.25);
        \draw (-0.75,0.25)--(-0.75,-0.25);
        \draw (-0.25,0.25)--(-0.25,-0.25);
        \draw (0.25,0.25)--(0.25,-0.25);
    \end{tikzpicture} & $\{\{0,1,2,3\},\{4\}\}$ & 32104\\
        \hline
        2 & 1, 2 & \begin{tikzpicture}[baseline = -3pt]
        \node at (0,0.3) {};
        \node at (0,-0.3) {};
        \node [anchor = center] at (0,0) {2};
        \node [anchor = center] at (-0.5,0) {1};

        \draw (-0.75,0.25)--(0.25,0.25);
        \draw (-0.75,-0.25)--(0.25,-0.25);
        \draw (-0.75,0.25)--(-0.75,-0.25);
        \draw (-0.25,0.25)--(-0.25,-0.25);
        \draw (0.25,0.25)--(0.25,-0.25);
    \end{tikzpicture} &\begin{tikzpicture}[baseline = -3pt]
        \node at (0,0.3) {};
        \node at (0,-0.3) {};
        \node [anchor = center] at (0,0) {2};
        \node [anchor = center] at (-0.5,0) {1};

        \draw (-0.75,0.25)--(0.25,0.25);
        \draw (-0.75,-0.25)--(0.25,-0.25);
        \draw (-0.75,0.25)--(-0.75,-0.25);
        \draw (-0.25,0.25)--(-0.25,-0.25);
        \draw (0.25,0.25)--(0.25,-0.25);
    \end{tikzpicture} & $\{\{0,1,2\},\{3\},\{4\}\}$ & 21034\\
        \hline
        3 & 1 & \begin{tikzpicture}[baseline = -3pt]
        \node at (0,0.3) {};
        \node at (0,-0.3) {};
        \node [anchor = center] at (0,0) {1};

        \draw (-0.25,0.25)--(0.25,0.25);
        \draw (-0.25,-0.25)--(0.25,-0.25);
        \draw (-0.25,0.25)--(-0.25,-0.25);
        \draw (0.25,0.25)--(0.25,-0.25);
    \end{tikzpicture} & \begin{tikzpicture}[baseline = -3pt]
        \node at (0,0.3) {};
        \node at (0,-0.3) {};
        \node [anchor = center] at (0,0) {1};

        \draw (-0.25,0.25)--(0.25,0.25);
        \draw (-0.25,-0.25)--(0.25,-0.25);
        \draw (-0.25,0.25)--(-0.25,-0.25);
        \draw (0.25,0.25)--(0.25,-0.25);
    \end{tikzpicture} &  $\{\{0,1\},\{2\},\{3\},\{4\}\}$ & 10234\\
        \hline
        4 & $\emptyset$ & \begin{tikzpicture}[baseline = -3pt]
        \node at (0,0.3) {};
        \node at (0,-0.3) {};
    \end{tikzpicture} &  & $\{\{0\},\{1\},\{2\},\{3\},\{4\}\}$ & 01234 \\
        \hline
    \end{tabular}
\end{center}
    \begin{itemize}
    \item[(2)] Consider brick-$\tau$-exceptional sequence $(X_3,X_2,X_1)$ corresponding to the ppo-admissible sequence $(\gamma_3,\gamma_2,\gamma_1)$ from Example~\ref{ex:ppo_admissible_sequence}(3). Now denote $\W_0 = \mods\Pi(A_7)$, and for $i \in \{1,2,3\}$, denote $\W_i = \J_{\W_{i-1}}(X_{i})$. We then have the following data, where by convention we list partition elements of $\eta^{-1}(\W_i)$ so that reading left-to-right gives a linear extension of the partial order $\preceq_{\eta^{-1}(\W_i)}$.(Note that some of the corresponding permutation preorders were also computed in Example~\ref{ex:ppo_admissible_sequence}(3).)
    \end{itemize}

\begin{center}
    \begin{tabular}{|c||c|c|c||c|c|}
        \hline
        i & $\text{simples}(\W_i)$ & $\mathfrak{x}(X_{i})$ & $\mathfrak{y}_{\W_i}(X_{i})$ & $\eta^{-1}(\W_i) \in \ppo(6)$ & $(\eta \circ \mu)^{-1}(\W_i) \in \mathfrak{S}_7$\\
        \hline
        0 & 1, 2, 3, 4, 5, 6 & \begin{tikzpicture}[baseline = -2.5pt]
        \node at (0,0.3) {};
        \node at (0,-0.3) {};
        \node [anchor = center] at (0,0) {4};
        \node [anchor = center] at (-0.5,0) {3};
        \node [anchor = center] at (-1,0) {2};

        \draw (-1.25,0.25)--(0.25,0.25);
        \draw (-1.25,-0.25)--(0.25,-0.25);
        \draw (-1.25,0.25)--(-1.25,-0.25);
        \draw (-0.75,0.25)--(-0.75,-0.25);
        \draw (-0.25,0.25)--(-0.25,-0.25);
        \draw (0.25,0.25)--(0.25,-0.25);
    \end{tikzpicture} & \begin{tikzpicture}[baseline = -2.5pt]
        \node at (0,0.3) {};
        \node at (0,-0.3) {};
        \node [anchor = center] at (0,0) {4};
        \node [anchor = center] at (-0.5,0) {3};
        \node [anchor = center] at (-1,0) {2};

        \draw (-1.25,0.25)--(0.25,0.25);
        \draw (-1.25,-0.25)--(0.25,-0.25);
        \draw (-1.25,0.25)--(-1.25,-0.25);
        \draw (-0.75,0.25)--(-0.75,-0.25);
        \draw (-0.25,0.25)--(-0.25,-0.25);
        \draw (0.25,0.25)--(0.25,-0.25);
    \end{tikzpicture} & $\{\{0,1,2,3,4,5,6\}\}$ & 6543210 \\
        \hline
        1 & $\text{\tiny$\begin{matrix}4\\3\\2\\1\end{matrix}$}, \ 2, 3, \text{\footnotesize$\begin{matrix}4\\5\end{matrix}$},\ 6$ & 
    \begin{tikzpicture}[baseline=5pt]
        \node at (0,0.8) {};
        \node at (0,-0.3) {};
        \node [anchor = center] at (0,0) {6};
        \node [anchor = center] at (-0.5,0) {5};
        \node [anchor = center] at (-0.5,0.5) {4};
        \node [anchor = center] at (-1,0.5) {3};
        \node [anchor = center] at (-1.5,0.5) {2};

        \draw (-1.75,0.75)--(-0.25,0.75);
        \draw (-1.75,0.25)--(0.25,0.25);
        \draw (-0.75,-0.25)--(0.25,-0.25);
        \draw (-1.75,0.75)--(-1.75,0.25);
        \draw (-1.25,0.75)--(-1.25,0.25);
        \draw (-0.75,0.75)--(-0.75,-0.25);
        \draw (-0.25,0.75)--(-0.25,-0.25);
        \draw (0.25,0.25)--(0.25,-0.25);
    \end{tikzpicture}
 & \begin{tikzpicture}[baseline=5pt]
        \node [anchor = center] at (0,0) {6};
        \node [anchor = center] at (-0.5,0) {5};
        \node [anchor = center] at (-0.5,0.5) {4};
        \node [anchor = center] at (-1,0.5) {3};
        \node [anchor = center] at (-1.5,0.5) {2};

        \draw (-1.75,0.75)--(-0.25,0.75);
        \draw (-1.75,0.25)--(0.25,0.25);
        \draw (-0.75,-0.25)--(0.25,-0.25);
        \draw (-1.75,0.75)--(-1.75,0.25);
        \draw (-1.25,0.75)--(-1.25,0.25);
        \draw (-0.75,0.75)--(-0.75,-0.25);
        \draw (-0.25,0.75)--(-0.25,-0.25);
        \draw (0.25,0.25)--(0.25,-0.25);
    \end{tikzpicture} & $\{\{0,4\},\{1,2,3,5,6\}\}$ & 4065321\\
        \hline
        2 & $\text{\tiny$\begin{matrix}4\\3\\2\\1\end{matrix}$}, \ 2, 3, \text{\footnotesize$\begin{matrix}4\\5\end{matrix}$}$ & 
    \begin{tikzpicture}[baseline=10pt]
        \node at (0,1.25) {};
        \node [anchor = center] at (0,0) {5};
        \node [anchor = center] at (0,0.5) {4};
        \node [anchor = center] at (0,1) {3};
        \node [anchor = center] at (-0.5,1) {2};

        \draw (0.25,-0.25)--(0.25,1.25);
        \draw (-0.25,-0.25)--(-0.25,1.25);
        \draw (-0.75,0.75)--(-0.75,1.25);
        \draw (0.25,-0.25)--(-0.25,-0.25);
        \draw (0.25,0.25)--(-0.25,0.25);
        \draw (0.25,0.75)--(-0.75,0.75);
        \draw (0.25,1.25)--(-0.75,1.25);
    \end{tikzpicture}
 &  \begin{tikzpicture}[baseline=10pt]
        \node at (0,1.25) {};
        \node at (0,-0.25) {};
        \node [anchor = center] at (0,0) {5};
        \node [anchor = center] at (0,0.5) {4};
        \node [anchor = center] at (0,1) {3};
        \node [anchor = center] at (-0.5,1) {2};
        \node [anchor = center] at (-0.5,0.5) {3};

        \draw (0.25,-0.25)--(0.25,1.25);
        \draw (-0.25,-0.25)--(-0.25,1.25);
        \draw (-0.75,0.25)--(-0.75,1.25);
        \draw (0.25,-0.25)--(-0.25,-0.25);
        \draw (0.25,0.25)--(-0.75,0.25);
        \draw (0.25,0.75)--(-0.75,0.75);
        \draw (0.25,1.25)--(-0.75,1.25);
    \end{tikzpicture} & $\{\{0,4\},\{1,2,3,5\},\{6\}\}$ & 4053216\\
        \hline
        3 & $\text{\tiny$\begin{matrix}4\\3\\2\\1\end{matrix}$}, \ 2, \text{\footnotesize$\begin{matrix}4\\5\end{matrix}$}$ &  & &  $\{\{0,4\},\{1,2\},\{3,5\},\{6\}\}$ & 4021536\\
        \hline
    \end{tabular}
\end{center}
\end{example}

\section{Application: Shellability of the shard intersection order}\label{sec:EL}

In this section, we use ppo-admissible sequences to give a representation theoretic proof that the order complex of the shard intersection order on the Coxeter group $A_n$ is shellable. This is one of the main results of \cite{bancroft}, where it is proven by constructing an EL-labeling (see \cite[Defintion~3.2.1]{wachs} or the discussion preceeding Lemma~\ref{lem:increasing_endpt_unique} below for the definition). We also prove shellability by describing an EL-labeling, although one that differs from that in \cite{bancroft}.

From the perspective of representation theory, our EL-labeling comes from the labeling of the lattice of wide subcategories of a $\tau$-tilting finite algebra by (brick-)$\tau$-exceptional sequences. This labeling was implicitly constructed in \cite{BM_wide} and studied explicitly in \cite[Section~5]{BaH2}. Indeed, in \cite[Theorem~A]{BaH2}, it was shown that the labeling coming from (brick-)$\tau$-exceptional sequences is an EL-labeling under the assumption that the transitive closure of the Hom-relation on bricks is acyclic. While this is not the case for the algebras $RA_n$ and $\Pi(A_n)$, the arc-model established in this paper allows us to modify the proofs of \cite[Section~5]{BaH2} to prove we still have an EL-labeling for these algebras. It remains an open question whether the labeling coming from (brick-)$\tau$-exceptional sequences is an EL-labeling in full generality.

Recall from Theorem~\ref{thm:edge-labeling} that there is a map $\arclab: \cov(\ppo(n)) \rightarrow \arc(n)$ which, by Corollary~\ref{cor:sat_chain} induces a bijection $\arclab: \stopc(\ppo(n)) \rightarrow \ppoadm(n)$. Similarly, by Theorem~\ref{thm:sat_top}, there is a map $\brlab: \cov(\wide(RA_n)) \rightarrow \brick(RA_n)$ which induces a bijection $\brlab: \stopc(\wide(RA_n)) \rightarrow \tex(RA_n)$. Then by Lemma~\ref{lem:tau_perp} and Theorem~\ref{thm:tau_exceptional_RAn}, we have that $\brlab \circ \eta = \sigma^{-1} \circ \arclab$.

Recall that by convention we have $\textnormal{u} < \textnormal{e} < \textnormal{o}$, and let $\leq_{\word}$ be the lexicographic ordering on $\word(\textnormal{uoe})$. Let $\leq_{\wrd}$ be the induced lexicographic order on $[n-1] \times \word(\textnormal{uoe})$. We can then make $\brick(RA_n)$ and $\arc(n)$ into posets by setting $X \leq_{\wrd} Y$ if and only if $\sigma(X) \leq_{\wrd} \sigma(Y)$. This makes both $\brlab$ and $\arclab$ into examples of \emph{edge labelings} of the posets $\wide(RA_n)$ and $\ppo(n)$, respectively. We identify these labelings for the remainder of the section.

Given a saturated chain $\mathfrak{c} = (\P_{k} \lneq_{\ppo} \cdots \lneq_{\ppo} \P_0)$ in $\ppo(n)$, we associate a tuple $$\arclab(\mathfrak{c}) := \left(\arclab(\P_{k} \lneq_{\ppo} \P_{k-1}),\arclab(\P_{k-2} \lneq_{\ppo} \P_{k-1})\ldots,\arclab(\P_1\lneq_{\ppo} \P_0)\right).$$
If $\P_0 = \{[n]\}$, we note that $\arclab(\mathfrak{c})$ is a ppo-admissible sequence by Remark~\ref{rem:ppo_reduction}. More generally, by Corollary~\ref{cor:ppo_seq_shuffle} we have that $\arclab(\mathfrak{c})$ is a shuffle of $\ppo$-admissible sequences of arcs, one for each partition element of $\P_0$.

Following the convention in \cite[Section~5]{BaH2}, we totally order the set of all tuples in reflected lexicographic order (reading left-to-right). Equivalently, this is the usual lexicographic order, reading tuples from right-to-left. We say $\mathfrak{c}$, or equivalently $\arclab(\mathfrak{c})$, is \emph{increasing} if $\arclab(\P_{i}\lneq_{\ppo} \P_{i-1}) \lneq_{\wrd} \arclab(\P_{i+1}\lneq_{\ppo} \P_{i})$ for all $i$.

The goal of the remainder of this section is to show that the labeling $\arclab$ is an \emph{edge-lexicographic (EL) labeling}. By definition, this means that for each interval $[\P,\Q]$ in $\ppo(n)$ there is a unique increasing maximal chain, and furthermore that this increasing chain is smallest in the reflected lexicographic order.

\begin{lemma}\label{lem:increasing_endpt_unique}
    Let $\mathfrak{c}$ be an increasing chain in $\ppo(n)$, and denote $\arclab(\mathfrak{c}) = (\gamma_k,\ldots,\gamma_1)$. Then $\ell(\gamma_1),\ldots, \ell(\gamma_k)$ are all distinct.
\end{lemma}

\begin{proof}
    We prove the result by induction on $k$. If $k = 1$, then there is nothing to show, so suppose that $k > 1$ and that the result holds for all $k' < k$.
    
    If there does not exist $j > 1$ such that $\ell(\gamma_1) = \ell(\gamma_j)$, the result follows from the induction hypothesis. Thus suppose otherwise. We will show that $\gamma_j \leq_{\wrd} \gamma_1$.
    
    Let $\P$ be the largest element (with respect to $\leq_{\ppo}$) in the chain $\mathfrak{c}$. Then $\gamma_j$ must be $\J_{\P}(\gamma_1)$-ppo-admissible. Since $\ell(\gamma_1) = \ell(\gamma_j)$, it follows that $r(\gamma_j) \in L_{\P}(\gamma_1)$. Denote $\gamma_1 = \left(\ell(\gamma_1),s_{\ell(\gamma_1)}^1\cdots s_{r(\gamma_1)}\right)$ and likewise for $\gamma_j$, and let $k$ be the smallest index for which $s_k^1 \neq s_k^j$. By the definition of ppo-admissibility, if $\ell(\gamma_1) + k \in U_{\P}(\gamma_1)$, then either $s_1^k = \text{u}$ or $s_1^k = \text{e}$. In the either case, it follows that $L_{\P}(\gamma_1) \preceq_{\{[n]\}} U_{\P}(\gamma_1)$. This means $s_j^k = \text{u}$, and so $s_1^k = \text{e}$ and $\gamma_j \leq_{\wrd} \gamma_1$. Similarly, we cannot have that $\ell(\gamma_1) + k \notin U_{\P}(\gamma_1) \cup L_{\P}(\gamma_1)$, since this would imply that $s_k^1 = s_k^j$ by the definition of $\J_\P(\gamma_1)$-ppo-admissibility. The remaining case to consider is $\ell(\gamma_1) + k \in L_{\P}(\gamma_1)$. If $\ell(\gamma_1) + k > r(\gamma_1)$, then there exists $k' < k$ such that $s_1^{k'} = \text{e} \neq s_j^{k'}$, contradicting the minimality of $k$. Thus $\ell(\gamma_1) + k < r(\gamma_1)$ and $s_1^k = \text{o} > s_j^k$. We conclude that $\gamma_j \leq_{\wrd} \gamma_1$, as claimed.
\end{proof}

\begin{lemma}\label{lem:quotients_subs_increasing} Let $\gamma_1,\gamma_2 \in \arc(n)$.
    \begin{enumerate}
        \item If $\gamma_2$ is a quotient arc of $\gamma_1$, then $\gamma_1 \leq_{\wrd} \gamma_2$.
        \item If $\ell(\gamma_1) = \ell(\gamma_2)$ and $\gamma_1$ is a submodule arc of $\gamma_2$, then $\gamma_1 \leq_{\wrd} \gamma_2$.
    \end{enumerate}
\end{lemma}

\begin{proof}
    (1) If $\ell(\gamma_1) \neq \ell(\gamma_2)$ then we are done. Otherwise, we have either that $\gamma_1 = \gamma_2$ or that $\gamma_1$ must pass under $r(\gamma_2)$. We conclude that $\gamma_1 \leq_{\wrd} \gamma_2$.

    (2) If $\gamma_1=\gamma_2$ then there is nothing to show. Otherwise we have that $\gamma_1$ must pass over $r(\gamma_2)$. We conclude that $\gamma_1 \leq_{\wrd} \gamma_2$.
\end{proof}

For the remainder of the theorem, our proofs will rely explicitly on concepts from representation theory. It remains an interested problem, however, to find a purely combinatorial proof that $\arclab$ is an EL-labeling. As such, we have stated each result in combinatorial terms whenever possible.

Recall that a \emph{torsion class} is a full subcategory $\mathcal{T} \subseteq \mods RA_n$ which is closed under taking extensions and quotients. Given an arbitrary subcategory $
\mathcal{C} \subseteq \mods RA_n$, we have that $\Filt(\Fac \mathcal{C})$ and $\lperp{\mathcal{C}}$ are both torsion classes, see e.g. \cite[Section~2]{thomas_intro}. An object $X \in \mathcal{T}$ is called \emph{split projective} in $\mathcal{T}$ every surjection $Y \twoheadrightarrow X$ with $Y \in \mathcal{T}$ is split. Since $RA_n$ is $\tau$-tilting finite, for every torsion class $\mathcal{T}$ contains a basic $\tau$-rigid module $M \in \mathcal{T}$ which is split projective in $\mathcal{T}$ and satisfies $\mathcal{T} = \mathrm{Gen}(M)$. By definition, the latter conditions says that every element of $\mathcal{T}$ can be written as a quotient of something in $\add(M)$. See \cite[Lemma~2.8]{IT} and \cite[Theorem~2.7]{AIR}.

\begin{proposition}\label{prop:existence_increasing}
    Let $\Q \leq_{\ppo} \P \in \ppo(n)$. Then the lexicographically smallest chain in $\ppo(n)$ which starts at $\P$ and ends at $\Q$ is increasing.
\end{proposition}

\begin{proof}
    Let $\mathfrak{c}$ be the lexicographically smallest chain which starts at $\P$ and ends at $\Q$, and denote $\arclab(\mathfrak{c}) = (\gamma_k,\ldots \gamma_1)$. We prove the result by induction on $k$.

    For $k = 1$, there is nothing to show, so suppose that $k > 1$ and that the result holds for $k-1$. Denote $\P_1 = \J_{\P}(\gamma_1)$. By construction, we have that $\gamma_2$ is minimal (with respect to $\leq_{\wrd}$) amongst those arcs in $\arc_{\P_1}(n)$ for which $\Q \leq_{\ppo} \J_{\P_1}$. Equivalently, by Theorem~\ref{thm:wide} and Lemma~\ref{lem:tau_perp}, $X:= \sigma^{-1}(\gamma_2)$ is minimal (with respect to $\leq_{\wrd}$) amonst those bricks in $\eta(\P_1)$ for which $\eta(\Q) \subseteq \J_{\eta(\P_1)}(X)$. Now by \cite[Corollary~1.5]{BM_wide}, there exists a brick $Y \in \eta(\P)$ and a quotient map $Y \twoheadrightarrow X$ such that $\sigma^{-1}(\gamma_1) \oplus Y \in \trigW{\eta(\P)}(\eta(\P))$ and $\J_{\eta(\P_1)}(X) = \J_\P(\sigma^{-1}(\gamma_1) \oplus Y)$. By Proposition~\ref{prop:quotients} and Lemma~\ref{lem:quotients_subs_increasing}, the minimality of $\gamma_1$ thus implies that $\gamma_1 \lneq_{\wrd} \sigma(Y) \leq_{\wrd} \gamma_2$. Therefore $\mathfrak{c}$ must be increasing by the induction hypothesis.
\end{proof}

\begin{theorem}\label{thm:EL_labeling}
    $\arclab$ is an EL-labeling.
\end{theorem}

\begin{proof}
    Let $\Q \leq_{\ppo} \P$. Let $\mathfrak{c}$ be the lexicographically smallest chain from $\P$ to $\Q$, and denote $\arclab(\mathfrak{c}) = (\gamma_k,\ldots,\gamma_1)$. We know that $\mathfrak{c}$ is increasing by Proposition~\ref{prop:existence_increasing}, so it remains only to show that there does not exist another increasing chain. We prove the result by induction on $k$. For $k = 1$, there is nothing to prove. Thus suppose that $k > 1$ and the result holds for $k-1$.
    
    Let $\mathfrak{c}'$ be any increasing chain which starts at $\P$ and ends at $\Q$, and denote $\arclab(\mathfrak{c}') = (\rho_k,\ldots \rho_1)$. Denote $\P_0 = \P$, and for $1 \leq i \leq k$ denote $\P_i = \J_{\P_{i-1}}(\rho_i)$ and $\mathcal{T}_i = \eta(\P_{i-1}) \cap \Filt\Fac(\sigma^{-1}(\rho_k),\ldots,\sigma^{-1}(\rho_i)) \in \tors(\eta(\P_{i-1}))$.

    We claim that $\sigma^{-1}(\rho_i)$ is split projective in $\mathcal{T}_i$ for all $i$. It suffices to show this for $i = 1$ then use induction. By \cite[Theorem~5.1]{MT}, there exists a unique sequence $(M_k,\ldots,M_1)$ of bricks such that (i) $M:= \bigoplus_{i = 1}^k M_i \in \trigW{\eta(\P)}(\eta(\P))$, (ii) for all $1 \leq i \leq k$ there is a surjection $h_i: M_i \twoheadrightarrow \sigma^{-1}(\rho_i)$ such that $\ker(h_i) \in \mathrm{Gen}(\oplus_{i' < i} M_{i'})$ and $\Hom(\oplus_{i' < i} M_i,\sigma^{-1}(\rho_i)) = 0$, and (iii) $\mathcal{T}_1 = \eta(\P) \cap \mathrm{Gen}(M)$. Now let $v$ be the smallest vertex on which the torsion class $\mathcal{T}$ is supported. Since $\mathfrak{c}$ is increasing, Lemma~\ref{lem:increasing_endpt_unique} then implies that $\sigma^{-1}(\rho_1) = M_1$ must be supported on $v$. Now if $M_1$ were not split projective in $\mathcal{T}$, then there would exist some $i > 1$ and a morphism $f: M_i \rightarrow M_1$ which is nonzero on the vertex $v$. Thus $\ell(\sigma(M_i)) = \ell(\rho_1)$ and there exists a brick $X \in \mathcal{T}$ with $\ell(\rho(X)) = \ell(\rho_1)$ admitting a factorization $M_i \twoheadrightarrow X \hookrightarrow M_1$. By Lemma~\ref{lem:quotients_subs_increasing}, it follows that $\sigma(M_i) \lneq_{\wrd} \rho_1$. But at the same time, we know that there is a surjection $g: M_i \twoheadrightarrow \sigma^{-1}(\rho_i)$. There are then two possibilities.

    If the surjection is not supported at the vertex $v$, then there exists some $j < i$ and a morphism $f': M_j \rightarrow M_i$ which is supported at $v$. But then we can replace $f: M_i \rightarrow M_1$ with $f' \circ f: M_j \rightarrow M_1$. The composition will be nonzero since $\dim M_j(v) = \dim M_i(v) = \dim M_1(v) = 1$ and both maps are supported at this vertex. We can then repeat the previous argument. Then since the modules $M_i$ and $M_j$ are bricks, we cannot have maps $M_i \rightarrow M_j$ and $M_j \rightarrow M_i$ which are both supported at vertex $v$. Thus, by iterating this argument, we can assume that the surjection $g$ is supported at the vertex $v$.
    
    If the surjection $g$ is supported at the vertex $v$, then Proposition~\ref{prop:quotients} implies that $\sigma(M_i)$ passes under the node $r(\rho_i)$ and passes on the same side as $\rho_i$ for every node between $\ell(\rho_i)$ and $r(\rho_i)$. Similarly, we have that $\sigma(\rho_i)$ passes under the node $r(\sigma(X))$ and either $r(\rho_1) = r(\rho(X))$ or $\rho_1$ passes over the node $r(\rho(X))$. But this means one of (i) $\rho_i$ passes under the node $r(\sigma(X))$ (if $r(\sigma(X)) < r(\rho_i)$), (ii) $X = \sigma^{-1}(\rho_i)$ (if $r(\sigma(X)) = r(\rho_i)$), or (iii) $\rho_i$ passes under $r(\sigma(X))$ (if $r(\sigma(X)) > r(\rho_i)$). In any case, it follows that $\rho_i \leq_{\wrd} \rho_1$, a contradiction.

    We have shown that $M_1 = \sigma^{-1}(\rho_1)$ is split projective in $\mathcal{T}_1$, and thus that $\sigma^{-1}(\rho_i)$ is split projective in $\mathcal{T}_i$ for all $i$ by induction. It then follows from \cite[Theorem~5.9]{BM_wide} that $M_i$ is split projective in $\mathcal{T}_1$ for all $i$. In particular, $(M_k,\ldots,M_1)$ is an ordering of the indecomposable split projective modules in $\T_1$.

    In particular, the paragraphs above imply that there is an ordering $(M_{i_k},\ldots,M_{i_1})$ of the indecomposable split projective modules in $\T_1$ such that, for all $1 \leq j \leq k$, there is a quotient map $h'_j: M_{i_j} \twoheadrightarrow \sigma^{-1}(\rho_j)$ with $\ker(h'_j) \in \mathrm{Gen}(\oplus_{j' < j} M_{i_{j'}})$ and $\Hom(\oplus_{j' < j} M_{i_{j'}},\sigma^{-1}(\rho_j)) = 0$. By our original induction hypothesis, it remains only to show that $i_1 = 1$; i.e., that $\gamma_1 = \rho_1$.
    
    Suppose for a contradiction that $M_{i_1} = M_j$ with $j > 1$. In particular, this means $\gamma_1 = \sigma(M_j)$ is minimal amongst all $\P$-$\ppo$-admissible arcs which satisfy $\Q \leq_{\ppo} \J_\P(\gamma_1)$, and so $\sigma(M_j) \leq_{\wrd} \sigma(M_1)$. If $\ell(\sigma(M_j)) < \ell(\sigma(M_1))$, then there must exist some $j' > 1$ such that $\ell(\rho_{j'}) \leq \ell(\sigma(M_j)) < \ell(\sigma(M_1))$, which contradicts the assumption that $\mathfrak{c}'$ is increasing. Thus we have $\ell(\sigma(M_j)) = \ell(\sigma(M_1)) = v-1$. Moreover, there is a quotient map $M_{i_1} \twoheadrightarrow \sigma^{-1}(\rho_j)$ whose kernel lies in $\mathrm{Gen}(\oplus_{j' = 1}^{j-1} M_{j'})$. By the minimality of $\sigma(M_j) = \sigma(M_{i_1})$, Lemma~\ref{lem:quotients_subs_increasing} then implies that this kernel cannot be supported at the vertex $v$, and so we must have $\ell(\rho_j) = \ell(\sigma(M_1)) = \ell(\rho_1)$. This contradicts Lemma~\ref{lem:increasing_endpt_unique}.
\end{proof}

\section*{Index of Notation}

{\small
\begin{center}        \renewcommand{\arraystretch}{1.25}
    \begin{longtable}{L{0.28\textwidth}C{0.3\textwidth}R{0.28\textwidth}}
        \hline
        Notation & Explanation & Reference\\
        \hline
        \hline
        $[n]$ & \{0,1,\ldots,n\} & Section~\ref{sec:perm_preorder}\\
        \hline
       $\gamma = \left(\ell(\gamma),s_{\ell(\gamma)+1}\cdots s_{r(\gamma)}\right)$ & arc & Def.~\ref{def:wd}\\
       \hline
        $\textnormal{o, e, u}$ & ``over'', ``under'', ``end'' & Def.~\ref{def:wd}\\
        \hline
        $\arc_{cw}(n), \arc_{cw}(n,k)$ & clockwise-ordered arc diagrams & Def.~\ref{def:clockwise}\\
        \hline
       $\arc_{nc}(n)$ & noncrossing arc diagrams & Def.~\ref{def:noncrossing}\\
       \hline
       $\cov(\X)$ & cover relations & Def.~\ref{def:cover}\\
       \hline
       $\overline{S}$ & closed integer support & Def.~\ref{def:perm_preorder}\\
       \hline
        $\ppo(n), \leq_{\ppo}$ & permutation preorders & Def.~\ref{def:perm_preorder}, Def.~\ref{def:perm_preorder2}\\
        \hline
         $\delta, \mu, \mu'$ & bijections between $\mathfrak{S}_{n+1}, \arc_{nc}(n)$, $\ppo(n)$ & Prop.~\ref{prop:bijection_ppo}\\
        \hline
        $\Sigma(\gamma), \Psi(\P)$ & shard and shard intersection & Def.~\ref{def:shards}\\
        \hline
        $\mathfrak{S}_{n+1}, \leq_{\mathrm{sio}}$ & shard intersection order on permutations & Def.-Thm.~\ref{defthm:sio}\\
        \hline
        $v_i^S$ & $i$-th node in partition element $S$ & Sec.~\ref{sec:recursive}\\
        \hline
        $\chi_{\P}, \widetilde{T}$ & gluing of ppo & Prop.~\ref{prop:perm_preorder_recursive}\\
        \hline
        $\xi_{\P}$, $\widehat{T}$ & splitting of ppo & Lem.~\ref{lem:preorder2}\\
        \hline
        $\ppo_{\P}(n)$ & $\P$-ppo-admissible arcs & Def.~\ref{def:ppo_admissible}\\
        \hline
        $\{[n]\}$ & permutation preorder with one element & Ex.~\ref{ex:ppo_adm}\\
        \hline
        $\phi_\P$, $\psi_\P$ & arcs under splitting/gluing of permutation preorders& Prop~\ref{prop:ppo_admissible_recursive}\\
        \hline
        $\J_\P(\gamma)$, $U_\P(\gamma)$, $L_\P(\gamma)$ & ppo-reduction & Def.~\ref{def:ppo_reduction}\\
        \hline
        $\arclab$ & arc label & after Thm.~\ref{thm:edge-labeling}\\
       \hline
        $\satc(-), \stopc(-)$ & saturated (top) chain & Def.~\ref{sat_top}\\
       \hline
	$\shuff(-,-), \shuff(-; -,-)$ & shuffles & Def.~\ref{def:shuffle}\\
	\hline
	$\Phi_k^{\gamma}$ & unshuffling of ppo-adm. seq. & Cor.~\ref{cor:ppo_seq_shuffle}\\
        \hline
	$\Lambda, K$ & finite-dimensional algebra over field $K$ & Sec.~\ref{sec:background}\\
	\hline
	$\mods(-)$ & finitely-generated right modules & Sec.~\ref{sec:background}\\
	\hline
	$\rk(-)$ & \# summands or simples & Sec.~\ref{sec:background}, Def.~\ref{def:rank}\\
        \hline
	$\add(-), \Fac(-)$ & additive closure, quotients & Sec.~\ref{sec:background}\\
	\hline
        $\rperp{(-)}$, $\lperp{(-)}$ & Hom-perpendicular categories & Sec.~\ref{sec:background}\\
	\hline
	$\tau$, $\tau_\W$ & AR translation (in $\W$) & Sec.~\ref{sec:tau_exceptional}, Def.~\ref{def:tau_in_wide}\\
	\hline
	$\trig(\Lambda), \trigW{\W}(\Lambda)$ & $\tau$-rigid modules (in $\W$) & Sec.~\ref{sec:tau_exceptional}\\
	\hline
	$\itrig(\Lambda), \itrigW{\W}(\Lambda)$ & indec. $\tau$-rigid modules (in $\W$) & Sec.~\ref{sec:tau_exceptional}\\
	\hline
	$\wide(\Lambda)$, $\tors(\Lambda)$ & wide subcategories, torsion classes & Sec.~\ref{sec:tau_exceptional}\\
	\hline
	$\J(M), \J_\W(M)$ & $\tau$-perpendicular category & Def.~\ref{def:tau_perp}, Def.~\ref{def:tau_in_wide}\\
        \hline
        $\tex(\Lambda)$, $\tex(\Lambda,k)$ & $\tau$-exceptional sequences & Def.~\ref{def:tau_exceptional}\\
        \hline
        $\brick(\Lambda), \sbrick(\Lambda)$ & bricks and semibricks & Sec.~\ref{sec:bricks}, Def.~\ref{def:semibricks}\\
        \hline
        $\beta, \beta^{-1}_\W$ & brick-$\tau$-rigid correspondence & Thm.~\ref{thm:DIJ}, Rem.~\ref{rem:DIJ}\\
        \hline
        $\btex(\Lambda)$ & brick-$\tau$-exceptional sequences & Def.~\ref{def:tau_exceptional_brick}\\
        \hline
        $\Pi(A_n), RA_n$ & preprojective algebra and quotient by 2-cycles& Def.~\ref{def:preproj}\\
        \hline
        $\sigma$ & bijection $\{\text{bricks}\} \rightarrow \{\text{arcs}\}$ & Prop.~\ref{prop:arcBricks}\\
        \hline
        $\eta$ & bijection $\{\text{perm. preorders}\} \rightarrow \{\text{wide subcats.}\}$ & Thm.~\ref{thm:wide}\\
        \hline
         $\mathfrak{x}(X), \mathfrak{y}(X)$, $\mathfrak{y}_\mathcal{W}(X)$ & diagrams associated to brick & Sec.~\ref{sec:tau_exceptional_preproj}\\
         \hline
         $\brlab$ & brick-label & Sec.~\ref{sec:EL}\\
        \hline
        $\leq_{\wrd}$ & word-order & Sec.~\ref{sec:EL}\\
        \hline
    \end{longtable}
\end{center}\renewcommand{\arraystretch}{1.25}}


\bibliographystyle{amsalpha}
\bibliography{biblio.bib}

\end{document}